\documentclass[letterpaper,12pt,reqno]{amsart}

\usepackage[margin=1in]{geometry}
\usepackage{hyperref}
\hypersetup{
     colorlinks   = true,
     citecolor    = black,
     linkcolor    = blue
}

\usepackage{graphicx,amsmath,amssymb,amsthm,paralist,color,tikz-cd}
\usepackage{mathrsfs}
\usepackage{mathtools}
\usepackage{cite}
\usepackage{setspace}
\usepackage[color=blue!30]{todonotes}
\reversemarginpar

\usepackage{bbm} %blackboard bold

\newtheorem{theorem}{Theorem}[section]
\newtheorem{lemma}[theorem]{Lemma}

\newtheorem{question}[theorem]{Question}
\newtheorem{proposition}[theorem]{Proposition}

\newtheorem{corollary}[theorem]{Corollary}

\theoremstyle{definition}
\newtheorem{definition}[theorem]{Definition}
\newtheorem*{definition-nono}{Definition}

\newtheorem{remark}[theorem]{Remark}

\newtheorem*{acknowledgement}{Acknowledgements}

\newcommand{\N}{\mathbb{N}}
\newcommand{\Z}{\mathbb{Z}}
\newcommand{\Q}{\mathbb{Q}}
\newcommand{\R}{\mathbb{R}}
\newcommand{\C}{\mathbb{C}}

\newcommand{\mc}{\mathcal}

\renewcommand{\a}{\alpha}
\renewcommand{\b}{\beta}
\newcommand{\g}{\gamma}

\renewcommand{\d}{\delta}
\newcommand{\e}{\varepsilon}
\renewcommand{\l}{\lambda}

\newcommand{\w}{\omega}
\newcommand{\s}{\sigma}
\newcommand{\vp}{\varphi}

\renewcommand{\th}{\theta}

\newcommand{\set}[1]{\left\{#1\right\}}
\renewcommand{\r}{\rightarrow}

\def\multiset#1#2{\ensuremath{\left(\kern-.3em\left(\genfrac{}{}{0pt}{}{#1}{#2}\right)\kern-.3em\right)}}
\newcommand{\norm}[1]{\left\lVert#1\right\rVert}

\newcommand{\mbf}{\mathbf}
\newcommand{\mrm}{\mathrm}


\numberwithin{equation}{section}

\title[Exceptional Trajectories]{Exceptional directions for the Teichm\"{u}ller geodesic flow and Hausdorff dimension}

\author[Al-Saqban]{Hamid Al-Saqban}
\address{Department of Mathematics, University of Maryland, College Park, MD}
\email{hqs@math.umd.edu}

\author[Apisa]{Paul Apisa}
\address{Department of Mathematics, University of Chicago, Chicago, IL}
\email{papisa@math.uchicago.edu}

\author[Erchenko]{Alena Erchenko}
\address{Department of Mathematics, Pennsylvania State University, University Park, PA}
\email{axe930@psu.edu}

\author[Khalil]{Osama Khalil}
\address{Department of Mathematics, Ohio State University, Columbus, OH}
\email{khalil.37@osu.edu}

\author[Mirzadeh]{Shahriar Mirzadeh}
\address{Department of Mathematics, Brandeis University, Waltham, MA}
\email{shahmir@brandeis.edu}

\author[Uyanik]{Caglar Uyanik}
\address{Department of Mathematics, Vanderbilt University, Nashville, TN}
\email{caglar.uyanik@vanderbilt.edu}

\date{}

\begin{document}

\begin{abstract}
We prove that for every flat surface $\omega$, the Hausdorff dimension of the set of directions in which Teichm\"{u}ller geodesics starting from $\omega$ exhibit a definite amount of deviation from the correct limit in Birkhoff's and Oseledets' Theorems 
is strictly less than $1$. 
This theorem extends a result by Chaika and Eskin where they proved that such sets have measure $0$.
We also prove that the Hausdorff dimension of the directions in which Teichm\"{u}ller geodesics diverge on average in a stratum is bounded above by $1/2$, strengthening a classical result due to Masur. Moreover, we show that the Hausdorff codimension of the set of non-weakly mixing IETs with permutation $(d, d-1, \dots, 1)$, where $d$ is an odd number, is exactly $1/2$ and strengthen a result by Avila and Leguil. 
\end{abstract}

\maketitle
%\singlespacing
{\hypersetup{linkcolor=black}

\tableofcontents
}

\section{Introduction}

The problem of determining the size of the set of points with non-dense orbits under a partially hyperbolic transformation has a long history. These include orbits which escape to infinity, remain confined inside a proper compact set or simply miss a given open set. In the most studied setting, the transformation preserves a natural ergodic measure and hence these non-dense orbits have measure zero. Thus, it is natural to ask whether different types of non-dense orbits are more abundant than others with respect to other notions of size among which \emph{Hausdorff dimension} is the most common.

Many instances of this problem have been studied for algebraic partially hyperbolic flows on homogeneous spaces. For such flows, Margulis conjectured in his 1990 ICM address that orbits with closure a compact subset of a (non-compact) homogeneous space that misses a countable set of points have full Hausdorff dimension \cite[Conjectures A,B]{Margulis-Non-dense}. A full resolution of these conjectures was provided in subsequent papers of Kleinbock and Margulis \cite{Kleinbock-MargulisCojA} and Kleinbock and Weiss \cite{Kleinbock-MargulisCojB}. 
This phenomenon of abundance of non-dense orbits also takes place in the setting of hyperbolic dynamical systems. For example, Urba\'nski showed in \cite{Urbanski-NonDense} that non-dense orbits of Anosov flows on compact manifolds have full Hausdorff dimension. Then, in \cite{Dolgopyat-BoundedAnosov}, Dolgopyat studied the Hausdorff dimension of orbits of Anosov flows and diffeomorphisms which do not accumulate on certain low entropy subsets. It was shown that these trajectories have full Hausdorff dimension in many cases.

On the other hand, non-dense orbits of divergence type tend to be less abundant. In the homogeneous setting, it was shown in \cite{KKLM-SingSystems} that the divergent on average trajectories for certain flows on $\mrm{SL}(n,\R)/\mrm{SL}(n,\Z)$ do not have full Hausdorff dimension. In fact, an explicit upper bound on the Hausdorff dimension is given, generalizing earlier papers by Cheung \cite{Cheung-singular} and Cheung and Chevallier \cite{CheungChevallier}. In the setting of strata of quadratic differentials, Masur showed that the Hausdorff dimension of the set of non-uniquely ergodic directions for the associated translation flow is bounded above by 1/2. These correspond to divergent orbits for the Teichm\"uller geodesic flow.

In this article, we quantify the abundance of non-dense orbits in the setting of Teichm\"uller dynamics. Theorem~\ref{thrm: divergent on average} is the analogue of the result of \cite{KKLM-SingSystems} on the dimension of directions in which orbits of the Teichm\"uller geodesic flow are divergent on average. It provides a strengthening of Masur's result mentioned above. As for non-dense orbits, we study the more general problem concerning the set of directions at a fixed basepoint in which trajectories exhibit a definite amount of deviation from the correct limit in Birkhoff's and Oseledets' Theorems. Theorems~\ref{thrm: Birkhoff deviations thrm} and \ref{thrm: Oseledets deviations} show that the Hausdorff dimension of these sets of directions is bounded away from $1$ uniformly as the basepoint varies in the complement of certain proper submanifolds of the stratum. In particular, this implies that the intersection of the set of non-dense orbits with any Teichm\"uller disk in the complement of these finitely many proper submanifolds has positive Hausdorff codimension (see Corollary~\ref{cor: Birkhoff corollary} and Section~\ref{section: remarks}).

These results generalize prior work of Chaika and Eskin \cite{ChaikaEskin} in which the aforementioned exceptional sets were shown to have measure $0$.
The work of Chaika and Eskin was used in \cite{Hubert-windtree} to study the diffusion rate of billiard orbits in periodic wind-tree models. It was shown that for any choice of side lengths of the periodic rectangular obstacles, diffusion of orbits has a constant polynomial rate in almost every direction. Theorems~\ref{thrm: Birkhoff deviations thrm} and \ref{thrm: Oseledets deviations} imply that the directions exhibiting different diffusion rates do not have full Hausdorff dimension.
Prior to the work of Chaika and Eskin, Athreya and Forni ~\cite{AthreyaForni} established a polynomial bound on the deviation of Birkhoff averages of sufficiently regular functions along orbits of translation flows on flat surfaces in almost every direction. This full measure set of directions was chosen so that the average of a certain continuous function along the Teichm\"uller flow orbits is close to its expected value. Theorem ~\ref{thrm: Birkhoff deviations thrm} can be used to show that the directions which do not satisfy this bound are of dimension strictly smaller than $1$. 

It is well known that Teichm\"uller dynamics is closely tied to interval exchange transformations. In particular, Theorem~\ref{thrm: divergent on average} allows us to derive a lower bound on the Hausdorff codimension of the set of non-weakly mixing IETs with permutation $(d, d-1, \dots, 1)$, where $d$ is an odd number. In combination with the result of \cite{ChaikaMasur-dIET} establishing the upper bound, this allows us to compute the precise Hausdorff codimension.

\subsection*{Formulation of Results}

Let $g\geqslant 1$ and let $\alpha=(\alpha_1, \dots, \alpha_n)$ be an integral partition of $2g-2$. An \textbf{abelian differential} is a pair 
$(M,\omega)$, where $M$ is a Riemann surface of genus $g$ and $\omega$ is a holomorphic $1$-form on $M$ whose zeroes have multiplicities $\alpha_1, \dots, \alpha_n$. Throughout this paper, $\mathcal H_1(\alpha)$ will denote a stratum of Abelian differentials with area $1$
%i.e.,
%the set of pairs $(M,\omega)$, where $M$ is a Riemann surface of genus $g$ and $\omega$ is a holomorphic $1$-form on $M$ whose zeroes have multiplicities $\alpha_1, \dots, \alpha_n$ and 
with respect to the induced area form on $M$. We refer to points of $\mathcal H_1(\alpha)$ as translation surfaces. For the sake of brevity, we will often refer to $\omega$ itself as an element of $\mathcal H_1(\alpha)$.

We recall that there are well-defined local coordinates on a stratum, called \textbf{period coordinates} (e.g., see \cite[Section 2.3]{ForniMatheus} for details), such that all changes of coordinates are given by affine maps. In period coordinates, $\mathrm{SL}_2(\mathbb R)$ acts naturally on each copy of $\mathbb C$. Moreover, the closure of any $\mathrm{SL}_2(\mathbb R)$ orbit is an affine invariant manifold \cite{EMM}, i.e., a closed subset of $\mathcal H_1(\alpha)$ that is invariant under the $\mathrm{SL}_2(\mathbb R)$ action and looks like an affine subspace in period coordinates. Therefore, it is the support of an ergodic $\mathrm{SL}_2(\mathbb R)$ invariant probability measure.

The action of the following one parameter subgroups of $\mathrm{SL}_2(\R)$ will be referred to throughout the article.
	\begin{align*}
	g_t = \begin{pmatrix}
		e^t & 0 \\ 0 & e^{-t}
	\end{pmatrix}, \qquad
    r_\th = \begin{pmatrix}
    \cos\th & \sin\th \\ -\sin\th &\cos\th
    \end{pmatrix}, \qquad
    h_s = \begin{pmatrix}
    1 & s\\ 0 & 1
    \end{pmatrix}, \qquad
    \check{h}_s = \begin{pmatrix}
    1 & 0 \\ s & 1
    \end{pmatrix}
	\end{align*}

We recall that the actions of $g_t$, $r_{\theta}$, $h_s$ and $\check{h}_s$ correspond to the Teichm\"{u}ller geodesic flow, the rotation of the flat surface by the angle $\theta$, and the expanding and contracting horocycle flows, respectively.   

	\subsection*{Birkhoff's Ergodic Theorem} Chaika and Eskin \cite{ChaikaEskin} proved that for any translation surface $(M, \omega)\in \mathcal H_1(\alpha)$, and any continuous compactly supported function $f$ on $\mathcal H_1(\alpha)$, for almost all $\theta\in [0,2\pi]$,
\begin{equation}\label{Birkhoff}
\lim\limits_{T\rightarrow\infty}\frac{1}{T}\int_0^T f(g_tr_{\theta}w)dt = \int_{\mathcal M}fd\nu_{\mathcal M},
\end{equation}
where $\mathcal M = \overline{\mathrm{SL}_2(\mathbb R)\omega}$ is the smallest affine invariant manifold containing $\omega$ and $\nu_{\mathcal M}$ is the affine measure whose support is $\mathcal M$.

In this paper, we show that the Hausdorff dimension of the set of directions exhibiting a definite amount of deviation from the correct limit in \eqref{Birkhoff} is strictly less than $1$.

  \begin{theorem} \label{thrm: Birkhoff deviations thrm}
     Suppose $\mc{M}\subseteq \mc{H}_1(\a)$ is an affine invariant submanifold and $\nu_{\mc{M}}$ is the affine measure whose support is $\mc{M}$.
      Then, for any bounded continuous function $f$ on $\mc{M}$ and any $\e>0$, there exist affine invariant submanifolds $\mc{N}_1,\dots, \mc{N}_k$, properly contained in $\mc{M}$, and $0<\d< 1$, such that for all $\omega \in \mc{M}\backslash \left( \cup_{i=1}^k \mc{N}_i \right)$,      
      the Hausdorff dimension of the set
      \begin{equation*}
          \set{\th\in[0,2\pi]: \limsup_{T\r\infty} \left| \frac{1}{T} \int_0^T f(g_tr_\th\omega)\;dt -
                   \int_{\mathcal M} f \;d\nu_\mc{M}\right| \geqslant \e }
      \end{equation*}
      is at most $\d$.
  \end{theorem}
  
  \begin{remark}
  We note that the upper bound in Theorem~\ref{thrm: Birkhoff deviations thrm} is uniform as the basepoint $\omega$ varies in the complement of finitely many proper affine invariant submanifolds in $\mc{M}$.
This, in particular, includes points $\omega$ whose $\mrm{SL}(2,\R)$ orbit is not dense in $\mc{M}$.
  \end{remark}
  
   In Theorem~\ref{thrm: discrete Birkhoff deviations thrm for horocycles},  we obtain a version of Theorem~\ref{thrm: Birkhoff deviations thrm} for discrete Birkhoff averages which is needed for later applications.  
  It is worth noting that the exceptional sets considered in Theorem~\ref{thrm: Birkhoff deviations thrm} are non-empty in most examples and can, in fact, have positive Hausdorff dimension.
  By using the results in~\cite{KleinbockWeiss}, one can find a compact set $K$ such that the Hausdorff dimension of trajectories which are contained completely in $K$ is at least $1- \delta'$ for some $0<\delta' < 1$.
  By taking $f$ to be supported in the complement of $K$ and to have $\nu_\mc{M}(f) \neq 0$, these bounded trajectories will belong to the exceptional set for all $\e$ sufficiently small. 
  A similar argument shows that directions in which geodesics diverge on average (Definition~\ref{defn: doa}) belong to the exceptional sets of compactly supported function with non-zero average.
    
    Using the uniform dimension estimate in Theorem~\ref{thrm: Birkhoff deviations thrm}, we obtain the following corollary.
    \begin{corollary} \label{cor: Birkhoff corollary}
    Suppose $\mc{M}\subseteq \mc{H}_1(\a)$ is an affine invariant submanifold and $\nu_{\mc{M}}$ is the affine measure whose support is $\mc{M}$.
      Then, for any bounded continuous function $f$ on $\mc{M}$ and any $\e>0$, there exist affine invariant submanifolds $\mc{N}_1,\dots, \mc{N}_k$, properly contained in $\mc{M}$, and $0<\d <1$, such that for all $\omega \in \mc{M}\backslash \left( \cup_{i=1}^k \mc{N}_i \right)$, 
      the Hausdorff dimension of the set
      \begin{equation*}
          \set{x\in \mathrm{SL}_2(\R)\cdot \omega: \limsup_{T\r\infty} \left|
          \frac{1}{T} \int_0^T f(g_tx)\;dt -
          \int_{\mathcal M} f \;d\nu_\mc{M} \right| \geqslant \e }
      \end{equation*}
      is at most $2+\d$.
    \end{corollary}

    In particular, by a standard approximation argument, we see that for any non-empty open subset $U$ of a connected component $\mc{C}$ of the stratum $\mc{H}_1(\a)$, the Hausdorff dimension of the set
    \begin{equation*}
          \set{x\in \mathrm{SL}_2(\R)\cdot \omega: g_t x \notin U \text{ for all } t>0 }
    \end{equation*}
    is strictly less than the dimension of $\mathrm{SL}_2(\R)$, which is $3$.
    This being true uniformly over all Teichm\"uller curves $\mathrm{SL}_2(\R)\cdot \omega$ in the complement of finitely many lower dimensional invariant submanifolds of $\mc{C}$.

    \subsection*{Oseledets' Theorem for the Kontsevich-Zorich Cocycle}
	The next object of our study is the Lyapunov exponents of the Kontsevich-Zorich cocycle. 
      Consider the Hodge bundle whose fiber over every point $(X,\omega)\in \mc{H}_1(\a)$ is the cohomology group $H^{1}(X,\mathbb{R})$. 
      Let $\mathrm{Mod}(X)$ be the mapping class group, i.e. the group of isotopy classes of orientation preserving homeomorphisms of $X$. Fix a fundamental domain in the Teichm\"uller space for the action of $\mathrm{Mod}(X)$. Consider the cocycle $\tilde{A}:\mathrm{SL}_2(\R)\times\mc{H}_1(\alpha)\to Mod(X)$, where for $x$ in the fundamental domain, $\tilde{A}(g,x)$ is the element of $\mathrm{Mod}(X)$ that is needed to return the point $gx$ to fundamental domain. Then, the \textbf{Kontsevich-Zorich cocycle} $A(g,x)$ is defined by 
        \[ A(g,x)=\rho(\tilde{A}(g,x))
        \]
      where $\rho:\mathrm{Mod}(X)\to Sp(2g,\mathbb{Z})$ is given by the induced action of $\mathrm{Mod}(X)$ on cohomology.
      We recall the notion of a strongly irreducible $\mathrm{SL}_2(\R)$ cocycle.
      \begin{definition-nono}[Strongly Irreducible Cocycle]
      Let $(X,\nu)$ be a probability space admitting an action of a locally compact group $G$ which leaves $\nu$ invariant.
      Let $\pi: V\r X$ be a vector bundle over $X$ on which $G$ acts fiberwise linearly.
      We say that $V$ admits a $\nu$-\textbf{measurable almost invariant splitting} if there exists $n>1$ and 
      for $\nu$-almost every $x$, the fiber $\pi^{-1}(x)$ splits
      into non-trivial subspaces $V_1(x), \dots, V_n(x)$ satisfying $V_i(x) \cap V_j(x) = \set{0}$ for all $i\neq j$
      and $gV_i(x) = V_i(gx)$ for all $i$, $\nu$-almost every $x\in X$ and for almost every $g\in G$ with respect to the (left) Haar measure on $G$. And, finally, the map $x \mapsto V_i(x)$ is required to be $\nu$-measurable for all $i$.
      
      The $G$ action on $V$ is said to be \textbf{strongly irreducible} with respect to $\nu$ if the $G$-action doesn't admit any $\nu$-measurable almost invariant splitting.
      \end{definition-nono}

      In this setting, we prove the following statement about deviations in the Lyapunov exponents of the Kontsevich-Zorich cocycle.
      \begin{theorem} \label{thrm: Oseledets deviations}
      Suppose $\mc{M}\subseteq \mc{H}_1(\a)$ is an affine invariant submanifold and $\nu_{\mc{M}}$ is the affine measure whose support is $\mc{M}$.
      Let $V$ be a continuous (on $\mc{M}$) $\mathrm{SL}_2(\R)$ invariant sub-bundle of (some exterior power of) the Hodge bundle. Assume that $A_V$ is strongly irreducible with respect to $\nu_{\mc{M}}$, where $A_V$ is the restriction of the Kontsevich-Zorich cocycle to $V$.
      Then, for any $\e>0$, there exist affine invariant submanifolds $\mc{N}_1,\dots, \mc{N}_k$, properly contained in $\mc{M}$, and $0<\d< 1$, such that for all $\omega \in \mc{M}\backslash \left( \cup_{i=1}^k \mc{N}_i \right)$, the Hausdorff dimension of the set
      \begin{equation*}
          \set{\th\in[0,2\pi]: \limsup_{t\r\infty} \frac{\log \norm{A_V(g_t,r_\th\omega)} }{t}
                  \geqslant \l_V + \e }
      \end{equation*}
      is at most $\d$, where $\l_V$ denotes the top Lyapunov exponent for $A_V$ with respect to $\nu_\mc{M}$.
    \end{theorem}

    This complements a result in~\cite{ChaikaEskin} where they show that under the same hypotheses, for \emph{every} $\omega$ and for Lebesgue almost every 
    $\th \in [0,2\pi]$, the following limit exists
    \begin{equation*}
    	\lim_{t\r\infty} \frac{\log \norm{A_V(g_t,r_\th\omega)} }{t} =   \l_V
    \end{equation*}
    
    It is shown in~\cite[Theorem A.6]{EskinMirzakhani} that the Kontsevich-Zorich cocycle is in fact semisimple, which means that, after passing to a finite cover, the Hodge bundle splits into $\nu_\mc{M}$-measurable $\mathrm{SL}_2(\R)$ invariant, strongly irreducible subbundles.
    Moreover, it is shown in~\cite{Filip2016} that such subbundles can be taken to be continuous (and in fact real analytic) in period coordinates.
    Additionally, it is well-known that the top Lyapunov exponent of the $k^{th}$ exterior power of the cocycle is a sum of the top $k$ exponents of the cocycle itself.
    In this manner, we can deduce the deviation statement for \emph{all} Lyapunov exponents by examining the top Lyapunov exponents of exterior powers of the cocycle.
    The following Corollary is the precise statement.
    For more details on this deduction, see the proof of Theorem $1.4$ in~\cite{ChaikaEskin}.
    
    \begin{corollary} \label{cor: Oseledets deviations for the hodge bundle} 
      Suppose $(M,\omega) \in \mc{H}_1(\a)$ and $\nu_{\mc{M}}$ is the affine measure whose support is $\mc{M} = \overline{\mathrm{SL}_2(\R) \omega}$.
      Let $A$ be the Kontsevich-Zorich cocycle over $\mc{M}$. Denote by $\l_i$ the Lyapunov exponents of $A$ (with multiplicities) with respect to $\nu_{\mathcal M}$. For any $\theta\in[0,2\pi]$, suppose $\psi_1(t,\theta)\leq\dots\leq\psi_{2g}(t, \theta)$ are the eigenvalues of the matrix $A^*(g_t, r_{\theta}\omega)A(g_t, r_{\theta}\omega)$.
      Then, the Hausdorff dimension of the set
      \begin{equation*}
          \set{\th\in[0,2\pi]: \limsup_{t\r\infty} \frac{\log \norm{\psi_i(t,\theta)} }{t}
                  \geqslant 2\l_i + \e }
      \end{equation*}
      is strictly less than $1$.
    \end{corollary}

    \subsection*{Divergent Trajectories}
    The study of exceptional trajectories in Birkhoff's and Oseledets' theorems lends itself naturally to studying trajectories which frequently miss large sets with good properties.
    This problem is closely connected to studying divergent geodesics, i.e. geodesics which leave every compact subset of $\mc{H}_1(\a)$.
%     These are precisely the directions in which the systole tends to zero under the geodesic flow.
    Masur showed in \cite{Masur_nonergodic} that, for every translation surface $\omega$, the set of directions $\th$ for which $g_t r_\th \omega$ is divergent has Hausdorff dimension at most $1/2$.
    Cheung \cite{Cheung_nonergodic} showed that this upper bound is optimal by constructing explicit examples for which this upper bound is realized.
    
    In this paper, we study \textbf{divergent on average geodesics}, i.e., geodesics that spend asymptotically zero percent of the time in any compact set.
    \begin{definition}\label{defn: doa}
    A direction $\theta\in[0, 2\pi]$ corresponds to a divergent on average geodesic $g_t r_\th \omega$ if for every compact set $K \subset \mc{H}_1(\a)$,
    \begin{equation*}
     \lim_{T\r\infty} \frac{1}{T} \int_0^T \chi_K(g_tr_\th \omega) \;dt = 0.
    \end{equation*}
    \end{definition}
    Note that the set of divergent on average geodesics contains the set of divergent geodesics. Therefore, Theorem \ref{thrm: divergent on average} below strengthens \cite{Masur_nonergodic}.
%     by showing that Masur's upper bound on the Hausdorff dimension of the set of directions for divergent geodesics is also an upper bound on the Hausdorff dimension of the set of directions for divergent on average geodesics in strata.     
    \begin{theorem}\label{thrm: divergent on average}
    For any translation surface the Hausdorff dimension of the directions that correspond to divergent on average geodesics is at most $1/2$.
    \end{theorem}
     See also Theorem~\ref{thrm: H.dim of non-recurrent directions,fixed compact set} where we consider the set of directions with a prescribed divergence behavior in open strata with finitely many invariant submanifolds removed, which may be of independent interest.
     
     Combining Theorem~\ref{thrm: divergent on average} with the results in~\cite{BoshernitzanNogueira}, we derive the following bound on the dimension of non-weakly mixing interval exchange transformations (IETs) whose permutation is of type $W$.
     We refer the reader to Section~\ref{section: IETs} for detailed definitions.

     \begin{corollary} \label{cor: IETs}
     Suppose $\pi$ is a type $W$ permutation. Then, the Hausdorff codimension of the set of non-weakly mixing IETs (with respect to the Lebesgue measure) with permutation $\pi$ is at least $1/2$.
     \end{corollary}
     
     For $d\in \N$, we say a permutation $\pi$ on $\set{1,\dots,d}$ is a rotation if $\pi(i+1) = \pi(i)+1 \textrm{ mod } d $ for $1\leq i\leq d$.
     Avila and Forni~\cite{AvilaForni} showed that for any irreducible permutation, which is not a rotation, Lebesgue almost every IET is weakly mixing.
     In~\cite{AvilaLeguil}, this result was extended to show that for all such permutations, non-weakly mixing IETs have positive Hausdorff codimension.
     Thus, Corollary~\ref{cor: IETs} is an improvement of~\cite{AvilaLeguil} in the case of type $W$ permutations.
     Moreover, it is shown in~\cite{ChaikaMasur-dIET} that if $\pi$ is the permutation $(d,d-1,\dots,1)$ for $d\geq 5$, then the Hausdorff codimension of the set of non-weakly mixing IETs with permutation $\pi$ is at most $1/2$ (the case $d=4$ was done in~\cite{AthreyaChaika}).
     When $d$ is odd, the permutation $(d,d-1,\dots,1)$ is type $W$. Thus, we identify the exact Hausdorff dimension in this case.

    \subsection*{Outline of Proofs and Paper Organization}
    Our general approach is to deduce the desired results (Theorems~\ref{thrm: Birkhoff deviations thrm}, ~\ref{thrm: Oseledets deviations} and ~\ref{thrm: divergent on average}) from the analogous results for horocycle arcs (Theorems~\ref{thrm: Birkhoff deviations thrm for horocycles}, ~\ref{thrm: Oseledets deviations thrm for horocycles} and ~\ref{thrm: divergent on average for horocycles}, respectively). The reason is that horocycles are more convenient to work with as the geodesic flow normalizes the horocycle flow in $\mathrm{SL}_2(\R)$.
    This is carried out along with the proof of Corollary~\ref{cor: Birkhoff corollary} in Section~\ref{section_general_approach}.
    
	The strategy for proving Theorem~\ref{thrm: Birkhoff deviations thrm for horocycles} on deviations of Birkhoff averages consists of three main steps.
    First, we show that the convergence in~\eqref{Birkhoff} holds uniformly as the basepoint $\omega$ varies over compact sets in the complement of finitely many proper affine submanifolds. Theorem~\ref{thrm: quantitative Chaika-Eskin} is the precise statement.
    This result strengthens a result in~\cite{ChaikaEskin} and may be of independent interest.
    
    Next, we show that the Hausdorff dimension of directions whose geodesics frequently miss large compact sets, chosen with the help of a height function, is bounded away from $1$.
    This statement is made precise in Theorem~\ref{thrm: H.dim of non-recurrent directions,fixed compact set} whose proof is the main content of Section~\ref{section: doa}.    
   Using similar techniques, Theorem~\ref{thrm: divergent on average for horocycles} is proved in Section~\ref{HD_divergent_on_average}.
    
    Theorem~\ref{thrm: Birkhoff deviations thrm for horocycles} is proved in Section~\ref{section: birkhoff}.
    The idea is to treat a long orbital average as a sum of orbital averages over shorter orbit segments.
    With the help of Theorem~\ref{thrm: quantitative Chaika-Eskin},
    we show that most orbit segments which start from a suitably chosen large compact set with good properties, will have an orbital average close to the correct limit.
    Using Theorem~\ref{thrm: H.dim of non-recurrent directions,fixed compact set}, we control the dimension of those orbit segments which miss our good compact set.
    
    A key step is to show that the sum of such averages over orbit segments behaves like a sum of weakly dependent random variables, which is achieved by Lemma~\ref{propn: independence lemma}.
    This allows us to show that the measure of badly behaved long orbit averages decays exponentially.
    
	The proof of Theorem~\ref{thrm: Oseledets deviations thrm for horocycles} treating deviations in Oseledets' theorem spans Section~\ref{section: random walks} and Section ~\ref{section: oseledets}.
    It follows the same strategy as the one outlined above.
    It is shown in~\cite{ChaikaEskin} that Oseledets' theorem holds uniformly in the basepoint over large \emph{open} sets for random walk trajectories.
    Using Egorov's and Lusin's theorems, we translate these results into results about the Teichm\"{u}ller geodesic flow.
    This relies on the classical fact that a random walk trajectory is tracked by a geodesic, up to sublinear error.
    
    Finally, we show that trajectories which frequently miss such a large set with good properties exhibit deviation in the discrete Birkhoff averages of its indicator function.
    The dimension of those trajectories is in turn controlled by Theorem~\ref{thrm: discrete Birkhoff deviations thrm for horocycles}.
    
    In Section~\ref{section: IETs} we prove Corollary~\ref{cor: IETs}. In Proposition~\ref{propn: short intervals and recurrence} we relate the criterion for weak mixing of IETs with a type W permutation in \cite{BoshernitzanNogueira} and recurrence of Teichm\"uller geodesics in a stratum. The combination of this relation and our Theorem~\ref{thrm: divergent on average} finishes the proof.
    
    In Section~\ref{section: remarks}, we describe how to modify the proof of Theorem~\ref{thrm: Birkhoff deviations thrm} to show that the Hausdorff dimension of abelian differentials $\omega$ for which ergodic integrals along their Teichm\"uller flow orbit exhibit a definite amount of deviation from the correct limit in Birkhoff's theorem is strictly less than the dimension of $\overline{\mrm{SL}(2,\R)\omega}$.

    \begin{acknowledgement}
		The authors would like to thank Jon Chaika for suggesting the problems addressed in this article and for generously sharing his ideas on the project.
        This work grew out of the AMS Mathematics Research Communities workshop ``Dynamical Systems: Smooth, Symbolic, and Measurable" in June 2017, and we are grateful to the organizers. This material is based upon work supported by the National Science Foundation under Grant Number DMS 1641020. All authors are thankful for their support.
        This material is also based upon work supported by the National Science Foundation Graduate Research Fellowship Program under Grant Number DGE-1144082. P.A. gratefully acknowledges their support. C.U. gratefully acknowledges support from the NSF grants DMS-1405146 and DMS-1510034. 
	\end{acknowledgement}

\section{Preliminaries}
\label{section_general_approach}

   \subsection{Reduction to horocycles}
   We explain how to deduce Theorems~\ref{thrm: Birkhoff deviations thrm}, ~\ref{thrm: Oseledets deviations} and ~\ref{thrm: divergent on average} from the analogous results for horocycle arcs Theorems~\ref{thrm: Birkhoff deviations thrm for horocycles}, ~\ref{thrm: Oseledets deviations thrm for horocycles} and ~\ref{thrm: divergent on average for horocycles}, respectively.
   
   For any $\th \in [-\pi/4,\pi/4]$, the following equality holds.
  \[ r_\th = \check{h}_{-\tan\th} g_{\log\cos\th} h_{\tan\th} \]
  Recall that $g_t$ contracts $\check{h}_{-\tan\th}$, i.e., $g_t\check{h}_{-\tan\th}g_{-t}=\check{h}_{-e^{-2t}\tan\th}$, and $g_t g_{\log\cos\th} = g_{t+\log\cos\th}$. Therefore, we have that in each theorem formulated in the introduction $\th$ belongs to the exceptional set 
  if and only if $\tan\th$ belongs to the exceptional set in the corresponding theorem formulated below.
  
  Finally, the bounds for the Hausdorff dimensions of the corresponding sets are preserved as the map $\th \mapsto \tan \th$ is bi-Lipschitz on $[-\pi/4,\pi/4]$. 
   
\begin{theorem}[Analogue of  Theorem~\ref{thrm: Birkhoff deviations thrm}] \label{thrm: Birkhoff deviations thrm for horocycles}

      Suppose $\mc{M}\subseteq \mc{H}_1(\a)$ is an affine invariant submanifold and $\nu_{\mc{M}}$ is the affine measure whose support is $\mc{M}$.
      Then, for any bounded continuous function $f$ on $\mc{M}$ and any $\e>0$, there exist affine invariant submanifolds $\mc{N}_1,\dots, \mc{N}_k$, properly contained in $\mc{M}$, and $\d\in (0,1)$, such that for all $\omega \in \mc{M}\backslash \left( \cup_{i=1}^k \mc{N}_i \right)$,      
      the Hausdorff dimension of the set
      \begin{equation*}
          \set{s\in[-1,1]: \limsup_{T\r\infty} \frac{1}{T} \int_0^T f(g_th_s\omega)\;dt
                  \geq \int\limits_{\mathcal M} f \;d\nu_\mc{M} + \e }
      \end{equation*}
      is at most $\d$.
  \end{theorem}
  We remark that minor modifications of the proof of Theorem~\ref{thrm: Birkhoff deviations thrm for horocycles} also yield an upper bound on the Hausdorff dimension of the set of directions for which the $\liminf$ is less than the correct limit by a definite amount.
  Moreover, the exceptional set in Theorem~\ref{thrm: Birkhoff deviations thrm} can be written as 
  \begin{multline*}
  \set{\theta\in[0,2\pi]: \limsup_{T\r\infty} \left|\frac{1}{T} \int_0^T f(g_tr_{\theta}\omega)\;dt
                  - \nu_\mc{M}(f)\right|\geq \e } \\=
	\set{\theta\in[0,2\pi]: \limsup_{T\r\infty} \frac{1}{T} \int_0^T f(g_tr_{\theta}\omega)\;dt
                  \geq \nu_\mc{M}(f) + \e }\\
                  \cup
	\set{\theta\in[0,2\pi]: \liminf_{T\r\infty} \frac{1}{T} \int_0^T f(g_tr_{\theta}\omega)\;dt
                  \leq \nu_\mc{M}(f) - \e }
  \end{multline*}
  where $\nu_\mc{M}(f) = \int_{\mc{M}} f\;d_\mc{M}$.
  Thus, Theorem~\ref{thrm: Birkhoff deviations thrm} follows from the reduction to horocycles, Theorem~\ref{thrm: Birkhoff deviations thrm for horocycles}, and its variant for the $\liminf$.

 \begin{theorem}[Analogue of Theorem~\ref{thrm: Oseledets deviations}] \label{thrm: Oseledets deviations thrm for horocycles} 
 
 Suppose $(M,\omega) \in \mc{H}_1(\a)$ and $\nu_{\mc{M}}$ is the affine measure whose support is $\mc{M} = \overline{SL_2(\R) \omega}$. Let $V$ be a continuous (on $\mc{M}$) $SL_2(\R)$ invariant sub-bundle of (some exterior power of) the Hodge bundle. Assume that $A_V$ is strongly irreducible with respect to $\nu_{\mc{M}}$, where $A_V$ is the restriction of the Kontsevich-Zorich cocycle to $V$.
      Then, for any $\e>0$, there exist affine invariant submanifolds $\mc{N}_1,\dots, \mc{N}_k$, properly contained in $\mc{M}$, and $\d\in (0,1)$, such that for all $\omega \in \mc{M}\backslash \left( \cup_{i=1}^k \mc{N}_i \right)$, the Hausdorff dimension of the set
      \begin{equation*}
            \set{s\in[-1,1]: \limsup_{t\r\infty} \frac{\log \norm{A_V(g_t,h_s\omega)}}{t}
                    \geq \l_V + \e }
        \end{equation*}
      is at most $\d$, where $\l_V$ denotes the top Lyapunov exponent for $A_V$ with respect to $\nu_\mc{M}$.
 
\end{theorem}

\begin{theorem}[Analogue of Theorem~\ref{thrm: divergent on average}]\label{thrm: divergent on average for horocycles}

 Suppose $(M,\omega) \in \mc{H}_1(\a)$. Then, the Hausdorff dimension of the set
      \begin{equation*}
            \set{s\in[-1,1]: \text{the geodesic } g_t h_s \omega \text{ is divergent on average} }
        \end{equation*}
      is less than or equal to $\frac{1}{2}.$
\end{theorem}

\subsection{Properties of Hausdorff dimension}
    The exceptional sets we study in this paper are of the form $A=\limsup\limits_{n\r\infty} A_n$, that is 
    \[A=\bigcap\limits_{n\geq 1}\bigcup\limits_{l\geq n}A_l\]
    for a sequence of subsets $A_n$ of the real line.
    
    In this section, we reduce the problem of finding an upper bound on the Hausdorff dimension of such sets to the problem of finding efficient covers of the $A_n$ (see Lemma~\ref{lemma: H.dim of limsup set via covering lemma}).
    
    First, we recall the definition of the Hausdorff dimension.    
    Let $A$ be a subset of a metric space $X$. For any  $\rho ,\beta >0$, we define
        \[ H_\rho^\b (A) = \inf\set{ \sum_{I\in \mc{U}} diam(I)^\beta: \mc{U} \text{ is a cover of A by balls of diameter } <\rho  }. \]
        
        Then, the $\beta$-dimensional Hausdorff measure of $A$ is defined to be
        \[ H^\b (A) = \lim_{\rho \r 0} H_\rho^\b(A). \]
 
 \begin{definition}
        The Hausdorff dimension of a subset $A$ of a metric space $X$ is equal to
        \[ dim_{H}(A) = \inf\set{\b\geq 0: H^\b(A) = 0}= \sup \set{\b\geq 0: H^\b(A) = \infty} \]
 \end{definition}
 
        The following lemma provides an upper bound on the Hausdorff dimension of a set for which we have efficient covers.         
      \begin{lemma} \label{lemma: H.dim of limsup set via covering lemma}
        Let $\set{A_n}_{n\geq 1}$ be a collection of subsets of $\R$.
        Suppose there exist constants $C,C',t>0$ and $\l\in (0,1)$ such that for each $n$, $A_n$ can be covered with $C e^{2(1-\l)tn}$ intervals of radius $C' e^{-2tn}$.
        Then, the Hausdorff dimension of the set $A = \limsup\limits_{n\r\infty} A_n$ is at most $1-\l$.
      \end{lemma}
      
      \begin{proof}
      	Let $\b \in (1-\l, 1)$ and $H^\beta$ denote the $\beta$-dimensional Hausdorff (outer) measure on $\R$.
         We show that $H^\b(A) = 0$, and that implies the Lemma.
         For any $\rho \in (0,1)$, let $n_0=n_0(\rho)$ be a natural number such that $e^{-2tn} < C\rho $ for all $n\geq n_0$. Notice that $n_0$ tends to infinity as $\rho$ goes to $0$.
         Denote by $\mc{U}_n$ a cover of the set $A_n$ by $C e^{2(1-\l)tn}$ intervals of radius $C'e^{-2tn}$. Then, $\mc{U}=\bigcup_{n\geq n_0} \mc{U}_n$ is a cover of $A$ for which the following holds.
         \begin{multline*}
         	\sum_{I \in \mc{U}} diam(I)^\beta = \sum_{n\geq n_0} \sum_{I\in \mc{U}_n} diam(I)^\b 
            	= (C')^\b \sum_{n\geq n_0} \# \mc{U}_n  e^{-2\b tn}
                \leq (C')^\b C \sum_{n\geq n_0}  e^{2(1-\l-\b)tn},
         \end{multline*}
         where $\#\mc{U}_n$ is the number of intervals in the cover $\mc{U}_n$.
        
        Thus, since $1-\l-\beta<0$, we obtain
        \[ H^\b_\rho(A) \leq (C')^\b C \sum_{n\geq n_0}  e^{2(1-\l-\b)tn} = (C')^\b C \frac{e^{2(1-\l-\beta)tn_0}}{1-e^{2(1-\l-\beta)t}}
        \xrightarrow{\rho\r 0} 0. \]
        
        This implies that $H^\b(A) = 0$ for all $\b\in (1-\l,1)$.
      \end{proof}
      
      Let us also recall some basic facts about Hausdorff dimension which will be useful for us. The first concerns the dimension of product sets.
      \begin{proposition}[Corollary 8.11 in~\cite{Mattila}]
      \label{propn: dimension of products}
      If $A,B \subset \R^d$ are Borel sets, then $dim_H(A\times B) \geq dim_H(A) + dim_H(B)$.      
      If, in addition, the upper packing dimension of $B$ is equal to its Hausdorff dimension, then
      \[ dim_H(A\times B) = dim_H(A) + dim_H(B) \]
      \end{proposition}
   	  We remark that the lower bound on the dimension of the product is a classical fact while the upper bound can be obtained directly when $B$ is an open ball, which is the case we will be interested in.

%%%%%%%%%%%%%%%%%%%%%%%%%%%%%%%%%%%%%
\subsection{Proof of Corollary~\ref{cor: Birkhoff corollary}}
	Using a simple approximation argument, we may assume that $f$ is Lipschitz.
	By a similar argument to the one following Theorem~\ref{thrm: Birkhoff deviations thrm for horocycles}, it suffices to prove that the following set
    \begin{equation*}
         \overline{B}(f,\e) := \set{x\in \mc{M}: \limsup_{T\r\infty} 
          \frac{1}{T} \int_0^T f(g_tx)\;dt \geq
          \int_{\mathcal M} f \;d\nu_\mc{M} +  \e }
      \end{equation*}
      has positive Hausdorff codimension in $\mc{M}$.
      Let $\d>0$ and $\mc{N}_1,\dots, \mc{N}_k$ be the affine invariant submanifolds properly contained in $\mc{M}$ which are provided by Theorem~\ref{thrm: Birkhoff deviations thrm for horocycles}, depending on $f$ and $\e$ and suppose $\omega \in \mc{M}\backslash \left(\cup_{i=1}^k \mc{N}_i \right)$.
      
      Since the action of $\mathrm{SL}(2,\R)$ is locally free\footnote{For every compact set $K\subset \mc{M}$, there exists a bounded neighborhood of identity $B\subset \mathrm{SL}(2,\R)$ such that the map $(g,x) \mapsto gx$ is injective from $B\times K$ into $\mc{M}$.},
      we can find a small neighborhood of identity $\mc{O}_\omega \subset \mathrm{SL}(2,\R)$ such that the map $g \mapsto g\omega$ is injective on $\mc{O}_\omega$.
      By making $\mc{O}_\omega$ smaller if necessary, we may assume that $\mc{O}_\omega$ is the diffeomorphic image of an open bounded neighborhood of $0$ in the Lie algebra of $\mathrm{SL}(2,\R)$ under the exponential map.
      In particular, there are bounded neighborhoods $\mc{O}_\omega^s, \mc{O}_\omega^c$ and $\mc{O}_\omega^u$ of $0$ in $\R$ such that the map
      \begin{equation} \label{eqn: coords}
      (z, r,s) \mapsto \check{h}_z g_r h_s 
      \end{equation}  
      is a diffeomorphism from $\mc{O}_\omega^s \times \mc{O}_\omega^c \times \mc{O}_\omega^u$ onto $\mc{O}_\omega$.
      Define the following set
      \begin{equation*}
         \overline{B}(f,\e)^u_\omega := \set{s\in \mc{O}_\omega^u: \limsup_{T\r\infty} 
          \frac{1}{T} \int_0^T f(g_th_s\omega)\;dt \geq
          \int_{\mathcal M} f \;d\nu_\mc{M} +  \e }
      \end{equation*}
      By Theorem~\ref{thrm: Birkhoff deviations thrm for horocycles}, the Hausdorff dimension of $\overline{B}(f,\e)^u_\omega$ is at most $\d \lneq 1$.
      Now, suppose that $x = g\omega \in \overline{B}(f,\e)\cap \mc{O}_\omega \omega$ and write $g = \check{h}_z g_r h_s$.
      Since $g_t$ contracts $ \check{h}_z$ and commutes with $g_r$, using the fact that $f$ is Lipschitz, we see that $s\in \overline{B}(f,\e)^u_\omega$.
      Conversely, for all $s\in \overline{B}(f,\e)^u_\omega$ and all $(z,r) \in \mc{O}_\omega^s \times \mc{O}_\omega^c$, we have that $\check{h}_z g_r h_s\omega \in \overline{B}(f,\e)\cap \mc{O}_\omega \omega$.
      
      In particular, we have the identification
      \[
      		\overline{B}(f,\e)\cap \mc{O}_\omega \omega\cong \mc{O}_\omega^s \times 
      		\mc{O}_\omega^c \times \overline{B}(f,\e)^u_\omega\]
      under the smooth coordinate map in~\eqref{eqn: coords}.      
      Thus, by Proposition~\ref{propn: dimension of products}, since the upper packing dimension of an open interval in $\R$ is equal to its topological and Hausdorff dimension, we get that
      \begin{align*}
      &dim_H(\overline{B}(f,\e)\cap \mc{O}_\omega\omega) = 
          dim_H\left(\mc{O}_\omega^s \times 
      		\mc{O}_\omega^c \times \overline{B}(f,\e)^u_\omega \right) 
            \leq 2+\d
      \end{align*} 
        
      The above argument shows that that dimension of the intersection of $\overline{B}(f,\e)$ with any open subset of an $\mathrm{SL}(2,\R)$ orbit in the complement of $\cup_{i=1}^k \mc{N}_i$ is at most $2+\d \neq 3$.

      \iffalse
      To upgrade this bound to a bound on the dimension of $\overline{B}(f,\e)$, we wish to use Proposition~\ref{propn: dimension and foliations}.
      
      The local freeness of the $SL(2,\R)$ action on $\mc{M}$ implies that the $SL(2,\R)$ orbits form a smooth foliation of $\mc{M}$.
      In particular, given $\omega \in \mc{M}\backslash \left(\cup_{i=1}^k \mc{N}_i \right)$, we can find a neighborhood $\mc{U}_\omega$ of $\omega$ in $\mc{M}\backslash \left(\cup_{i=1}^k \mc{N}_i \right)$ and a foliation chart
      \[ \phi : \mc{U}_\omega \r (-1,1)^3 \times (-1,1)^{dim(\mc{M})-3} \]
      such that for some $B\subset (-1,1)^3$ and all $y \in (-1,1)^{dim(\mc{M})-3}$, $\phi^{-1}(B\times\set{y})$ is an orbit of $\phi^{-1}(0,y)$, where $0\in(-1,1)^3$, under a fixed bounded neighborhood of identity in $SL(2,\R)$.
      Thus, by Proposition~\ref{propn: dimension and foliations}, applied with $A = \overline{B}(f,\e)\cap \mc{U}_\omega$, $\eta = 2+\d$ and $m = dim(\mc{M})-3$, we get that
      \[ dim_H(\overline{B}(f,\e)\cap \mc{U}_\omega) \leq dim(\mc{M})-3 + 2+\d = dim(\mc{M})-1+\d \lneq dim(\mc{M}) \]
        Thus, we conclude that $dim_H(\overline{B}(f,\e)\cap \mc{M}\backslash \left(\cup_{i=1}^k \mc{N}_i \right))$ is strictly smaller than $dim(\mc{M})$.
        But, since $ \mc{N}_1,\dots, \mc{N}_k$ are proper submanifolds of $\mc{M}$ of smaller dimension, we see that 
        \[ dim_H(\overline{B}(f,\e)) \lneq dim(\mc{M}) \]
        which is the desired statement.
        \fi

\section{The Contraction Hypothesis and Analysis of Recurrence}
\label{section: doa}

	In this section, we study the problem of the Hausdorff dimension of trajectories with prescribed divergence behavior.
    We prove an abstract result for $\mathrm{SL}_2(\R)$ actions on metric spaces which satisfy the \emph{Contraction Hypothesis} (Definition~\ref{defn: height functions}) in the terminology of Benoist and Quint~\cite[Section 2]{BenoistQuint}. The results in this section closely follow the ideas in~\cite{KKLM-SingSystems}.
    
    Let $X$ be a manifold equipped with a smooth $\mathrm{SL}(2,\R)$ action.
    For $t, \d>0$, $N\in \N$, $Q \subset  X$ a (compact) set and $x\in X$, define the following set
    \begin{align} \label{defn: set of non-recurrent directions}
    	Z_x(Q, N, t, \d) = \set{ s\in[-1,1] : \frac{1}{N} \sum_{l=1}^N \chi_Q(g_{lt} h_s x) \leq 1-\d  }
    \end{align}
    where $\chi_Q$ denotes the indicator function of $Q$.
    
 %   The sets $Q$ we shall be interested in will be sublevel sets of so-called Margulis functions.
 %   These were introduced in~\cite{EskinMargulisMozes} in the context of homogeneous spaces and in~\cite{EskinMasur} and~\cite{Athreya2006} in the context of moduli spaces.

    \begin{definition} [The Contraction Hypothesis]
    \label{defn: height functions}
    Let $Y$ be a proper $\mathrm{SL}(2,\R)$-invariant submanifold of $X$ ($Y=\emptyset$ is allowed).
    The action of $\mathrm{SL}_2(\R)$ on $X$ is said to satisfy the \textbf{contraction hypothesis} with respect to $Y$
    if there exists a proper, $SO(2)$-invariant function $\a:X\r [1,\infty]$
    satisfying the following properties:
      \begin{enumerate}
      	\item $\a(x) = \infty$ if and only if $x\in Y$.
        \item There is constant $\s>0$ such that for all $x\in X\backslash Y$ and all $t>0$,
        	\begin{align}
            \label{eqn: growth of height function with g_t}
        		e^{-\s t} \a(x) \leq \a(g_t x) \leq e^{\s t} \a(x)
        	\end{align}
            
		\item There exists a constant $b = b(Y) > 0$, such that for all $a \in (0,1)$
        there exists $t_0 = t_0(a)>1$
        	so that for all $t > t_0$ and all $x\in X\backslash Y$,
            \begin{align}
            \label{eqn: height function circle integral estimate}
            	\int_0^{2\pi} \a(g_t r_\theta x) \;d\theta \leq a \a(x) + b 
            \end{align}
        \item For all $M\geq 1$, the sets $\overline{\set{x\in X: \a(x) \leq M}}$, denoted by $X_{\leq M}$, form a compact exhaustion of $X\backslash Y$. 
      \end{enumerate}
   The function $\alpha$ is called the \textbf{height function}.
    \end{definition}
    
    We remark that the study of height functions as in Definition~\ref{defn: height functions} originated in~\cite{EskinMargulisMozes} in the context of homogeneous spaces.
    
Throughout this section $X$ is a manifold equipped with a smooth $\mathrm{SL}(2,\mathbb{R})$ action and satisfies the contraction hypothesis with respect to $Y$, which is a proper $\mathrm{SL}(2,\mathbb{R})$-invariant submanifold of $X$.

%    Following the arguments in~\cite{KKLM-SingSystems}, we obtain the following:
Our goal in this section is to prove the following theorem.
    
    \begin{theorem} \label{thrm: H.dim of non-recurrent directions,fixed compact set}
        Given $\d > 0$, there exists $M_0 = M_0(\d)>0$ and $t_0 > 0$ such that for all $M\geq M_0$ and all $t \geq t_0$,
        there exists $\l=\l(\d,t)  \in (0,1)$
        such that for all $x\in X\backslash Y$,
        the Hausdorff dimension of the set
        $\limsup_{N\r\infty} Z_x( X_{\leq M}, N,t,\d)$ is at most $1-\l$, where $\l$ tends to $0$ as $t \r \infty$.
    \end{theorem}

    The proof of Theorem~\ref{thrm: H.dim of non-recurrent directions,fixed compact set} can be found in Section~\ref{proof_thrm_H.dim,fixed compact}. It should be noted that the difference between this theorem and ~\cite[Theorem 1.5]{KKLM-SingSystems} is the flexibility in the step size $t$. As a result, the upper bound on the Hausdorff dimension of the considered set depends on $t$.
    In fact, the proof of Theorem~\ref{thrm: H.dim of non-recurrent directions,fixed compact set} gives an explicit value for $\l$ as a function of $t$ and $\delta$.
    See also remark~\ref{remark: Hdim bound formula} for an explicit choice for $M_0$ depending on $\delta$.

    %%%%%%%%%%%%%%%%%%%%%%%
    
    \subsection{Estimates for integrals over horocycle orbits}
In this section we obtain an integral estimate similar to~\eqref{eqn: height function circle integral estimate} for integrals over an entire horocycle orbit.

    \begin{lemma}
    \label{lemma: integral estimate for horocycle interval}
    	Let $\a: X \r [1,\infty]$ be a height function.
        Then, there is a constant $\bar b>0$, such that for all $\bar a \in (0,1)$,
        there exists  $\bar{t_0} = \bar{t_0}(\bar a)>0$ so that for all $t > \bar{t_0}$ and all $x\in X\backslash Y$, 
            \begin{equation}
            \label{eqn: height function horocycle interval integral estimate}
            	\int_{-1}^1 \a(g_t h_s x) \;ds \leq \bar a \a(x) + \bar b 
            \end{equation}
    \end{lemma}
    
    \begin{proof}
    	Let $b=b(Y)$ and for any $a\in(0,1)$ we have $t_0=t_0(a)$ as in Definition~\ref{defn: height functions}.
        Then, by~\eqref{eqn: height function circle integral estimate},
         for any $t>t_0$ and all $x\in X$,
        \[
        	\int_{-\pi/4}^{\pi/4} \a(g_t r_\theta x) \;d\theta \leq a \a(x) + b 
        \]
        
        For any $\theta \in [-\pi/4,\pi/4]$ the following equality holds.
        \[  g_t r_\theta = 
        \underbrace{
        	\begin{pmatrix}
        		1 & 0 \\ - e^{-t} \tan(\theta) & 1
        	\end{pmatrix}
            }_{ \check{h}_{- e^{-t} \tan(\theta)}}
            \underbrace{
            \begin{pmatrix}
            	\cos(\theta) & 0 \\ 0 & \sec(\theta)
            \end{pmatrix}
            }_{g_{\log(\cos(\theta))} }
            g_t h_{\tan(\theta)}
        \]
        
        Let $K = SO(2)$. From the $KAK$ decomposition of $\mathrm{SL}_2(\R)$, $K$-invariance of $\a$
        and Property $(2)$ in Definition~\ref{defn: height functions},
        it is easy to see that there exists a positive constant $c_0\geq 1$, that is independent of $t$,
        such that for all $\theta \in [-\pi/4,\pi/4]$ and all $x\in X$,
        \[ \a(g_t h_{\tan(\theta)} x) \leq c_0 \a(g_t r_\theta x). \]
        
        Thus, we get that
        \[
        	\int_{-\pi/4}^{\pi/4} \a(g_t h_{\tan(\theta)} x) \;d\theta \leq c_0 (a \a(x) + b). 
        \]
        
        Using a change of variable $s = \tan(\theta)$ and noting that the Jacobian of this change of variable is uniformly bounded on $[-\pi/4,\pi/4]$, we obtain 
        \[
        	\int_{-1}^{1} \a(g_t h_{s} x) \;ds \leq 2c_0 (a \a(x) + b). 
        \]
      That implies the lemma with $\bar b = 2c_0 b$, $a=\frac{\bar a}{2c_0}\in(0,1)$ and $\bar t_0 = t_0(a)$.
    \end{proof}

    Our next lemma replaces integration over a compact subinterval of the horocycle with an integral over the entire horocycle orbit against a Gaussian measure.
    This is a technical step needed to carry over some results from~\cite{KKLM-SingSystems} to our setting.
    
    \begin{lemma}
    \label{lemma: integral estimate for horocycle Gaussian}
    Let $\a: X \r [1,\infty]$ be a height function.
        Then, there is a constant $b>0$, such that for all $a \in (0,1)$,
        there exists  $t_0 = t_0(a)>0$ so that for all $t > t_0$ and all $x\in X\backslash Y$, 
            \begin{equation}
            \label{eqn: height function horocycle integral estimate Gaussian}
            	\int_{\R} \a(g_t h_s x) \;d\rho_1(s) \leq a \a(x) +b
            \end{equation}
        where $d\rho_1(s) = e^{-s^2}ds$ is a mean $0$, variance $1$ Gaussian.
    \end{lemma}
    
    \begin{proof}
    
   By Property $(2)$ in Definition~\ref{defn: height functions}, the $KAK$ decomposition of $\mathrm{SL}_2(\R)$ and $K$-invariance of $\a$, there exist constants $\s>0$ and $C>1$, such that for all $q\in \R$,
        \[ \a(h_q x) \leq C |1+q|^\s. \]
       
       Let $\bar b$ and  $ t_0 = \bar{t_0}(\bar a)$ for $\bar a\in(0,1)$ be given by Lemma~\ref{lemma: integral estimate for horocycle interval}. For any $n\in \Z$, we define $f(n) = \min\set{|n|, |n+2|}$.
      
      Then, for any $t>t_0$ and $x\in X$ we have the following.
        
        \begin{align*}
        	\int_{\R} \a(g_t h_s x) \;d\rho_1(s) &=
            	\sum_{n\in 2\Z-1} \int_{[0,2]+n} \a(g_t h_s x) e^{-s^2}ds \\
                &\leq\sum_{n\in 2\Z-1} e^{-(f(n))^2} \int_{-1}^1 \a(g_t h_sh_{n+1} x) ds \\
                &\leq \sum_{n\in 2\Z-1} e^{-(f(n))^2} (\bar a \a(h_{n+1} x) + \bar b)
                 \qquad \qquad \qquad\text{, by Lemma~\ref{lemma: integral estimate for horocycle interval}} \\
                &\leq C \bar a \a(x) \sum_{n\in 2\Z-1} e^{-(f(n))^2}|n+2|^\s 
                + \bar b \sum_{n\in 2\Z-1} e^{-(f(n))^2}.
        \end{align*}
        
        This completes the proof by setting $$b=\bar b \sum_{n\in 2\Z-1} e^{-(f(n))^2} \quad\text{ and }\quad \bar a = \min\left\{\frac{1}{2}, \frac{a}{C\sum_{n\in 2\Z-1} e^{-(f(n))^2}|n+2|^\s}\right\}.$$ 

    \end{proof}

    \subsection{Coverings and Long Excursions}

In this section, we aim to find efficient coverings for the set of directions for which geodesics take long excursions outside of certain fixed compact sets.
We closely follow~\cite[Section 5]{KKLM-SingSystems}. 

Throughout this section, we fix $x$ in $X\backslash Y$ and use $Z_x(M,N,t)$ to denote the set \linebreak $Z_x(X_{\leq M}, N,t,1)$ defined in~\eqref{defn: set of non-recurrent directions}. Moreover, let $b>0$ and $t_0 = t_0(a) > 1$ for $a \in (0,1)$ be as in Lemma~\ref{lemma: integral estimate for horocycle Gaussian}.

The following is the main result of this section.

  \begin{proposition} \label{propn: integral estimate over Z_x}
    There exist constants $C_1,C_2 \geq 1$ (independent of $x$ and $a$)
    such that for all $M > C_2 b/a$, all $t\geq t_0$ and all $N\in \N$,
    \begin{align}
    	\int_{Z_x(M,N-1,t)} \a(g_{Nt} h_s x) \;ds \leq C_1 [(2a)^N \a(x) + (2a)^{N-1} b]
    \end{align}
  \end{proposition}
  
  Here we relax the restrictions on the height of $x$ and on the dependence of $M$ on $t$ in comparison with \cite[Proposition 5.1]{KKLM-SingSystems}
  
  \begin{proof}
  	The proof is the same as that of \cite[Proposition 5.1]{KKLM-SingSystems} with minor modifications that we discuss now.
    Using Property (2) in Definition~\ref{defn: height functions}, let $C_2\geq 1$ be such that
    for all $s\in [-2,2]$,
    \[  C_2^{-1} \leq \frac{\a(h_s x)}{\a(x)} \leq C_2.  \]
    
    Consider $M > C_2 b/a$ and $t>t_0$. Let $y\in X\backslash Y$ be so that $\a(y) >b/a$.
    By Lemma~\ref{lemma: integral estimate for horocycle Gaussian}, we have
    \begin{align} \label{eqn: remove b from integral estimate}
    	\int_\R \a(g_t h_s y) \;ds &\leq a \a(y) + b \leq 2 a \a(y).
    \end{align}
    
    Let $N\in \N$ and define the following set.
    \begin{align*}
    	Z = \set{(s_1,\dots,s_{N})\in \R^{N}: \a(g_{t} h_{s_k} g_t h_{s_{k-1}}\cdots
        		g_t h_{s_1} x) > b/a, \forall k =1,\dots, N-1 }
    \end{align*}
    Applying~\eqref{eqn: remove b from integral estimate} repeatedly, we get
    \begin{align} \label{eqn: iterated integral estimate}
    	\int \cdots \int_Z \a(g_{t} h_{s_N} \cdots g_t h_{s_1} x) \;d \rho_1(s_N)
        \cdots d\rho_1(s_1) &\leq (2a)^{N-1} \int_\R \a(g_t h_s x) \;d\rho(s)
        \nonumber \\
        &\leq (2a)^N \a(x) + (2a)^{N-1} b
    \end{align}
    On the last line, we applied Lemma~\ref{lemma: integral estimate for horocycle Gaussian}, but since we don't insist that $\a(x) > b/a$, we get the extra term with $b$.
    The rest of the proof is identical to that of ~\cite[Proposition 5.1]{KKLM-SingSystems}, where we take the constant $C_1$ to be the implicit constant in that proof depending only on the Radon-Nikodym derivative of the Lebesgue measure with respect to a certain bounded variance Gaussian on $[-1,1]$.
  \end{proof}

  %%%%%%%%%%%%%%%%%%%%%%%%%%%%%%%%%%%%%%%%%%%%%%%%
  As a corollary, we obtain the following covering result.
  
  \begin{corollary}  \label{cor: first covering lemma}
  There exists $b > 0$ such that for all $a \in (0,1)$, there exists $t_0>1$
  such that for all $x\in X$, all $M > C_2^2 b/a$, all $N\in \N$, and all $t>t_0$,
  the set $Z_x(M,N,t)$ can be covered by $2C_1 C_2(2a)^N e^{2tN} \max\set{1,\frac{\a(x)}{M}}$ intervals of radius $e^{-2tN}$, 
  where $C_1,C_2 > 1$ are the absolute constants in Proposition~\ref{propn: integral estimate over Z_x}.
  \end{corollary}
  
  \begin{proof}
  	By taking into account the different upper bound obtained in Proposition~\ref{propn: integral estimate over Z_x}, the proof is identical to that of~\cite[Corollary 5.2]{KKLM-SingSystems}.
  \end{proof}

  %%%%%%%%%%%%%%%%%%%%%%%%%%%%%%
   \begin{proposition} [cf. Theorem 1.5 in~\cite{KKLM-SingSystems}]\label{propn: main covering statement}
    Suppose $x\in X\backslash Y$.
    Then, for any $\d,a\in (0,1)$ there exist $M_0 > 1$ and $t_0>1$, depending only on $a$, such that
    for all $M>M_0$, all $t> t_0$ and all $N\in \N$, the set $Z_x(X_{\leq M},N,t,\d)$ can be covered with $2^N C_1^N (2a)^{\d N} e^{2tN} C(x)$ intervals 
		of radius $e^{-2tN}$, where $C_1 > 1$ is as in Corollary~\ref{cor: first covering lemma} and
        $C(x) = \max\set{1, \frac{\a(x)}{M}}$.
	\end{proposition}
    
  \begin{proof} %[Proof of Proposition~\ref{propn: main covering statement}]
  We now describe the modifications needed on the proof of~\cite[Theorem 1.5]{KKLM-SingSystems} in order to prove the proposition.
   In the same notation as in Corollary~\ref{cor: first covering lemma}, we take $M_0 = C_2^2 b/a$ and let $M>M_0$. Then, the rest of the proof follows the same induction scheme used in~\cite[Theorem 1.5]{KKLM-SingSystems} with the base case being Corollary~\ref{cor: first covering lemma}.
   The only modification on the scheme is to skip the steps involving enlarging $M$ depending on the largeness of the step size $t$ and instead work directly with the bound on covers provided by the preceding corollary.
   
   In particular, in the second case of the inductive step in~\cite[Theorem 1.5]{KKLM-SingSystems}, $M$ is assumed large enough depending on $t$ to apply their covering result~\cite[Corollary 5.2]{KKLM-SingSystems} which only applies to $x\in X$ with $\a(x)$ sufficiently large.
   Since Corollary~\ref{cor: first covering lemma} above works for all $x$, such restriction on $M$ is not needed.
  \end{proof}

 \subsection{Proof of Theorem~\ref{thrm: H.dim of non-recurrent directions,fixed compact set}}\label{proof_thrm_H.dim,fixed compact}   

Let $C_1$ be the constant in Proposition~\ref{propn: main covering statement}. Fix $a \in (0,1/2)$ such that $\d > -\ln(2C_1)/\ln(2a)$.

        Let $M_0= M_0(a)$ and $t_0= t_0(a)$ be as in Proposition~\ref{propn: main covering statement}.
        Let $M>M_0$ and $t>t_0$. Define $Q = X_{\leq M} $, $\g = -\ln(2a)/2t$ and $\b = \ln (2C_1)/2t $, i.e., $2C_1 = e^{2\b t}$.
        Then, by Proposition~\ref{propn: main covering statement}, for all $N\in \N$, we can cover the set $Z_x(Q,N,t,\d)$ with
        $C' e^{2tN(1+ \b -\d\g)  }$ intervals of radius $e^{-2tN}$, for some constant $C'$ depending only on $x$.
        Note that $C'$ is finite by our assumption that $x \in X\backslash Y$.
        
        Therefore, by Lemma~\ref{lemma: H.dim of limsup set via covering lemma}, the Hausdorff dimension of the set $\limsup_{N\rightarrow \infty} Z_x(Q,N,t,\d)$ is at most
        $1 + \b -\d\g >0$.
       By the choice of $a$, this upper bound is strictly less than $1$. 
       Finally, we note that by definition of $\b$ and $\g$, our upper bound is uniform over all $x\in X\backslash Y$.
       
        \begin{remark} \label{remark: Hdim bound formula}
    The proofs of Proposition~\ref{propn: main covering statement} and Theorem~\ref{thrm: H.dim of non-recurrent directions,fixed compact set} show that one can choose $M_0 = c'' b e^{c'/ \d} $
    in the conclusion of Theorem~\ref{thrm: H.dim of non-recurrent directions,fixed compact set},
    for some positive constants $c'$ and $c''$, where $b$ is as in Definition~\ref{defn: height functions}.
    \end{remark}
    
    \section{Hausdorff Dimension of The Divergent on Average Directions}\label{HD_divergent_on_average}
%\subsection{Proof of Theorem~\ref{thrm: divergent on average}}

In this section we prove Theorem~\ref{thrm: divergent on average for horocycles} which implies Theorem~\ref{thrm: divergent on average} (see Section~\ref{section_general_approach}).

\begin{proof}[Proof of Theorem~\ref{thrm: divergent on average for horocycles}]
	Consider the set
    \[ Z = \set{s\in [-1,1]: g_t h_s \omega \textrm{ diverges on average} } \].
    
    Notice that for all compact sets $Q \subset \mc{H}_1(\a)$, all $0 \leq \d \leq 1$ and all $t>0$,
    \begin{equation*}
    	Z \subseteq \liminf_{N\r \infty} Z_\omega(Q,N,t,\d):=\bigcup_{N_0=1}^{\infty}\bigcap_{N=N_0}^{\infty}Z_\omega(Q,N,t,\d)
    \end{equation*}
    where $Z_\omega(Q,N,t,\d)$ is defined in~\eqref{defn: set of non-recurrent directions}.
    
	Building on earlier work of~\cite{EskinMasur}, it is shown in~\cite{Athreya2006} that the $\mrm{SL}_2(\R)$ action on $X=\mc{H}_1(\a)$ satisfies the contraction hypothesis with respect to $Y = \emptyset$.
    The precise statement is the following:
    \begin{lemma} [Lemma 2.10 in~\cite{Athreya2006}] 
    \label{lemma: Athreya integral estimate}
    	For every $0<\eta<1$, there exists a function $\a_\eta\colon~X~\r~\R^+$ satisfying item $(1),(2)$ and $(4)$ of Definition~\ref{defn: height functions}.
        Moreover, there are constants $c = c(\eta)$, and $t_0 = t_0(\eta)>0$
        	so that for all $t > t_0$,
        	there exists $b= b(t,\eta)>0$ such that for all $x\in X$,
            \begin{align}
            \label{eqn: Athreya circle integral estimate}
            	\frac{1}{2\pi} \int_0^{2\pi} \a_\eta(g_t r_\theta x) \;d\theta \leq ce^{-(1-\eta) t} \a_\eta(x) + b 
            \end{align}
    \end{lemma}
    
    By using the integral estimate in~\eqref{eqn: Athreya circle integral estimate} in place of the one in~\eqref{eqn: height function circle integral estimate}, we can prove the following analogue of Lemma~\ref{lemma: integral estimate for horocycle Gaussian}.
    
    \begin{corollary} \label{cor: gaussian integral estimate for Athreya function}
    For every $0<\eta<1$, let $\a_\eta\colon~X~\r~\R^+$ be as in Lemma~\ref{lemma: Athreya integral estimate}.
    Let $\s$ be as in Definition~\ref{defn: height functions}.
        Then, there are constants $c' = c'(\eta,\s)$, and $t_0 = t_0(\eta)>0$
        	so that for all $t > t_0$,
        	there exists $b'= b'(t,\eta)>0$ such that for all $x\in X$,
            \begin{equation}
            \label{eqn: Athreya horocycle integral estimate Gaussian}
            	\int_{\R} \a_\eta(g_t h_s x) \;d\rho_1(s) \leq c'e^{-(1-\eta) t} \a_\eta(x) +b'
            \end{equation}
        where $d\rho_1(s) = e^{-s^2}ds$ is a mean $0$, variance $1$ Gaussian.
        In particular, if $\a_\eta(x) > \frac{b' e^{(1-\eta)t}}{c'}$, then
        \begin{equation}
            \label{eqn: Athreya horocycle integral estimate Gaussian for high x}
            	\int_{\R} \a_\eta(g_t h_s x) \;d\rho_1(s) \leq 2 c'e^{-(1-\eta) t} \a_\eta(x) 
            \end{equation}
    \end{corollary}
    
    As a result, we deduce the upper bound on $dim_H(Z)$ from a covering result for the sets
    $Z_\omega(Q,N,t,\d)$, where $Q$ will be a sublevel set of a height function $\a_\eta$ for $\eta\in (0,1)$.
    Moreover, Equation~\eqref{eqn: Athreya horocycle integral estimate Gaussian for high x} is the exact analogue of the integral estimate in~\cite[Corollary $4.2$]{KKLM-SingSystems}.
    Fixing any choice of the parameter $\eta \in (0,1)$, 
    the rest of the proof of~\cite[Theorem 1.1]{KKLM-SingSystems} applies verbatim in our setting to get that $dim_H(Z) \leq 1- \frac{\d(1-\eta)}{2}$.
     By sending $\d$ to $1$ and $\eta$ to $0$, we get the theorem.   
\end{proof}

\section{Uniformity in Birkhoff's Theorem}\label{uniformity_B}

  The purpose of this section is to prove a uniform version of \cite[Theorem 1.1]{ChaikaEskin} due to Chaika and Eskin on the pointwise equidistribution of Teichm\"{u}ller geodesics with respect to the Lebesgue measure on a horocycle arc.
 % This result is in the spirit of the results of~\cite{EMM} and~\cite{DaniMargulis}.
  This step is crucial for our Hausdorff dimension estimates in the large deviation problems.

  Throughout this section, suppose $\mc{M} \subset \mc{H}_1(\a)$ is a fixed $\mathrm{SL}_2(\R)$ invariant affine submanifold.
  For an affine invariant submanifold $\mc{N}\subset \mc{H}_1(\a)$,
  we denote by $\nu_{\mc{N}}$ the unique $\mathrm{SL}_2(\R)$ invariant Lebesgue probability measure supported on $\mc{N}$. For any bounded continuous function $\phi$ on $\mathcal H_1(\alpha)$, let $\nu_{\mathcal N}(\phi) = \int_{\mc{H}_1(\a)}\phi \;d\nu_{\mathcal N}$.
  For any $T>0$, $s\in [-1,1]$ and $x\in \mc{H}_1(\a)$, we denote by $A^T_s(x)$ the measure defined by
  \begin{align} \label{defn: orbit measures}
  A^T_s(x)(\vp) := \frac{1}{T} \int_0^T \vp(g_th_sx)\;dt
  \end{align} 
  for any bounded continuous function $\vp$ on $\mc{M}$.
  Similarly, for $N\in \N$ and $l>0$, we define the measure $S_s^N(x)$ in the following way.
  \begin{align} \label{defn: discrete orbit measures}
  	S_s^N(x)(\vp) := \frac{1}{N} \sum_{n=1}^N \vp(g_{ln}h_s x)
  \end{align}
  Notice that $S_s^N(x)$ depends on the step size $l$, though we do not emphasize this in the notation.
  
  For any $h\in \mathrm{SL}_2(\mathbb R)$, we define $hA_s^T(x)$ and $hS_s^N(x)$ in the following way.
   \begin{align} \label{defn: orbit measures_and_discrete_a}
  hA^T_s(x)(\vp) := \frac{1}{T} \int_0^T \vp(hg_th_sx)\;dt \qquad \text{ and } \qquad 	hS_s^N(x)(\vp) := \frac{1}{N} \sum_{n=1}^N \vp(hg_{ln}h_s x)
  \end{align} 
  
   Let $C_c^{\infty}(\mc{H}_1(\a))$ be a space of smooth compactly supported functions on $\mc{H}_1(\a)$. For any $\vp\in C_c^{\infty}(\mc{H}_1(\a))$, we define a Sobolev norm $\mc{S}(\vp)$ of $\vp$ by 
    \begin{align} \label{eqn: sobolev norm upper bound}
    	\mc{S}(\vp): = \norm{\vp}_{Lip} + \norm{\vp}_\infty,
    \end{align} 
    where $\norm{\phi}_{Lip}$ and $\norm{\vp}_\infty$ denote the Lipschitz constant and the maximum on $\mc{H}_1(\a)$ of $\vp$, respectively.
    Then, for all $g\in \mathrm{SL}_2(\R)$ and all $x\in \mc{H}_1(\a)$, one has
    \begin{align*} 
    	|\vp(gx) - \vp(x)| \leq \mc{S}(\vp) d(g,Id)
    \end{align*}
    where $d(.,.)$ denotes some metric on $\mathrm{SL}_2(\R)$.
  
  The following theorem is the main result of this section.

	\begin{theorem} \label{thrm: quantitative Chaika-Eskin}
    Suppose $f$ is a bounded continuous function on $\mc{H}_1(\a)$. 
        Then, for any $\e>0$ there exist finitely many proper affine $\mathrm{SL}_2(\R)$ invariant submanifolds of $\mc{M}$, denoted by $\mc{N}_1,\dots, \mc{N}_l$ such that
        for any compact set $F \subset \mc{M}\backslash \left(\cup_{i=1}^l \mc{N}_i \right)$
        and any $\kappa >0$, there exists $T_0 = T_0(F,\kappa, \e, f)>0$ such that
        for all $x\in F$ and all $T>T_0$,
        \begin{align}
        	\left| \set{s\in [-1,1]: \left| A^T_s(x)(f) - \nu_{\mc{M}}(f)  \right|  \geqslant \e }   \right| < \kappa
        \end{align}
	\end{theorem}

	The proof of Theorem~\ref{thrm: quantitative Chaika-Eskin} (see Section \ref{proof_thm_5.1}) is based on a combination of the techniques used to prove ~\cite[Theorem 1.1]{ChaikaEskin} and ~\cite[Theorem 2.11]{EMM}, paying additional care to the unipotent invariance of limiting distributions.
    Following the same idea, we also prove the following discrete version of Theorem~\ref{thrm: quantitative Chaika-Eskin} (see Section \ref{proof_thm_5.2} for the proof).

    \begin{theorem} \label{thrm: quantitative discrete Chaika-Eskin}
    Suppose $f$ is a bounded continuous function on $\mc{H}_1(\a)$. 
        Then, for any $\e>0$ there exist finitely many proper affine $\mathrm{SL}_2(\R)$ invariant submanifolds of $\mc{M}$, denoted by $\mc{N}_1,\dots, \mc{N}_k$ such that
        for any compact set $F \subset \mc{M}\backslash \left(\cup_{i=1}^k \mc{N}_i \right)$,
        any $\kappa >0$ and all $l >0$,
        there exists $N_0 = N_0(F,\kappa,l,\e, f)>0$ such that
        for all $x\in F$ and all $N>N_0$,
        \begin{align}
        	\left| \set{s\in [-1,1]: \left| S_s^N(x)(f) - \nu_{\mc{M}}(f)  \right|  \geqslant \e }   \right| < \kappa
        \end{align}
    \end{theorem}
    
    Theorems~\ref{thrm: quantitative Chaika-Eskin} and~\ref{thrm: quantitative discrete Chaika-Eskin} are in the spirit of the results of~\cite{EMM} and~\cite{DaniMargulis}.
    %%%%%%%%%%%%%%%%%%%%%%%%%%%%%%
    
    \subsection{Some Finiteness and Recurrence Results}

	In this section we formulate some facts that we use throughout Section~\ref{uniformity_B}.
    
	The following lemma will provide us with the finite exceptional collection of invariant submanifolds in Theorem~\ref{thrm: quantitative Chaika-Eskin}.
	\begin{lemma} [Lemma 3.4 in~\cite{EMM}] \label{lemma: finite collection}
    	Given $\e>0$ and $\vp \in C_c(\mc{H}_1(\a))$.
		There exists a finite collection $\mc{C}$ of proper affine invariant
submanifolds of $\mc{M}$ with the following property:
	if $\mc{N} \subset \mc{M}$ is an affine invariant submanifold
    such that $|\nu_{\mc{N}}(\vp) - \nu_{\mc{M}}(\vp)|\geq \e$, then
    $\mc{N}$ is contained in some $\mc{N}' \in \mc{C}$.
	\end{lemma}
    
    %%%%%
    
    The following proposition shows that most geodesic trajectories avoid any given finite collection of proper submanifolds of $\mc{M}$.
  %  Its proof is a standard consequence of the so-called contraction hypothesis Proposition 2.13 in~\cite{EMM}. See also~\cite[Lemma 3.1]{EskinMargulis-RandomWalks}.
    
    \begin{proposition} [Proposition 3.8 in~\cite{EMM}] 
    \label{propn: avoidance proposition}
    Given $\e>0$ and any (possibly empty) proper affine invariant submanifold $\mc{N}$, there exists an open neighborhood $\Omega_{\mc{N},\e}$ of $\mc{N}$ with the following property:
    the complement of $\Omega_{\mc{N},\e}$ is compact 
    and for any compact set $F\subset \mc{H}_1(\a)\backslash\mc{N}$,
    there exists $T_0 = T_0(F)>0$, so that for any
    $T>T_0$ and any $x\in F$,
    \begin{align}
    	\int_{-1}^1 A^T_s(x)(\chi_{\Omega_{\mc{N},\e}})  \;ds < \e
    \end{align}
    where $\chi_{\Omega_{\mc{N},\e}}$ denotes the indicator function of the set $\Omega_{\mc{N},\e}$.
    \end{proposition}
    
    %%%%%%%%%%%%%%%
    
    The following discrete version of Proposition~\ref{propn: avoidance proposition} also holds. 
    
    \begin{proposition} \label{propn: discrete avoidance proposition}
    Given $\e>0$ and any (possibly empty) proper affine invariant submanifold $\mc{N}$, there exists an open neighborhood $\Omega_{\mc{N},\e}$ of $\mc{N}$ with the following property:
    the complement of $\Omega_{\mc{N},\e}$ is compact 
    and for any compact set $F\subset \mc{H}_1(\a)\backslash\mc{N}$
    and any $l > 0$,
    there exists $N_0>0$, so that for any
    $N>N_0$ and any $x\in F$,
    \begin{align}
    	\int_{-1}^1 S^N_s(x)(\chi_{\Omega_{\mc{N},\e}})  \;ds < \e
    \end{align}
    where $\chi_{\Omega_{\mc{N},\e}}$ denotes the indicator function of the set $\Omega_{\mc{N},\e}$.
    \end{proposition}

The proof of Proposition~\ref{propn: discrete avoidance proposition} is similar to that of Proposition~\ref{propn: avoidance proposition}, i.e., it is a consequence of the contraction hypothesis (see Definition~\ref{defn: height functions}) shown in ~\cite[Proposition 2.13]{EMM}. See also~\cite[Lemma 3.1]{EskinMargulis-RandomWalks}.

    \begin{proof}
	By Proposition 2.13 in~\cite{EMM}, there exists a height function $f_\mc{N}$ 
    with $X = \mc{H}_1(\a)$ and $Y = \mc{N}$ (see Definition~\ref{defn: height functions}).
	Let $m_F = \sup\set{f_\mc{N}(x): x\in F}$. Notice that $m_F\geq 1$ as, by definition, $f_{\mathcal N}(x)\geq 1$ for any $x\in\mathcal H_1(\alpha)$. 
	Then, by Lemma~\ref{lemma: integral estimate for horocycle interval},
	there exists $t_1>0$ such that for all $t >t_1$ and all $x\in F$,
	\begin{align*}
		\int_{-1}^{1} f_\mc{N}(g_t h_s x)\;ds \leq \frac{1}{m_F}f_\mc{N}(x) + b
        \leq b+1.
	\end{align*}
    Moreover, by Property (2) in Definition~\ref{defn: height functions}, there exists $M=M(t_1)>0$ such that for all $0\leq t \leq t_1$ and all $x$, 
    \begin{equation*}
    	 f_\mc{N}(g_t h_s x) \leq M f_\mc{N}(x).
    \end{equation*}

    Let $L>0$ be such that $\frac{b+2}{L} < \e$.
    Define a set $\Omega_{\mc{N},\e}$ in the following way.
    \begin{equation*}
    	\Omega_{\mc{N},\e} = \set{x\in\mc{H}_1(\a): f_\mc{N}(x)> L}^o,
    \end{equation*}
    where $\{\cdot\}^o$ denotes the interior of a set.
    Then, by Property $(4)$ in Definition~\ref{defn: height functions}, $\Omega_{\mc{N},\e}$ is an open neighborhood of $\mc{N}$ with compact complement.    
    Let $N_0\in \N$ be sufficiently large so that
    \[ \frac{ M m_Ft_1}{lN_0} < 1. \]
    Then, using the above estimates, we get
    \begin{align}\label{ineq_for_f_N}
    	\int_{-1}^1 S_s^N(x)(f_\mc{N})\;ds &=
        \frac{1}{N} \sum_{1\leq n \leq t_1/l} \int_{-1}^1 f_\mc{N}(g_{ln}h_sx)\;ds
        +\frac{1}{N} \sum_{t_1/l<  n\leq N} \int_{-1}^1 f_\mc{N}(g_{ln}h_sx)\;ds\nonumber\\
        &\leq \frac{ M m_Ft_1}{lN} + b+1 \leq b+2.
    \end{align}
    Notice that for any $n\in\mathbb N$ and $s\in[-1,1]$, we have 
    \begin{equation}\label{ineq_chi}
    f_{\mathcal N}(g_{ln}h_sx)\geq L\chi_{\Omega_{\mc{N},\e}}(g_{ln}h_sx).
    \end{equation}
    Therefore, by \eqref{ineq_for_f_N} and \eqref{ineq_chi}, we obtain
    $$\int_{-1}^1 S_s^N(x)(\chi_{\Omega_{\mc{N},\e}})\;ds \leq \frac{b+2}{L}<\e.$$
%    The claim now follows by Chebyshev's inequality.
    \end{proof}

    \subsection{Effective Unipotent Invariance}
\label{section: effective U invariance}

	In this section, we show a quantative version of \linebreak ~\cite[Proposition 3.1]{ChaikaEskin} (Proposition~\ref{propn: effective U invariance}) regarding almost sure unipotent invariance of limit points of measures of the form~\eqref{defn: orbit measures}.
    Also, we state an analogue of it for discrete averages (Proposition~\ref{propn: discrete effective U invariance}), whose proof is identical to the flow case.
  See~\cite{Khalil} for a generalization of this phenomenon to semisimple Lie group actions.
  
%    The following is proved in Section 3 of~\cite{ChaikaEskin}, albeit for circle arcs. The proof for horocycle arcs is simpler, since the group of elements $h_s$ is normalized by $g_t$. 
Suppose $x\in \mc{H}_1(\a)$, $\phi\in C_c^\infty(\mc{H}_1(\a))$ and $\beta\in\mathbb R$. For $t>0$ and $s\in [-1,1]$, we define
    \begin{equation} \label{defn f_t(s)}
    	f_t(s) = \vp(g_th_sx) - \vp(h_{\b}g_th_s x).
    \end{equation}

The following lemma formulated for horocycle arcs is an analogue of Lemma 3.3 in \cite{ChaikaEskin} which is proved for circle arcs. 

 \begin{lemma} [Analogue of Lemma 3.3 in~\cite{ChaikaEskin}]
    \label{lemma: exponential decay of correlations}
    There exists a constant $C_1>0$ such that for all $x\in \mc{H}_1(\a)$, $\phi\in C_c^\infty(\mc{H}_1(\a))$, $\beta\in\mathbb R$ and all $t_1,t_2>0$,
    \begin{equation*}
    	\left|  \int_{-1}^1 f_{t_1}(s)f_{t_2}(s)\;ds \right| 
        	\leq C_1 \mc{S}(\vp)^2 e^{-2|t_1 - t_2|},
    \end{equation*}
where $\mathcal S(\phi)$ and $f_t(s)$ are defined in \eqref{eqn: sobolev norm upper bound} and \eqref{defn f_t(s)}, respectively.
    \end{lemma}
    
We note that the proof of Lemma~\ref{lemma: exponential decay of correlations} is identical to the proof of ~\cite[Lemma 3.3]{ChaikaEskin} and simpler if one takes into account that the group of elements $h_s$ is normalized by $g_t$.    

    \begin{proposition} [Quantative version of Proposition 3.1 in~\cite{ChaikaEskin}]
    \label{propn: effective U invariance}
    	Suppose $\b\in \R$.
        Then, there exists a constant $C>0$,
        such that for all $T>0$, all $x\in \mc{H}_1(\a)$ 
        and all $\vp\in C_c^\infty(\mc{H}_1(\a))$,
        the Lebesgue measure of the set
        \begin{align*}
        	\set{s\in [-1,1]: \left| A_s^T(x)(\vp) - (h_{\b}A_s^T(x))(\vp) \right| > \frac{\mc{S}(\vp)}{T^{1/8}}}
        \end{align*}
        is at most $C/T^{1/4}$.
    \end{proposition}
    
     The version of Proposition~\ref{propn: effective U invariance} for discrete averages is the following.
    \begin{proposition}
    \label{propn: discrete effective U invariance}
    	Suppose $\b\in\R$.
        Then, there exists a constant $C>0$,
        such that for all $N>0$, all $x\in \mc{H}_1(\a)$ 
        and all $\vp\in C_c^\infty(\mc{H}_1(\a))$,
        the Lebesgue measure of the set
        \begin{align*}
        	\set{s\in [-1,1]: \left| S_s^N(x)(\vp) - (h_{\b}S_s^N(x))(\vp) \right| > \frac{\mc{S}(\vp)}{N^{1/8}}}
        \end{align*}
        is at most $C/N^{1/4}$.
    \end{proposition}

    \begin{proof} [Proof of Proposition~\ref{propn: effective U invariance}]
    
    	By Fubini's theorem and Lemma~\ref{lemma: exponential decay of correlations}, one has
        \begin{align*}
        	\int_{-1}^1 \left| A_s^T(\vp) - (h_{\b}A_s^T)(\vp) \right|^2 \;ds
            	&\leq \frac{1}{T^2} \int_{[0,T]^2} 
                \left|  \int_{-1}^1 f_{t_1}(s)f_{t_2}(s)\;ds \right| \;dt_1\;dt_2 \\
                &= \frac{1}{T^2} \int_{|t_1-t_2| < T^{1/2}} 
                \left|  \int_{-1}^1 f_{t_1}(s)f_{t_2}(s)\;ds \right| \;dt_1\;dt_2 
                + C_1 \mc{S}(\vp)^2 e^{-2T^{1/2}} \\
                &\leq \frac{ 16 ||\vp||_\infty^2 }{T^{1/2}} + C_1 \mc{S}(\vp)^2 e^{-2T^{1/2}}\\
                &\leq 2C_2 \mc{S}(\vp)^2 T^{-1/2}
        \end{align*}
    where we used the facts that $|f_t(s)| \leq 2||\vp||_\infty$, the measure of the region $|t_1 - t_2|<T^{1/2}$ is at most $2T^{3/2}$,
    and $C_2 > 16C_1$ is a constant such that for all $T>0$, one has
    \[ e^{-2T^{1/2}} \leq C_2 T^{-1/2}. \]
    Using the Chebyshev's inequality, we obtain the proposition.
     \end{proof}

    \subsection{Proof of Theorem~\ref{thrm: quantitative Chaika-Eskin}}\label{proof_thm_5.1}

	Fix positive constants $\e$ and $\kappa$.
    Let $\mc{C}$ be the finite collection of affine invariant submanifolds
    $\mc{N}_1,\dots \mc{N}_k$ of $\mc{M}$ given by Lemma~\ref{lemma: finite collection} applied to the given function $f$ and $\e/2$.
    Consider a compact subset $F\subset \mc{M}\backslash\cup_i \mc{N}_i$.
     
     Let $\e'>0$ be a sufficiently small number such that $\sqrt{\e'}<\min\left\{\frac{\kappa}{3},\frac{\e}{9\|f\|_{\infty}}\right\}$.
     By Proposition~\ref{propn: avoidance proposition}, since $\mc{C}$ is a finite collection, there exists an open neighborhood $\Omega_{\mc{C},\e'}$  of $\cup_{i} \mc{N}_i$ and $T_0>0$ depending on $\e'$ and $F$ such that
     for all $T>T_0$ and all $x\in F$, we have
     \[ \int_{-1}^1 A^T_s(x)\left(\chi_{\Omega_{\mc{C},\e'}}\right)  \;ds < \e' \]
     and hence, by Chebyshev's inequality, we get that the measure of the set
     \begin{align} \label{defn: sets D(x,T)}
     	D(x,T,\e') = \set{s\in [-1,1] : A^T_s(x)
        	\left(\chi_{\Omega_{\mc{C},\e'}}\right) \geq \sqrt{\e'}}
     \end{align}
     is at most $\sqrt{\e'}$.
     
     Let $\Phi = \set{\vp_n:n\in \N} \subset C_c^\infty(\mc{H}^1(\a))$ be a countable dense collection of functions.
     For each $N$, let $\Phi_N = \set{\vp_n: 1\leq n\leq N, n\in \mathbb N} \subset \Phi$.
     Consider $\b_1,\b_2 \in \R$ with $\b_1/\b_2 \notin \Q$.
     By Proposition~\ref{propn: effective U invariance}, there exists a constant $C>0$ such that for all $T>0$ and all $x$, the measure of the sets
     \begin{equation} \label{defn: sets B(x,T)}
   		B(x,T, \Phi_N) := \bigcup_{n=1}^N B(x,T,\vp_n)
     \end{equation}
     is at most $CN T^{-1/4}$, where $\mathcal S(\cdot)$ is a Sobolev norm (see \eqref{eqn: sobolev norm upper bound}) and
     \begin{align*}
        	B(x,T,\vp_n) := \set{s\in [-1,1]: \left| A_s^T(x)(\vp_n) - (h_{\b_i}A_s^T(x))(\vp_n) \right| > \frac{\mc{S}(\vp_n)}{T^{1/8}} \text{ for } i=1,2}.
        \end{align*}
     
    We prove the theorem by contradiction. Suppose that the conclusion of the theorem does not hold for our choice of $F$ and $\kappa$.
    Then, there exists a sequence $x_n \in F$ and $T_n \r \infty$ such that
    for each $n\in \N$, the measure of the set
    \begin{align} \label{defn: sets Z_x}
    	Z_{x_n}(f,T_n) := \set{s\in[-1,1]: \left| A_s^{T_n}(x)(f)-\nu_{\mc{M}}(f) \right| \geq \e }
    \end{align}
     has measure at least $\kappa$.
     
     By our estimates on the measures of the sets in~\eqref{defn: sets D(x,T)},~\eqref{defn: sets B(x,T)} and~\eqref{defn: sets Z_x}, 
     and the choice of $\e'$ such that $\sqrt{\e'} < \kappa/3$,
     then the following holds. For all $n$ sufficiently large so that $CT_n^{-1/8} < \kappa/3$, we have
     \begin{equation} \label{eqn: non-empty intersection}
     	Z_{x_n}(f,T_n) \cap D(x_n, T_n, \e')^c \cap B(x_n, T_n, \Phi_{T_n^{1/8}})^c
        \neq \emptyset
     \end{equation}
     where for a set $A\subset [-1,1]$, we use $A^c$ to denote its complement.
     Therefore, for all $n$ sufficiently large we can choose a point $s_n$ that belongs to the intersection in \eqref{eqn: non-empty intersection}.
     Since the space of Borel measures on $\mc{H}_1(\a)$ of mass at most $1$ is compact in the weak-$\ast$ topology, after passing to a subsequence if necessary, we may assume that there is a Borel measure $\nu$ such that
     \[ A_{s_n}^{T_n}(x_n) \xrightarrow{weak-\ast} \nu \]
     Note that a priori $\nu$ may be the $0$ measure. We show that this is not the case.
     
     We claim that $\nu$ is $\mathrm{SL}_2(\R)$ invariant.
     By Theorem 1.4 due to Eskin and Mirzakhani~\cite{EskinMirzakhani}, it is sufficient to show that
     $\nu$ is invariant by $P$, the subgroup of upper triangular matrices.
     Clearly, $\nu$ is invariant by $g_t$ for all $t$.
     Moreover, by the dominated convergence theorem, it suffices to show that $\nu$ is invariant by $h_{\b_1}$ and $h_{\b_2}$ as they generate a dense subgroup of $U = \{h_s: s\in\mathbb R\}$.
     
     Since smooth functions are dense in the set of compactly supported continuous functions, it suffices to show that for $i=1,2$, $h_{\b_i}\nu(\vp_k) = \nu(\vp_k)$, where $h_{\b_i}\nu(\vp_k):=\int_{\mathcal H_1(\alpha)} \phi_k(h_{\b_i}\omega)d\nu(\omega)$ for all $\vp_k \in \Phi$, our countable dense collection of smooth compactly supported functions.
     
     Fix some $\vp_k \in \Phi$.
     Note that for all $n$ sufficiently large, we have that $\vp_k \in \Phi_{T_n^{1/8}}$ and, therefore, $s_n \notin B(x_n, T_n, \vp_k)$.
    As a result, we have
     \[ \left| A_{s_n}^{T_n}(x_n)(\vp_k) - h_{\b_i} A_{s_n}^{T_n}(x_n)(\vp_k)\right| \leq \frac{\mc{S}(\vp_k)}{T_n} \xrightarrow{n\r\infty} 0.  \]
     Therefore, $\nu$ is $\mathrm{SL}_2(\R)$ invariant.
     
     Moreover, since $s_n \in Z_{x_n}(f,T_n)$ for all $n$, we obtain that
     \[ |\nu(f) - \nu_{\mc{M}}(f)| \geq \e. \]
     We show that this is not possible.
     By Proposition 2.16 in~\cite{EMM}, there are countably many affine invariant submanifolds in $\mc{H}_1(\a)$. Thus, since $\nu$ is $\mathrm{SL}_2(\R)$ invariant, it has a countable ergodic decomposition of the form
     \[  \nu = \sum_{\mc{N}\subseteq \mc{M}} a_{\mc{N}}\nu_{\mathcal N}, \]
     where the sum is taken over all such proper (possibly empty) affine invariant submanifolds and $a_{\mc{N}}\in [0,1]$ for all $\mc{N}$.     
     Note that since $s_n \notin D(x_n, T_n, \e')$, we have
     \begin{align*}
     	\sum_{\substack{\mc{N}: \exists \mc{N}'\in \mc{C},\\ \mc{N}\subseteq \mc{N}'} } a_{\mc{N}} \leq 
     	\nu(\Omega_{\mc{C},\e'}) \leq \sqrt{\e'}
     \end{align*}
     Since the complement of $\Omega_{\mc{C},\e'}$ is compact and $A_{s_n}^{T_n}(x_n)(1-\chi_{\Omega_{\mc{C},\e'}})\geq 1-\sqrt{\e'}$, the total mass of $\nu$ is at least $1-\sqrt{\e'}$.

     Furthermore, we have that $|\nu_{\mc{N}}(f)-\nu_\mc{M}(f)|<\e/2$, for all $ \mc{N}$ not contained in any member of $\mc{C}$, by definition of the collection $\mc{C}$.
     
     Let $|\nu|: = \sum_{\mc{N}} a_{\mc{N}}$ be the total mass of $\nu$.
     Then, we have that
     \begin{align*}
      \e \leq |\nu(f) - \nu_{\mc{M}}(f)|  
      &= \left| (1-|\nu|)\nu_\mc{M}(f) + 
      \sum_{\mc{N}\subseteq \mc{M}} a_{\mc{N}}(\nu_{N}(f) - \nu_\mc{M}(f))  \right| \\
      &\leq \sqrt{\e'} ||f||_\infty +
      	\sum_{\substack{\mc{N}: \exists \mc{N}'\in \mc{C},\\ \mc{N}\subseteq \mc{N}'} } a_{\mc{N}} |\nu_{\mc{N}}(f) - \nu_{\mc{M}}(f)| +
        \sum_{\substack{\mc{N}: \nexists \mc{N}'\in \mc{C},\\ \mc{N}\subseteq \mc{N}'} } a_{\mc{N}} |\nu_{\mc{N}}(f) - \nu_{\mc{M}}(f)| \\
        &\leq \sqrt{\e'} ||f||_\infty + 2||f||_\infty \nu(\Omega_{\mc{C},\e'}) + |\nu|\e/2  \\
        &\leq 3||f||_\infty \sqrt{\e'} + \e/2<\e.
     \end{align*}
     We get the desired contradiction by our choice of $\e'$.

    \subsection{Proof of Theorem~\ref{thrm: quantitative discrete Chaika-Eskin}}\label{proof_thm_5.2}
The proof is similar to the proof of Theorem~\ref{thrm: quantitative Chaika-Eskin} (the flow case), and relies on using Propositions~\ref{propn: discrete avoidance proposition} and~\ref{propn: discrete effective U invariance} instead of Propositions~\ref{propn: avoidance proposition} and ~\ref{propn: effective U invariance}, respectively.
The proof also goes by contradiction. Assuming that the conclusion of the theorem does not hold, we construct a $\mathrm{SL}_2(\R)$ invariant measure $\nu$. The analysis of its ergodic decomposition implies a contradiction as in Section~\ref{proof_thm_5.1}.

The following lemma allows us to show that the constructed measure $\nu$ is $\mathrm{SL}_2(\R)$ invariant.
	
  \begin{lemma} \label{lemma: classification of discrete P invariant measures}
	Let $l >0 $ and let $P_l$ be the group generated by elements of the form
    $g_{ln}h_s$ for $n\in \Z$ and $s\in \R$.
    Then, $\nu$ is $\mathrm{SL}_2(\R)$ invariant if $\nu$ is a $P_l$ ergodic invariant probability measure on $\mc{M}$.
  \end{lemma}
  
  \begin{proof}
	Denote by $\bar{\nu}$ the measure defined by
    \begin{align} \label{eqn: ergodic decomposition 1}
    	\bar{\nu} := \frac{1}{l} \int_0^l (g_t)_\ast \nu \;dt,
    \end{align}
    where $(g_t)_\ast \nu$ is the pushforward of $\nu$.
    
    Then, $\bar{\nu}$ is invariant by the group of upper triangular matrices $P$. 
    Notice that for any $t\in(0,l)$ we have $(g_t)_\ast \nu$ is invariant by the group $U=\{h_s: s\in\mathbb R\}$ due to the fact that $U$ is normalized by $g_t$ and $\nu$ is invariant by $U$. That implies that $\bar \nu$ is invariant by $U$ as it is a convex combination of $U$ invariant measures. Similarly, we can show that $\bar \nu$ is invariant under $\mathbb Z$ action of $g_l$. To show the invariance under the group $A = \{g_t: t\in\mathbb R\}$, we write $t=ml+r$ for some $m\in\mathbb Z$ and $r\in[0,1)$ and use the invariance by $\{g_{nl}:n\in\mathbb Z\}$. 
    
    As a result, by ~\cite[Theorem 1.4]{EskinMirzakhani}, $\bar\nu$ is $\mathrm{SL}_2(\R)$-invariant.
    Thus, $\bar{\nu}$ has the following ergodic decomposition with respect to the $\mathrm{SL}_2(\R)$ action:
    \begin{align} \label{eqn: ergodic decomposition 2}
    	\bar{\nu} = \sum_{\mc{N}\subseteq \mc{M}} a_\mc{N}\nu_\mc{N},
    \end{align}
    where each $\nu_\mc{N}$ is ergodic under the $\mathrm{SL}_2(\R)$ action.
    But, by Mautner's phenomenon, each $\nu_\mc{N}$ is ergodic under the action of $h_s$ for all $s\neq 0$.
    
    On the other hand, $(g_t)_\ast\nu$ is $h_s$-invariant for all $t$ and $s$.
    Hence, equations~\eqref{eqn: ergodic decomposition 1} and~\eqref{eqn: ergodic decomposition 2} give two decompositions of $\bar{\nu}$ for the action of $h_s$, one of which is a countable decomposition into ergodic measures.
    
    Thus, by uniqueness of the ergodic decomposition, 
    there exists a set $A \subseteq [0,l]$ of positive Lebesgue measure $|A|$
    and an affine invariant manifold $\mc{N}$ so that $a_\mc{N} = |A|/l$ and
    \[ \frac{1}{l} \int_A (g_t)_\ast \nu \;dt = a_\mc{N} \nu_\mc{N}. \]
    But, by ergodicity of $\nu_\mc{N}$ under the action of $h_s$, we have that
    $(g_t)_\ast \nu = \nu_\mc{N}$ for almost every $t\in A$.
    Since $\nu_\mc{N}$ is $\mathrm{SL}_2(\R)$ invariant, then so is $\nu$.
  \end{proof}

\section{Dimension of Directions with Large Deviations in Birkhoff's Theorem}
\label{section: birkhoff}

	The goal of this section is to prove Theorem~\ref{thrm: Birkhoff deviations thrm for horocycles}.
    We also outline the modifications on the proof needed to prove Theorem~\ref{thrm: discrete Birkhoff deviations thrm for horocycles} in Section~\ref{Birkhoff_discrete}.
    
    In what follows, $\mc{M} \subseteq \mathcal H_1(\alpha)$ is a fixed affine invariant manifold.
    By a simple approximation argument, it is enough to prove Theorem~\ref{thrm: Birkhoff deviations thrm for horocycles} when $f$ is a Lipschitz function.
    We let $\mathcal{S}(f)$ denote the Sobolev norm(see \eqref{eqn: sobolev norm upper bound}), and $\nu_{\mathcal M}(f) = \int\limits_{M}f\;d\nu_{\mathcal M}$. 
    
    Throughout this section we use the following notation.    
    For any positive $\e, N \in \mathbb R$, $M\in \N$ and a subset $Q \subseteq \mc{M}$, we define the following sets.
 
   \begin{align} \label{defn: B(f,N,epsilon,M)}
    	&B_\omega (f, N, \e, M): = \set{s\in [-1,1]: \frac{1}{MN}\int_0^{MN}f(g_th_s \omega) \;dt > \nu_{\mathcal M}(f) + \e  }\\
    	&Z_\omega(Q, M, N, \e) := \set{ s\in[-1,1]: 
        \frac{\#\set{0\leq i\leq M-1: g_{iN} h_s \omega \notin Q} }{M} >\e }    \nonumber
   \\
   &B_\omega (f, N, \e) := \limsup\limits_{M\r\infty} B_\omega (f,N,\e,M) \qquad
    Z_\omega(Q, N,\e) := \limsup_{M\r\infty}  Z_\omega(Q, N,\e,M)  \nonumber
     \end{align}
  It is straightforward to check that $B_\omega (f,N,\e)$ is equal to the exceptional set considered in Theorem~\ref{thrm: Birkhoff deviations thrm for horocycles}.
  The sets $Z_\omega(Q, M, N, \e)$ are the same as the ones defined in~\eqref{defn: set of non-recurrent directions}.
    
	Next, for any $s\in[-1,1]$, $i\in \N$ and positive $\beta, N\in\mathbb R$, we define the corresponding functions and sets:
     \begin{equation} \label{defn: functions f_i}
       f_i(s) := \frac{1}{N}\int_{iN}^{(i+1)N} f(g_th_s \omega) \;dt
       \end{equation}

       \begin{equation} \label{defn: the sets F_i}
       		F_i(\b) = \set{s: f_i(s) >  \nu_{\mc{M}}(f)+\b}
       \end{equation}
		Here, we drop the dependence on the basepoint $\omega$ from the notation for simplicity.
    %%%%%%%%%%%%%%%%%%%
    
  \subsection*{Strategy}
       The strategy for proving Theorem~\ref{thrm: Birkhoff deviations thrm for horocycles} consists of two steps.
       The first step is to use Theorem~\ref{thrm: quantitative discrete Chaika-Eskin} to control the measure of the sets $F_i(\b)$. 
       This is carried out in Lemma~\ref{lemma: bound on measure of recurrent part of F_i}.
       
       The next step is to show that the sets $F_i(\e/2)$ behave like level sets of independent random variables (Proposition~\ref{propn: independence lemma}).
       This will allow us to bound the measure of finite intersections of these sets.
       The proof of this independence property also yields a mechanism for controlling the number of intervals needed to cover such finite intersection using its measure (Lemma~\ref{lemma: birkhoff covering lemma}).
       
      In order to apply Theorem~\ref{thrm: quantitative discrete Chaika-Eskin}, we need to insure that our trajectories land in a pre-chosen compact set.
       Hence, we are forced to run the above argument but restricted to the "recurrent directions".
       This restriction to recurrent directions is shown in Lemma~\ref{lemma: bad set for f and recurrent part of F_i}.
       Applying Theorem~\ref{thrm: H.dim of non-recurrent directions,fixed compact set}, we control the Hausdorff dimension of the non-recurrent directions.

    %%%%%%%%%%%%%%%%%%%

    \subsection{Sets and Partitions}
\label{section: Birkhoff sets partitions}
       
 For $N>0$ and $i\in \N$, let $\mc{P}_i$ denote the partition of $[-1,1]$ into intervals of radius $ e^{-2iN}$.
       For a set $Q \subset \mc{H}_1(\a)$, define the following sub-partitions
       \[ \mc{R}_i(Q) = \set{ J\in \mc{P}_i: \exists s\in J, g_{iN}h_s \omega \in Q }  \]
       Let $\mc{D}_i(Q) = \mc{P}_i \setminus \mc{R}_i(Q)$.
       Here $\mc{R}$ signifies recurrence and $\mc{D}$ signifies divergence.
       We note that the definition of $\mc{R}_i$ depends on the basepoint $\w$ but we suppress this dependence in our notation.
       
       \begin{lemma}
       	\label{lemma: bad set for f and recurrent part of F_i}
        Suppose $Q \subset \mc{H}_1(\a)$. Then, for any $\omega \in \mc{H}_1(\a)$, $N,\e >0$, $0<\d \leq \frac{\e}{4 \mc{S}(f)}$, and $M\in \N$, 
        we have
            \[
            	B_\omega (f,N,\e,M) \subseteq Z_\omega(Q, M,N,\d) \cup 
                	\bigcup_{ \substack{ A \subseteq \set{0,\dots,M-1} \\
                				|A| = \lceil \d M \rceil } }
                                \bigcap_{i\in A}F_i^{\mc{R}}(\e/2), 
            \]
         where $\mc{R}_i := \mc{R}_i(Q)$ for all $i\in \N$ and 
         \[
                        F_i^{\mc{R}}(\e/2) = F_i(\e/2) \cap \bigcup_{J\in \mc{R}_i} J. 
         \]
       \end{lemma}
       
       \begin{proof}
        First, we notice that
         \begin{equation}\label{lemma: bad set for f contained in limsup of F_i}
                  B_\omega (f,N,\e,M) \subseteq \bigcup_{ \substack{ A \subseteq \set{0,\dots,M-1} \\
                                  |A| > 2\d M  } }
                                  \bigcap_{i\in A} F_i(\e/2).
           \end{equation}
         It holds by the following inequalities.
         \begin{align*}
                  \frac{1}{MN}\int_0^{MN} f(g_t h_s \omega) \;dt &=
                      \frac{1}{M} \sum_{i=0}^{M-1} f_i(s) = \frac{1}{M}  
                      \sum_{\substack{  0\leq i\leq M-1 \\ f_i(s) \leq \nu_\mc{M}(f)+\e/2 } } f_i(s)
                      + \frac{1}{M}  \sum_{\substack{  0\leq i\leq M-1 \\ f_i(s) > \nu_\mc{M}(f)+\e/2 } } f_i(s) \\
                      &\leq \nu_\mc{M}(f)+\e/2 + \frac{||f||_\infty}{M} \#\set{i: f_i(s) > \nu_\mc{M}(f)+\e/2}\\
                      &\leq \nu_\mc{M}(f)+\e/2 + \frac{\mc{S}(f)}{M} \#\set{i: s \in F_i(\e/2) }
              \end{align*}

              Thus, if $s \in B_\omega (f,N,\e,M)$, then we must have that
              \begin{align*}
                  \#\set{i: s \in F_i(\e/2)} > \frac{\e}{2\mc{S}(f)} M \geq 2\d M
              \end{align*}  
              by our choice of $\d$. 

              By ~\eqref{lemma: bad set for f contained in limsup of F_i}, it suffices to show the following to prove the lemma.
              \begin{align*}
                  Z_\omega(Q, M,N,\d)^c \cap 
                      \bigcup_{ \substack{ A \subseteq \set{0,\dots,M-1} \\
                                  |A| > 2\d M } }
                                  \bigcap_{i\in A} F_i (\e/2)
                                  \subseteq 
                                  \bigcup_{ \substack{ A \subseteq \set{0,\dots,M-1} \\
                                  |A|= \lceil \d M \rceil } }
                                  \bigcap_{i\in A} F_i^\mc{R} (\e/2),
              \end{align*}
              where for a set $E\subseteq [-1,1]$, $E^c$ denotes its complement.

              The set $Z_\omega(Q, M,N,\d)$ was defined to be the set of directions $s$
              such that $g_{iN} h_s \omega \notin Q$ for at least $\d M$ natural numbers $i <M$.
              Hence, we get that
              \begin{align*}
                  Z_\omega(Q, M,N,\d)^c \subseteq \bigcup_{ \substack{ B \subseteq \set{0,\dots,M-1} \\
                                  |B| > (1-\d) M } }
                                  \bigcap_{j\in B} \bigcup_{J\in \mc{R}_j} J.
              \end{align*}
              Indeed, the right hand side describes the set of directions $s$ which belong to
              $\bigcup_{J\in \mc{R}_j} J$ for at least $(1-\d)M$ natural numbers $j<M$.
              By definition of $\mc{R}_j$, this certainly contains the set of directions
              $s$ for which $g_{jN}h_s\omega \in Q$ for at least
              $(1-\d)M$ natural numbers $j<M$, that is the set on the left hand side.

              Notice that the following inclusions hold.
              \begin{align*}
                  \left[\bigcup_{ \substack{ A \subseteq \set{0,\dots,M-1} \\
                                  |A| > 2\d M } }
                                  \bigcap_{i\in A} F_i (\e/2)
                                  \right]
                                  \bigcap 
                                  &\left[ \bigcup_{ \substack{ B \subseteq \set{0,\dots,M-1} \\
                                  |B| > (1-\d) M } }
                                  \bigcap_{j\in B} \bigcup_{J\in \mc{R}_j} J \right]\\
                                  &\subseteq
                                  \bigcup_{ \substack{ A,B \subseteq \set{0,\dots,M-1} \\
                                  |A| > 2\d M \\ |B| > (1-\d)M} }
                                  \left[\bigcap_{i\in A} F_i (\e/2) \bigcap
                                  \bigcap_{j\in B} \bigcup_{J\in \mc{R}_j} J \right]\\
                                  &\subseteq
                                  \bigcup_{ \substack{ A,B \subseteq \set{0,\dots,M-1} \\
                                  |A| > 2\d M \\ |B| > (1-\d)M} }
                                  \bigcap_{i\in A \cap B} F_i^\mc{R} (\e/2)\\
                                  &\subseteq
                                  \bigcup_{ \substack{ A \subseteq \set{0,\dots,M-1} \\
                                  |A|> \d M  } }
                                  \bigcap_{i\in A} F_i^\mc{R} (\e/2) \\
              \end{align*}
            where for the last inclusion we used the fact that for two sets $A,B  \subseteq \set{0,\dots,M-1}$ with
            $|A| > 2\d M$ and $|B| > (1-\d)M$, we have that $|A\cap B| > \d M$.
            Moreover, notice that
           \[
                \bigcup_{ \substack{ A \subseteq \set{0,\dots,M-1} \\
                                    |A| > \d M } }
                                    \bigcap_{i\in A} F_i^\mc{R} (\e/2)
                \subseteq \bigcup_{ \substack{ A \subseteq \set{0,\dots,M-1} \\
                                    |A| = \lceil \d M \rceil } }
                                    \bigcap_{i\in A} F_i^{\mc{R}} (\e/2).
           \]
            This completes the proof.
       \end{proof}

    \subsection{Measure Bounds for $F_i$}\label{subsection: measure bounds}
       The next lemma allows us to control the measure of the proportion of a set $F_i$ in an element of the partition $\mc{R}_i(Q)$ for a suitably chosen large compact set with good properties.
       This will be a direct application of Theorem~\ref{thrm: quantitative Chaika-Eskin}.
      
       Let $\mc{N}_1,\dots,\mc{N}_k$ be proper affine invariant submanifolds as in Theorem~\ref{thrm: quantitative Chaika-Eskin} applied to $\e$ and $f$. By   ~\cite[Proposition 2.13]{EMM}, for any $i=1,\dots,k$ there exist height functions $f_{\mc{N}_i}$ such that for all $\ell>0$, the sets
       \begin{align*} 
             	C_\ell = \overline{\set{x\in \mc{H}_1(\a): \sum_1^k f_{\mc{N}_i}(x) \leq \ell}}
       \end{align*}
        are compact.      
        The following is the main result of this section which is the form we will use Theorem~\ref{thrm: quantitative Chaika-Eskin} in. Recall the definition of the sets $F_i(\b)$ in~\eqref{defn: the sets F_i}.
       \begin{lemma}
       \label{lemma: bound on measure of recurrent part of F_i}
       For all $\ell >0$ and all $a >0$, there exists $T_0 >0$ such that 
       for all $N>T_0$, $\b>\e$, $i\in \N$, all $\omega \in \mc{H}_1(\a)$
       and all $J \in \mc{R}_i(C_\ell)$, we have
        \[ \frac{\nu (J\cap F_i(\b) )  }{\nu(J)} \leqslant a,  \]
        where $\nu$ is the Lebesgue probability measure on $[-1,1]$.
       \end{lemma}
       
       \begin{proof}
      Denote by $B_1$ a neighborhood of radius $1$ around identity in $SL_2(\R)$. Fix $\ell>0$ and $a>0$.
       Let $\ell' > \ell$ be such that
             \[ B_1 C_{\ell} \subseteq C_{\ell'}. \]
              By Theorem~\ref{thrm: quantitative Chaika-Eskin} applied to $f$, $\e$, $a$ and the compact set $F = C_{\ell'} \subset \mc{M} \setminus \cup_{i=1}^k \mc{N}_i$, where $\mc{N}_i$ are given by that theorem, there exists $T_0$ such that for all $N > T_0$ and $x\in F$, we have
          \begin{align}
             \label{eqn: Birkhoff uniform bound on measure of bad set for f}
              \left| \set{s\in [-1,1]: \left| \frac{1}{N}\int_0^N f(g_t h_s x) \;dt - \nu_{\mc{M}}(f)  \right|  \geq \e }   \right| < a.
           \end{align}
              
       For any $i\in \N$, we define $\mc{R}_i:=\mc{R}_i(C_\ell)$. Fix $J \in\mc{R}_i$.
       Let $s_0 \in J$ be such that $g_{iN}h_{s_0}\omega \in C_\ell$.
       By our choice of $\ell'$, we have the following holds for any $s\in J$.
       \begin{align*}
       	g_{iN}h_s\omega = h_{e^{2iN}(s-s_0)}  g_{iN} h_{s_0} \omega \in B_1 C_\ell
        \subseteq C_{\ell'}
       \end{align*}
      In particular, the above holds for the center $c_0$ of the interval $J$.
       Let $s \in J -c_0$ be such that $s+c_0 \in F_i(\b)$.
       Then, we get that
       \begin{align*}
       	 \nu_M(f) + \b < \frac{1}{N} \int_0^N f(g_{t+iN} h_{s+c_0}\omega)\;dt = \frac{1}{N} \int_0^N f(g_{t} h_{e^{2iN}s}g_{iN}h_{c_0}\omega)\;dt
       \end{align*}
       Thus, we obtain the following.
       \begin{align*}
       	e^{2iN} \left(\left( J\cap F_i(\b) \right) -c_0\right) \subseteq 
        	\set{s \in [-1,1]: \left| \frac{1}{N} \int_0^N f(g_{t} 
            h_{s}g_{iN}h_{c_0}\omega)\;dt - \nu_\mc{M}(f)\right| \geq \b }
       \end{align*}
       Since $g_{iN}h_{c_0}\omega \in C_{\ell'}$, the Lemma follows from~\eqref{eqn: Birkhoff uniform bound on measure of bad set for f}.
       \end{proof}

       The following corollary is an immediate consequence of Lemma~\ref{lemma: bound on measure of recurrent part of F_i} and the fact that elements of $\mc{R}_i$ are disjoint.
       \begin{corollary}
       		For all $\ell >0$ and all $a >0$, there exists $T_0 >0$ such that 
       for all $N>T_0$, $\b>\e$, $i\in \N$, we have that
            \[ \nu \left( F_i(\b) \cap \bigcup_{J\in \mc{R}_i(C_\ell)} J \right) \leq a. \]
       \end{corollary}
     %  \fi
    
    \subsection{Independence of the Sets $F_i$}
       The goal of this section is to prove that the sets $F_i (\b)$ behave as if they are independent.
       More precisely, we will prove that the measure of the intersection of such sets is bounded above by the product of their measures, up to controlled error.
       Recall the definition of partitions $\mc{P}_i$ in Section~\ref{section: Birkhoff sets partitions}.
       
       We start with the following simple but key observation.
       \begin{lemma}
       \label{lemma: 0-1 law for partitions and F_i}
       	Suppose $i < j$, where $i$ and $j$ are natural numbers, and $\b>0$. Let $J \in \mc{P}_j$ be such that $J \cap F_i(\b) \neq \emptyset$.
        Then, $J \subseteq F_i\left(\b - \frac{\mc{S}(f)}{N} e^{2(i+1-j)N}\right)$.
       \end{lemma}
       
       \begin{proof}
       	Let $s \in J \cap F_i(\b)$.
        Then, $|s - \eta| \leq  e^{-2jN}$ for any $\eta \in J$.
        Hence, since $f$ is Lipschitz, we have that for all $t\in [iN, (i+1)N]$
        \[ 
        	\left|  f(g_t h_\eta \omega) - f(g_t h_s \omega) \right|
             \leq \norm{f}_{Lip} d(h_{e^{2t}(s-\eta)} , id)
             \leq \mc{S}(f) e^{2(t-j)N},
        \]
        where we use $d(g,h)$ to be the metric on $SL_2(\R)$ defined by the maximum absolute value of the entries of the matrix $gh^{-1} - Id$.
        Averaging the above inequality in $t$, we get that 
        \begin{align*}
        	|f_i(\eta) - f_i(s)| \leq \frac{\mc{S}(f)e^{-2jN} }{N}
            		\int_{iN}^{(i+1)N} e^{2t}\;dt \leq \frac{\mc{S}(f)}{N} e^{2(i+1-j)N}, 
        \end{align*}
        which implies the lemma.
       \end{proof}
       
%%%%%%%%%%%%%%%%%%%%%%%%%%%%%%%%%%%

       The following lemma is the main result of this section. Let the notation be the same as in Lemma~\ref{lemma: bound on measure of recurrent part of F_i}.
	   \begin{lemma} [Independence Lemma]
       \label{propn: independence lemma}
       Suppose $\e$ is given.
       Then, for all $\ell >0$ and all $a >0$, there exists $T_0 >0$ such that 
       for all $\omega \in \mc{H}_1(\a)$, $N>T_0$, $\b>\e + \frac{\mc{S}(f)}{N}$ 
       and finite sets $A \subset \N$, we have 
         \[
              \nu\left( \bigcap_{i\in A} \left( F_i(\b) \cap \bigcup_{J\in \mc{R}_i(C_\ell)} J  \right) \right)
              \leq a^{|A|},
         \]
         where $|A|$ is the number of elements in $A$.
	   \end{lemma}
       
       \begin{proof}
        Fix some $\ell$ and $a$.
        Let $T_0>0$ be as in Lemma~\ref{lemma: bound on measure of recurrent part of F_i}, $N>T_0$, and $A \subset \N$ with $p=|A|$. Up to relabeling, we may assume $A = \set{1,\dots,p}$. Finally, let $\omega \in \mc{H}_1(\a)$.

        For any $\b>\e + \frac{\mc{S}(f)}{N}$ and $i\in \N$, we define
        \[  F_i^{\mc{R}}(\b) := F_i(\b) \cap \bigcup_{J\in \mc{R}_i} J, \]
        where we use $\mc{R}_i$ to denote $\mc{R}_i(C_\ell)$.
        
        We proceed by induction on $p$.
        Since elements of $\mc{R}_p$ are disjoint, we have
        \begin{align*}
        	\nu\left( \bigcap_{i\in A} F_i^{\mc{R}}(\b) \right)
            = \nu\left(  \bigcup_{J\in \mc{R}_p} \left(  J \cap  F_p(\b) \cap 
            	\bigcap_{i=1}^{p-1}   F_i^{\mc{R}} (\b)    \right)    \right) = \sum_{J\in \mc{R}_p} \nu\left(  J \cap  F_p(\b) \cap 
            	\bigcap_{i=1}^{p-1}   F_i^{\mc{R}}  (\b)   \right)
        \end{align*}
        
        Moreover,
        \begin{align}
        \label{eqn: filtration property of F_i}
        \bigcap_{i=1}^{p-1}   F_i^{\mc{R}}  (\b)  
            \subseteq \bigcap_{i=1}^{p-1}   F_i^{\mc{R}}  \left(\b- \frac{\mc{S}(f)}{N} e^{2(i+1-p)N} \right).  
        \end{align}
        
        Let $J \in \mc{R}_p$ be such that $J \cap  \bigcap_{i=1}^{p-1}   F_i^{\mc{R}}  (\b) \neq \emptyset$.
        Then, $J \cap  \bigcap_{i=1}^{p-1}   F_i  (\b) \neq \emptyset$ and for any $i=1, \dots, p-1$ there exists $J'\in\mathcal R_i$ such that $J\cap J'\neq \emptyset$.
        Hence, by Lemma~\ref{lemma: 0-1 law for partitions and F_i}, 
        \begin{align*}
        	J \subseteq \bigcap_{i=1}^{p-1}   F_i  \left(\b- \frac{\mc{S}(f)}{N} e^{2(i+1-p)N} \right).
        \end{align*}
        By enlarging $N$ if necessary, we may assume that $\mc{P}_j$ is a refinement of $\mc{P}_i$ for $i\leq j$.
       Hence, we see that
        \begin{equation*}
        J\subseteq  \bigcap_{i=1}^{p-1} \bigcup_{J'\in \mc{R}_i} J'.
        \end{equation*}
        In particular, we obtain the following base step in our inductive procedure.
        \begin{align}
          \label{eqn: 0-1 law application}
          J \subseteq \bigcap_{i=1}^{p-1}   F_i^{\mc{R}}  \left(\b- \frac{\mc{S}(f)}{N} e^{2(i+1-p)N} \right)
        \end{align}
        for all $J \in \mc{P}_p$ satisfying $J \cap  \bigcap_{i=1}^{p-1}   F_i^{\mc{R}}  (\b) \neq \emptyset$.
		Therefore, it follows that
        \begin{align} \label{eqn: base step measure estimate}
        	\nu\left( \bigcap_{i\in A} F_i^{\mc{R}}(\b) \right) &\leq
            	\sum_{ \substack{ J\in \mc{R}_p \\ J \cap   \bigcap_{i=1}^{p-1}   F_i^{\mc{R}}  (\b) \neq \emptyset }  } 
                	\nu( J\cap F_p(\b) )\nonumber\\
                    &\leq a \sum_{ \substack{ J\in \mc{R}_p \\ 
                    	J \cap   \bigcap_{i=1}^{p-1}   F_i^{\mc{R}}  (\b) \neq \emptyset }  } 
                    	\nu(J) 
                        & \text{ by Lemma~\ref{lemma: bound on measure of recurrent part of F_i}}
                        \nonumber\\
                    &\leq a \nu \left(  \bigcap_{i=1}^{p-1}   F_i^{\mc{R}} \left(\b- \frac{\mc{S}(f)}{N} e^{2(i+1-p)N} \right)  \right)
                    & \text{ by~\eqref{eqn: 0-1 law application} }.
        \end{align}
        
        Our choice of $\b$ guarantees that for all $1\leq k\leq p$,
        \begin{equation*}
        	\b - \frac{\mc{S}(f)}{N} \sum_{j=0}^{j=k-1} e^{(i+1-(p-j))N} > \e.
        \end{equation*}
        Note here that our assumption that $A = \set{1,\dots,p} $ maximizes the sum in the above inequality.
        In other words, our choice of $\b$ guarantees that the above inequality holds where the sum is taken over any set of natural numbers $A$ of cardinality $p$.
        
        Hence, by induction on our base measure estimate in~\eqref{eqn: base step measure estimate},
        via repeated application of Lemma~\ref{lemma: 0-1 law for partitions and F_i},
        \begin{align*}
        	\nu\left( \bigcap_{i\in A} F_i^{\mc{R}}(\b) \right) &\leq
            	a \nu \left(  \bigcap_{i=1}^{p-1}   
            		F_i^{\mc{R}} \left(\b- \frac{\mc{S}(f)}{N} e^{2(i+1-p)N} \right)  \right) \\
                &\leq   a^2 \nu \left(  \bigcap_{i=1}^{p-2}   
            					F_i^{\mc{R}} \left(\b- \frac{\mc{S}(f)}{N} e^{2(i+1-p)N}  
                    				- \frac{\mc{S}(f)}{N} e^{2(i+1-(p-1))N}  \right)  \right) \\
				&\leq \dots \\
                &\leq  a^k \nu \left(  \bigcap_{i=1}^{p-k}   
            					F_i^{\mc{R}} \left(\b- \frac{\mc{S}(f)}{N}
                                	\sum_{j=0}^{j=k-1} e^{2(i+1-(p-j))N}
                    				  \right)  \right)  \\
                &\leq a^{p}
        \end{align*}
        as desired.
       \end{proof}

       %        Thus, combining~\eqref{eqn: filtration property of F_i} and~\eqref{eqn: 0-1 law application}, 
 %       we get that for all such $J\in \mc{R}_p$:
  %      \begin{align*}
   %     	J \cap  F_p(\b) \cap \bigcap_{i=1}^{p-1}   F_i^{\mc{R}}  (\b)
    %        	&\subseteq J \cap  F_p(\b) \cap 
     %           \bigcap_{i=1}^{p-1}   F_i^{\mc{R}}  \left(\b- \frac{\mc{S}(f)}{N} e^{2(i+1-p)N} \right)   \\
      %          &\subseteq J \cap  F_p(\b)
       % \end{align*}

    \subsection{A Covering Lemma}
       
       As a consequence of Lemma~\ref{propn: independence lemma}, we obtain the following bound on the number of intervals needed to cover intersections of the recurrent parts of the sets $F_i$.
       More precisely, we obtain the following.
       
       \begin{lemma}
       	\label{lemma: birkhoff covering lemma}
        Given $\e>0$.
       Then, for all $\ell >0$ and $a >0$, there exists $T_0 >0$ such that 
       for all $\omega \in \mc{H}_1(\a)$, $N>T_0$, all $\b>\e + \frac{2\mc{S}(f)}{N}$, $M \in \N$,
       and finite sets $A \subseteq \set{0,\dots, M-1}$, the following holds.
        \[
        	\# \set{ J \in \mc{P}_M: J \cap \bigcap_{i\in A} \left( F_i(\b) \cap \bigcup_{J'\in \mc{R}_i(C_\ell)} J'  \right) 
            		\neq \emptyset } \leq  e^{2MN} a^{|A|},
        \]
         where $|A|$ is the number of elements in $A$.
       \end{lemma}
       
       \begin{proof}
       Fix $\ell$ and $a$.
       Let $T_0>0$ be as in Proposition~\ref{propn: independence lemma}, $N>T_0$, $M \in \N$ and  
        $\b> \e + \frac{2\mc{S}(f)}{N}$.
        Suppose $A \subset \set{0,\dots, M-1}$.
       	For each $i\in \N$, let
        \[  F_i^{\mc{R}}(\b) := F_i(\b) \cap \bigcup_{J\in \mc{R}_i} J \]
        
       	As in the proof of Lemma~\ref{propn: independence lemma},
        a combination of Lemma~\ref{lemma: 0-1 law for partitions and F_i}
        and the fact that the partitions $\mc{P}_i$ form a refining sequence of partitions (which we may assume by enlarging $N$ slightly if necessary)
        shows that for all $J \in \mc{P}_M$,
        \[
        	J \cap \bigcap_{i\in A}  F_i^{\mc{R}} (\b)  \neq \emptyset
            		\Longrightarrow
        	J \subseteq \bigcap_{i\in A} F_i^{\mc{R}} \left(\b - \frac{\mc{S}(f)}{N} e^{2(i+1-M)N}\right). 
        \]
        
        In particular, for any $J\in\mc{P}_M$ satisfying
        $J \cap \bigcap_{i\in A}  F_i^{\mc{R}} (\b)  \neq \emptyset$, one has
        \[
        	J \subseteq \bigcap_{i\in A} F_i^{\mc{R}} \left(\b - \frac{\mc{S}(f)}{N}\right).
        \]
        
        Therefore, by our condition on $\b$ and Lemma~\ref{propn: independence lemma}, we get
        \begin{align}
        \label{eqn: estimate on measure of sum of intervals}
        	\sum_{ \substack{ J\in\mc{P}_M \\ J \cap \bigcap_{i\in A}  F_i^{\mc{R}} (\e/2)  \neq \emptyset } }
            	\nu(J) &\leq
                	\nu \left( \bigcap_{i\in A} F_i^{\mc{R}} \left(\b - \frac{\mc{S}(f)}{N}\right) \right) \leq a^{|A|}.
        \end{align}

        Recall that $\mc{P}_M$ is a partition of $[-1,1]$ into intervals of radius $ e^{-2MN}$.
        In particular, for $J\in \mc{P}_M$, $\nu(J) =  e^{-2MN}$.
        Combined with~\eqref{eqn: estimate on measure of sum of intervals}, this implies the lemma. 
       \end{proof}

    \subsection{Proof of Theorem~\ref{thrm: Birkhoff deviations thrm for horocycles}}

        Let us fix the following parameters so that we can apply Lemmas~\ref{propn: independence lemma} and ~\ref{lemma: birkhoff covering lemma}.
        Fix $\e>0$.
        Let $\d, a>0$ be sufficiently small so that the following holds.
            	\begin{align}\label{Birkhoff: condition on a}
                  	\d\leq \frac{\e}{4\mc{S}(f)}%\label{Birkhoff: condition on delta} 
                    \qquad \text{ and } \qquad 2 < a^{-\d}.
            	\end{align}
          Let $\mc{N}_1,\dots,\mc{N}_k$ be proper affine invariant submanifolds as in Theorem~\ref{thrm: quantitative Chaika-Eskin} applied to $\e/50$ and $f$.
             By ~\cite[Proposition 2.13]{EMM}, for any $i=1,\dots,k$ there exists a height functions $f_{\mc{N}_i}$.
             For $\ell>0$, let
             \begin{align*} 
             	C_\ell = \overline{\set{x\in \mc{H}_1(\a): \sum_1^k f_{\mc{N}_i}(x) \leq \ell}}.
             \end{align*}
             The function $\a = \sum_1^k f_{\mc{N}_i}$ satisfies all the properties in Definition~\ref{defn: height functions} (see ~\cite[Proposition 2.13]{EMM}).
             Suppose $\omega \in \mc{M}\backslash \left(\cup_{i=1}^k \mc{N}_i\right) $.
             Thus, $\a(\omega) < \infty$.             
             In particular, Theorem~\ref{thrm: H.dim of non-recurrent directions,fixed compact set} applies and guarantees the existence of some $\ell = \ell(\d)$ and $t_0>0$ so that for all $t>t_0$, one has
             \begin{equation} \label{eqn: birkhoff divergent set dimension bound}
             	dim_H(Z_\omega(C_{\ell}, t,\d)) \lneq 1.
             \end{equation}
             where the bound is uniform over all $\omega \in \mc{M}\backslash \left(\cup_{i=1}^k \mc{N}_i\right) $.
             
        Let $\ell>0$ be such that~\eqref{eqn: birkhoff divergent set dimension bound} holds.
		Let $T_0>0$ be as in Lemma~\ref{lemma: birkhoff covering lemma} applied to $f$, $\e/50$. Let $N > \max\set{T_0,t_0}$.
        Over the course of the proof, we will enlarge $N$ as necessary, depending only on $\e$, $a$ and $f$.
         
       Fix some $\omega \in \mc{M}\backslash \left(\cup_{i=1}^k \mc{N}_i\right) $.
       Recall the definition of the sets $F_i$ (see \eqref{defn: the sets F_i}), partitions $\mc{P}_i$ and $\mc{R}_i:=\mc{R}_i(C_\ell)$ (see Section~\ref{section: Birkhoff sets partitions}).
       By enlarging $N$ if necessary, we may assume that $\mc{P}_i$ form a refining sequence of partitions.
       For each $i\in \N$ and $\b>0$, define
       \[ F_i^{\mc{R}}(\b) = F_i(\b) \cap \bigcup_{J\in \mc{R}_i} J.  \]
       
       By Lemma~\ref{lemma: bad set for f and recurrent part of F_i}, we get that
       \begin{align}
       		B_\omega (f,N,\e) \subseteq Z_\omega(C_\ell, N,\d) \cup
            	\limsup_{M\r\infty} \bigcup_{ \substack{ A \subseteq \set{0,\dots,M-1} \\
                				|A| = \lceil \d M \rceil } }
                                \bigcap_{i\in A} F_i^{\mc{R}} (\e/2)
       \end{align}
       
    	Thus, by~\eqref{eqn: birkhoff divergent set dimension bound}, it suffices to bound the Hausdorff dimension of
        the second set on the right hand side.
        Let $M\in \N$ and define
        \[
        	\mc{F}_M^{\mc{R}} = \bigcup_{ \substack{ A \subseteq \set{0,\dots,M-1} \\
                				|A| = \lceil \d M \rceil } }
                                \bigcap_{i\in A} F_i^{\mc{R}} (\e/2).
        \]
        
        The number of sets of the form $A$ in the above union is at most $\binom{M}{\lceil \d M \rceil}$.
        Moreover, we may assume $N$ is large enough so that
        \[ \e/2 > \e/50 + \frac{2 \mc{S}(f)}{N} \]
        Hence, we may apply Lemma~\ref{lemma: birkhoff covering lemma} with $\e/50$ in place of $\e$ to get that
        when $N$ is large enough, we have
        \begin{align} \label{eqn: raw bound on cover}
        	\#\set{ J\in \mc{P}_M: J \cap \mc{F}_M^{\mc{R}} \neq \emptyset }
            		&\leqslant \sum_{\substack{ A \subseteq \set{0,\dots,M-1} \\
                				|A| = \lceil \d M \rceil }}
                    \#\set{ J\in \mc{P}_M: J \cap  \bigcap_{i\in A} F_i^{\mc{R}} (\e/2) \neq \emptyset }
                    \nonumber \\
                      &\leqslant \binom{M}{\lceil \d M \rceil} e^{2MN}
                      a^{\d M} \leq  2^M e^{2MN} a^{\d M}
        \end{align} 	
  		Let $\b = \ln(2)/2N$      
        and 
        $ \g = - \frac{1}{2N} \ln\left(a^\d \right) $.
        Then, ~\eqref{eqn: raw bound on cover} can be rewritten in the following way. 
        \begin{align*}
        	\#\set{ J\in \mc{P}_M: J \cap \mc{F}_M^{\mc{R}} \neq \emptyset }
                      &\leq    e^{2(1 +\b-\g   )MN}
        \end{align*}

        By Lemma~\ref{lemma: H.dim of limsup set via covering lemma}, we get that
        the Hausdorff dimension of $\limsup_M \mc{F}_M^{\mc{R}}$ is at most $1 +\b -\g$.
        This bound is strictly less than $1$ if and only if $2< a^{-\d}$, which holds by our choice of $a$ in~\eqref{Birkhoff: condition on a}.
        Finally, we note that our upper bound depends only on $f$ and $\e$ and is uniform in the choice of $\omega$ in $ \mc{M}\backslash \left(\cup_{i=1}^k \mc{N}_i\right) $.
        This completes the proof.

    \subsection{Deviations of Discrete Birkhoff Averages}\label{Birkhoff_discrete}
	The same methods used in this section to prove Theorem~\ref{thrm: Birkhoff deviations thrm for horocycles} also imply the following analogous statement for discrete Birkhoff averages.
	\begin{theorem} \label{thrm: discrete Birkhoff deviations thrm for horocycles}

      Suppose $\mc{M}\subseteq \mc{H}_1(\a)$ is an affine invariant submanifold and $\nu_{\mc{M}}$ is the affine measure whose support is $\mc{M}$.
      Then, for any bounded continuous function $f$ on $\mc{M}$ and any $\e>0$, there exist affine invariant submanifolds $\mc{N}_1,\dots, \mc{N}_k$, properly contained in $\mc{M}$, and $\d\in (0,1)$, such that for all $\omega \in \mc{M}\backslash \left( \cup_{i=1}^k \mc{N}_i \right)$ and all $l>0$,      
      the Hausdorff dimension of the set
      \begin{equation*}
          \set{s\in[-1,1]: \limsup_{N\r\infty} \left| \frac{1}{N} \sum_{n=1}^N f(g_{ln}h_s\omega) -
                   \int_{\mathcal M} f \;d\nu_\mc{M}\right| \geq \e }
      \end{equation*}
      is at most $\d$.
  	\end{theorem}

    We note that by modifying the definition of the functions $f_i$ in~\eqref{defn: functions f_i} to be
    \[ f_i(s) = \frac{1}{N}\sum_{k=iN}^{(i+1)N} f(g_{lk}h_s\omega)  \]
    the rest of the proof of Theorem~\ref{thrm: discrete Birkhoff deviations thrm for horocycles} follows verbatim as in the case of flows and as such we omit it.

\section{Random Walks and Oseledets' Theorem}
\label{section: random walks}

    In this section, we recall some results on the growth of the Kontsevich-Zorich cocycle along random walk trajectories on $\mc{H}_1(\a)$ which were proved in~\cite{ChaikaEskin}.
    Using the fact that a typical random walk trajectory is tracked by a geodesic up to sublinear error, we translate such results to results concerning the Teichm\"{u}ller geodesic flow.
    
  Suppose $(M, \omega) \in \mc{H}_1(\a)$ and $\nu_{\mc{M}}$ is the affine measure whose support is $\mc{M} = \overline{\mathrm{SL}_2(\R) \omega}$.  
    Let $V$ be a continuous $\mathrm{SL}_2(\R)$-invariant subbundle over $\mc{H}_1(\a)$ of (an exterior power of) the Hodge bundle.
    Denote by $A_V:\mathrm{SL}_2(\R) \times \mc{M} \r GL(V)$ the restriction of the Kontsevich-Zorich cocycle to $V$.
    Let $\norm{A_V(\cdot,\cdot)}$ be the Hodge norm on $V$ (see~\cite[Section 3.4]{ForniMatheus}).
    
   Denote by $\l_V$ the top Lyapunov exponent of this cocycle
    under the Teichm\"{u}ller geodesic flow with respect to $\nu_\mc{M}$.
    In particular, by Oseledets' multiplicative ergodic theorem, for $\nu_\mc{M}$ almost every $x\in \mc{M}$,
    \[  \lim_{t\r\infty}  \frac{\log \norm{A_V(g_t,x)} }{t} = \l_V. \]

    The cocycle $A_V$ satisfies the following (Lipschitz) property with respect to the Hodge norm: 
    there exists a constant $K\in \N$ such that
    for all $x\in \mc{M}$ and all $g\in \mathrm{SL}_2(\R)$,
    \begin{equation}
    \label{eqn: Lipschitz property of the cocycle}
    	\norm{A_V(g,x)} \leq \norm{g}^K,
    \end{equation}
    where for $g\in \mathrm{SL}_2(\R)$, we use $\norm{g}$ to denote the norm of $g$ in its standard action on $\R^2$.
    This follows from~\cite[Lemma 2.1']{Forni} (see also~\cite[Corollary 30]{ForniMatheus}).
    We note that the power $K$ appears since we are considering the action of an exterior power of the cocycle. Moreover, Forni's variational formula for the derivative of the cocycle along geodesics implies~\eqref{eqn: Lipschitz property of the cocycle} for general elements of $\mathrm{SL}_2(\R)$ by the $KAK$ decomposition, the cocycle property and the fact that $\norm{A(r_\th,\cdot)} = 1$ for all $\th$.
    
    Since $A_V(id,x) = id$ for all $x$, we see that $A_V(g,x)^{-1} = A_V(g^{-1},gx)$ for all $g\in \mathrm{SL}_2(\R)$ and $x\in\mathcal M$.
    Hence, by~\eqref{eqn: Lipschitz property of the cocycle}, we get
    \begin{equation}
    \label{eqn: Lipschitz property of the inverse of the cocycle}
    	\norm{A_V(g,x)^{-1}} \leq \norm{g^{-1}}^K
    \end{equation}
    
    We shall need the following facts about matrix norms which follow from the $KAK$ decomposition and the bi-invariance of $\norm{\cdot}$ under $K$.
    
    \begin{lemma}
    \label{lemma: matrix norm and distance to identity}
    	There exist constants $C_1 > 0$ such that for all $g\in \mathrm{SL}_2(\R)$,
        \begin{enumerate}
        \item $\log \norm{g} \leq C_1 d(g,id)$.
        \item $\norm{g^{-1}} = \norm{g}$
        \end{enumerate}
        where $d$ denotes the right invariant metric on $\mathrm{SL}_2(\R)$ and $id$ is the identity element.
    \end{lemma}
    
    \subsection{Random Walks}
    
    In the remainder of this section and the next section, we fix a compactly supported probability measure $\mu$ on $\mathrm{SL}_2(\R)$ which is $SO(2)$ bi-invariant and absolutely continuous with respect to the Haar measure.
    Let $\mathrm{SL}_2(\R)^{\mathbb{N}}$ be the space of infinite sequences of elements in $\mathrm{SL}_2(\R)$ equipped with the probability measure $\mu^{\mathbb{N}}$.
    For each $n$ define the random variable $\omega_n: \mathrm{SL}_2(\R)^{\mathbb{N}}\to \mathrm{SL}_2(\R)$ as
    \[(g_1,g_2,\ldots,g_n, \ldots)\mapsto\omega_n=g_n g_{n-1}\cdots  g_2g_1
    \] 
    For any fixed base point $x\in\mc{H}_1(\a)$, the orbit $\{\omega_nx\}_{n\in\mathbb{N}}$ in $\mc{H}_1(\a)$ is called a \emph{random walk} on $\mc{H}_1(\a)$. 
    
     A measure $\nu$ on $\mc{H}_1(\a)$ is called \emph{$\mu$-stationary} if $\mu\ast\nu=\nu$ where 
     \[
     \mu*\nu=\int_{\mathrm{SL}_2(\R)}(g_{\ast}\nu)d\mu(g).
     \]
     The measure $\nu_\mc{M}$ is an ergodic $\mu$-stationary measure i.e. it cannot be written as a non-trivial convex combination of other $\mu$-stationary measures.
     By a variant of Oseldets' theorem, due to~\cite{GoldsheidMargulis} in the setting of random walks,
     there exists $\l_V^\mu \in \R$ such that for $\nu_\mc{M}$-almost every $x$ and for $\mu^\N$ almost every $(g_1,g_2,\dots) \in \mathrm{SL}_2(\R)^\N$,
     \begin{equation*}
     	\lim_{n\r\infty} \frac{\log \norm{A_V(g_ng_{n-1}\cdots g_1,x)}}{n} = \l_V^\mu
     \end{equation*}
    The following sets were introduced in~\cite{ChaikaEskin} as a way to quantify uniformity in the above limit.
    
    \subsection*{The Sets $E_{good}(\e,L)$}
    Let $\e>0$ and $L\in \N$. Denote by $\mathrm{E_{good}}(\e,L)$ the set of points $y\in \mc{M}$ such that for all $v\in V$, there exists a set $H(v) \subseteq \mathrm{SL}_2(\R)^L$ such that
    \begin{enumerate}
    	\item $\mu^L(H(v) > 1-\e$,
     	\item For all $(g_1,\dots,g_L) \in H(v)$,
        	\[ \l_V^\mu - \e < \frac{\log \norm{A_V(g_L\cdots g_1, y)v} }{L}
            \leq \frac{\log \norm{A_V(g_L\cdots g_1, y)} }{L} < \l_V^\mu +\e  \]
    \end{enumerate}
    
    The following lemma is an important part of our proof as it is a key step in the proof of the Oseledets part of~\cite{ChaikaEskin}.
    \begin{lemma}[Lemma 2.11 in~\cite{ChaikaEskin}]
    \label{lemma: CE lemma measure of E good}
    	For any fixed $\e>0$, the sets $E_{good}(\e,L)$ are open and
        \[
        	\lim_{L\r\infty} \nu_{M} \left( E_{good}(\e,L)  \right) = 1
        \]
    \end{lemma}   
    %%%%%%%%%%%%%%%%%%%%%%%%%%%%%%%

    \subsection{From Random Walks to Flows}

Since we will be concerned with metric properties of the exceptional set, it will be important for us to translate random walk results into the language of Teichm\"{u}ller geodesics.
    It is a classical fact that random walk trajectories induced by a stationary measure on $\mathrm{SL}_2(\R)$ tracks (up to sublinear error) a Teichm\"{u}ller geodesic.
    This is made precise in the following:
    
    \begin{lemma} [Lemma 4.1 in~\cite{ChaikaEskin}]
    \label{lemma: sublinear tracking}
    	There exists $\l>0$, depending only on $\mu$, such that 
        there exists a measurable map $\Theta: \mathrm{SL}_2(\R)^\N \r [-\pi/2,\pi/2]$
        , defined $\mu^{\N}$-almost everywhere, so that for $\mu^\N$-a.e. 
        $\overline{g} = (g_1,g_2,\dots) \in \mathrm{SL}_2(\R)^\N$,
        \begin{equation} \label{eqn: sublinear tracking limit}
        	\lim_{n\r\infty} \frac{\log || g_{\l n}r_{\Theta(\overline{g})} 
            			(g_n\cdots g_1)^{-1}||  }{n} =0.
        \end{equation}
       
        Furthermore, $\Theta_\ast \mu^\N $ coincides with the normalized Lebesgue measure. 
        In particular, for any interval $[a,b] \subseteq [-\pi/2,\pi/2]$,
        \begin{align}
        	\label{eqn: sublinear tracking angle distribution}
            \mu^{\N} \left( \overline{g}: \Theta(\overline{g}) \in [a,b]  \right)
             = \frac{b-a}{\pi}.
        \end{align}
    \end{lemma}
    
    \begin{remark} \label{remark: random walk exponent and geodesic flow exponent}
      The relationship between the Lyapunov exponent of the random walk $\l_V^\mu$ and the Lyapunov exponent of the Kontsevich-Zorich cocycle under the Teichm\"{u}ller flow $\l_V$ is provided 
      by the parameter $\l$ in Lemma~\ref{lemma: sublinear tracking}
      as follows.
      \begin{align*}
          \l_V = \frac{\l_V^\mu}{\l}
      \end{align*}
    \end{remark}
    
    %%%%%%%%%%%%%%%%%%%%%%%%%%%%%%
    
    %%%%%%%%%%%%%%%%%%%%%%%%%%%%%%%%%%%%%%%%%%%%%%%%%%%%%%%%%%
    The following Lemma uses Lemma~\ref{lemma: sublinear tracking} to show that geodesic trajectories which start within the sets $E_{good}(\e,L)$ also exhibit good properties with respect to the cocycle.
    
    For simplicity, throughout this section we use the notation $A:= A_V$
    \begin{lemma}
    \label{lemma: E good with angles}
    	There exists a constant $C>0$, depending only on the constants of the cocycle such that the following holds:
    	for every $\e>0$, there exists $L_0>0$
        such that for all $L\in \N$ with $L\geq L_0$,
        for all $y\in E_{good}(\e,L)$ and all $v\in V$,
        there exists $\tilde{H}(v) \subseteq [-\pi/2,\pi/2]$ such that
        for all $\th \in \tilde{H}(v)$,
        \[
        	\l_V - C\e < \frac{\log ||A(g_{\l L}, r_\th y)|| }{\l L} < \l_V +C\e \]
        and such that $\nu(\tilde{H}(v)) > 1-3\e$,
        where $\nu$ is the normalized Lebesgue measure on $[-\pi/2,\pi/2]$.
    \end{lemma}
    
    \begin{proof}

        Let $\l$ is as in Lemma~\ref{lemma: sublinear tracking}.
        Using Egorov's theorem, we can find a set $\mc{U}\subseteq SL_2(\R)^\N$ with $\mu^\N(\mc{U}) > 1-\e$ so that the convergence in~\eqref{eqn: sublinear tracking limit} is uniform over $\mc{U}$.
        In particular, we can choose $L\in \N$ sufficiently large 
        so that for all $\overline{g}\in \mc{U}$:
        \begin{equation} \label{eqn: geodesic close to random walk}
        	 \frac{\log|| g_{\l L}r_{\Theta(\overline{g})} 
            			(g_L\cdots g_1)^{-1}|| }{L}  < \e
        \end{equation}

    	Fix $y\in E_{good}(\e,L)$ and $v\in V$.
        Let $H(v) \subseteq SL_2(\R)^L$ be as in the definition of $E_{good}(\e,L)$.
        We will regard $H(v)$ as a cylinder subset of $SL_2(\R)^\N$ in the natural way.  
        The set $\tilde{H}(v)$ will be essentially the image of $H(v)\cap \mc{U}$ under $\Theta$, except that $\Theta$ is only a measurable map.
        To go around this, we use Lusin's theorem to find a compact set 
        $\mc{K}\subset SL_2(\R)^\N$, such that $\mu^{\N}(\mc{K}) > 1-\e$ 
        and such that the restriction of $\Theta$ to $\mc{K}$ is continuous.
        Let
        \[  \tilde{H}(v) = \Theta\left( H(v) \cap \mc{U}\cap\mc{K} \right) \]
        
        Since $\Theta$ is continuous on $\mc{K}$ and $\tilde{H}(v)$ 
        is a Borel subset of $\mc{K}$, we see that $\tilde{H}(v)$ is Lebesgue measurable.
        Moreover, by Lemma~\ref{lemma: sublinear tracking}, one has
        \begin{equation*}
        	\nu(\tilde{H}(v)) = \mu^\N \left(  \Theta^{-1}(\tilde{H}(v)) \right)
            	\geq \mu^\N\left( H(v) \cap \mc{U}\cap \mc{K}\right) 
                > 1-3\e
        \end{equation*}
        
        To see that $\tilde{H}(v)$ satisfies the conclusion of the Lemma,
        let $\overline{g} \in H(v) \cap \mc{U} \cap \mc{K}$.
        For all $L$ sufficiently large so that~\eqref{eqn: geodesic close to random walk} holds for all $\overline{g} \in \mc{U}$,
        define $\e_L \in SL_2(\R)$ by the following equation
        \begin{equation*}
        	 g_{\l L}r_{\Theta(\overline{g})} =
            			\e_L g_L\cdots g_1
        \end{equation*}
    	with $\e_L \in SL_2(\R)$.
        Then, using the cocycle property, we get
        \begin{equation*}
            	A(g_{\l L},r_{\Theta(\overline{g})} \omega) =
                A(\e_L, g_L\cdots g_1 \omega)
                A(g_L\cdots g_1, \omega)
        \end{equation*}
        Hence, since $\overline{g}\in H(v)$, by definition of the set $E_{good}(\e,L)$ and by~\eqref{eqn: Lipschitz property of the cocycle}, we get
        \begin{align*}
        	\frac{\log ||A(g_{\l L},r_{\Theta(\overline{g})} \omega)|| }{L}
            &\leq \frac{\log || A(\e_L, g_L\cdots g_1 \omega) ||}{L}
            + \frac{\log ||A(g_L\cdots g_1, \omega) || }{L} \\
            &\leq \frac{K \log ||\e_L|| }{L} + \l_V^\mu + \e \\
            &\leq \l_V^\mu + \e(1+K)
        \end{align*}
        Similarly, using~\eqref{eqn: Lipschitz property of the cocycle}
        and $(2)$ of Lemma~\ref{lemma: matrix norm and distance to identity}, we get
        \begin{align*}
        	\frac{\log ||A(g_{\l L},r_{\Theta(\overline{g})} \omega)|| }{L}
            &\geq \frac{\log ||A(g_L\cdots g_1, \omega) || }{L}
              -\frac{\log || A(\e_L, g_L\cdots g_1 \omega)^{-1} ||}{L}\\
            &\geq \l_V^\mu - \e - \frac{K \log ||\e_L^{-1}|| }{L} = \l_V^\mu - \e - \frac{K  \log ||\e_L|| }{L}   \\
            &\geq \l_V^\mu - \e(1+K)
        \end{align*}
        
        Dividing both estimates by $\l$ and 
        noting that by remark~\ref{remark: random walk exponent and geodesic flow exponent}, $\l_V = \l_V^\mu/\l$, we get the desired conclusion.
    \end{proof}

As a corollary, we obtain the following statement for horocycles.
\begin{corollary}
    \label{cor: E good with horocycles}
    	There exists a constant $C_2>0$, depending only on the constants of the cocycle such that the following holds:
    	for every $\e>0$, there exists $L_0>0$
        such that for all $L\in \N$ with $L\geq L_0$,
        for all $y\in E_{good}(\e,L)$ and all $v\in V$,
        there exists $G(v) \subseteq [-2,2]$ such that
        for all $s \in G(v)$,
        \begin{equation} \label{eqn: E good with horocycles norm estimate}
        \l_V - C_2\e < 
            \frac{\log ||A(g_{\l L}, h_s y)|| }{\l L} 
            < \l_V +C_2\e 
        \end{equation}
        	
        and such that $|G(v)| \geq 4(1-30\e)$,
        where $|\cdot|$ is the Lebesgue measure on $[-2,2]$.
    \end{corollary}
    
    \begin{proof}
    	Fix $\e>0$. Suppose $L\in \N$ is sufficiently large 
        so that Lemma~\ref{lemma: E good with angles} holds, $y\in E_{good}(\e,L)$ and $v\in V$.
        Let $\tilde{H}(v)\subseteq [-\pi/2,\pi/2]$ and $C>0$ be as in the conclusion of Lemma~\ref{lemma: E good with angles}.
        Consider
        \begin{align*}
        G(v) =\tan\left(\tilde{H}(v)  \right) \cap [-2,2].
        \end{align*}
        We verify that the corollary holds for this set.
        Let
        \begin{equation*}
        	\rho = \tan^{-1}(2).
        \end{equation*}

        For every $\th \in \tilde{H}(v) \cap [-\rho,\rho]$ we write 
        $r_\th =\check{h}_{-\tan\th} g_{\log\cos\th}h_{\tan \th}$.
        Then, using the cocycle property, we see the following.
        \begin{equation*}
        	A(g_{\l L},r_\th y) =
            	A(\check{h}_{-e^{-2\l L}\tan\th} g_{\log\cos\th},
                g_{\l L} h_{\tan\th}y) A(g_{\l L},h_{\tan\th}y)
                A(\check{h}_{-\tan\th} g_{\log\cos\th}, h_{\tan\th}y)^{-1}
        \end{equation*}
        
        Therefore, using the Lipschitz 
        property~\eqref{eqn: Lipschitz property of the cocycle} 
        and~\eqref{eqn: Lipschitz property of the inverse of the cocycle}
        and the fact that $\th\in [-\rho,\rho]$,
        we get
        \begin{equation*}
         \frac{\log|| A(g_{\l L},r_\th y)||}{\l L} = 
         	\frac{\log||A(g_{\l L},h_{\tan\th}y)||}{\l L} + 
            O\left( \frac{1}{L} \right)
        \end{equation*}

        Thus, for $L$ large enough, and for all $s\in [-2,2]$ of the form $s=\tan \th$ with $\th\in \tilde{H}(v) \cap [-\rho,\rho]$, we can find a constant $C'>0$, independent of $L$, so that~\eqref{eqn: E good with horocycles norm estimate} holds. 
        
        Let $|\cdot|$ be the Lebesgue measure. Since $| \tilde{H}(v)| \geq \pi (1-3\e)$, we get that
        \[
        	\frac{| \tilde{H}(v) \cap [-\rho,\rho] |}{2\rho}
            \geq 1-\frac{3\e\pi}{2\rho}.
        \]
        Hence, since the Jacobian of the map $\th\mapsto \tan\th$ 
        is bounded by $\sec^2(\rho)=5$ on $[-\rho,\rho]$, the following holds.
        \[
        	\frac{\left| G(v)  \right|}{4} \geq 1-\frac{15\e\pi}{4} 
            	\geq 1- 30\e
        \]
    \end{proof}

    %%%%%%%%%%%%%%%%%%%%%%%%%%%%%%%%%%%%%%%%%%%%%%%%%%%
    \iffalse
    The next Lemma will allow us to use a more flexible set that the set $E_{good}(\e,L)$, which is more suitable for working with flows.
    It is an immediate consequence of the
    
    \begin{lemma}
    \label{lemma: intersections of E_good}
    	Given $\e,\d>0$, there exists $L_0 \in \N$, such that 
        for all $L \geq L_0$ and all $k\in \N$,
        \begin{equation*}
        	\mu_\omega\left(  \bigcap_{i=1}^k E_{good}(\e,L+i) \right)
             > 1-\d
        \end{equation*}
    \end{lemma}
    \fi

\section{Large Deviations in Oseledets' Theorem}
	\label{section: oseledets}
    
    The purpose of this section is to prove Theorem~\ref{thrm: Oseledets deviations thrm for horocycles} concerning the Hausdorff dimension of directions whose geodesics exhibit deviation of the top Laypunov exponent for the Kontsevich-Zorich cocycle.
    The structure of the proof is very similar to that of Theorem~\ref{thrm: Birkhoff deviations thrm for horocycles}.
    The idea is to relate the directions exhibiting deviation in Oseledets
    theorem along a Teichm\"{u}ler geodesic to the directions exhibiting deviation in
    Birkhoff's theorem for the indicator function of a large open set with good properties with respect to the cocycle.
    The proof is written in such a way so as to mirror the proof of Theorem~\ref{thrm: Birkhoff deviations thrm for horocycles} on deviations in Birkhoff's theorem.

    Throughout this section we retain the notation from the previous section and also use the following.
    For any positive $\e,L\in\mathbb R$ and $M\in \N$, we define the following sets.
   \begin{align*}
    	&B(A, L,\e,M) := \set{s\in[-1,1]:  \frac{\log \norm{A(g_{LM},h_s \omega)}}{LM} > \l_V +\e }\\
    	&B(A,L,\e) := \limsup_{M\r\infty} B(A, L,\e,M)
   \end{align*}
    Using the cocycle property, it is easy to check that for any $L>0$ 
    \[ B(A,L,\e) = \set{s\in[-1,1]: \limsup_{t\r\infty}  
    	\frac{\log \norm{A(g_t,h_s\omega)} }{t} \geq \l_V + \e } \]
    
 Moreover, for any $s\in[-1,1]$, $\b, L>0$ and $i\in\N$, we define the corresponding functions and sets. 
    \begin{equation*}
    	a_i(s) =  \frac{\log\norm{ A( g_{L}, g_{L i} h_s \omega)}}{L}
    \end{equation*}
    
    \begin{equation*}
    	A_i(\b) = \set{\th: a_i(s) > \l_V +\b }
    \end{equation*}
    The functions $a_i$ and sets $A_i$ play the role of the functions $f_i$ (see \eqref{defn: functions f_i}) and the sets $F_i$ (see \eqref{defn: the sets F_i}), respectively, in the proof of large deviations in Birkhoff's theorem.

    %%%%%%%%%%%%%%%%%%%%%%%%%%%%%%%%%%%

 \subsection{Sets and Partitions} \label{section: oseledets sets partitions}

 For $L>0$ and $i\in \N$, let $\mc{P}_i$ denote the partition of $[-1,1]$ into intervals of radius $e^{-2iL}$,
       By enlarging $L$ if necessary, we may assume $e^{L} \in \N$ and that
       $\mc{P}_{i+1}$ is a refinement of $\mc{P}_i$ for all $i$.      
       For $\e>0$, define the following sub-partitions
       \[ \mc{E}_i(\e,L) = \set{ J\in \mc{P}_i: \exists s\in J,
       		g_{iL} h_s \omega \in E_{good}(\e,L) }.  \]
       Here $\mc{E}$ signifies recurrence to the set $E_{good}$.

   The following Lemma is an analogue of Lemma~\ref{lemma: bad set for f and recurrent part of F_i}.
       
       \begin{lemma}
       	\label{lemma: bad set for A and recurrent part of A_i}
        	Let $\e_1,\e_2,L>0$, $M\in \N$, and $0< \d \leq \e/4K$, where $K\in\N$ is the exponent in
            in~\eqref{eqn: Lipschitz property of the cocycle}.
            Then,
            \[
            	B(A,L,\e_1,M) \subseteq Z_\omega(E_{good}(\e_2,L), M, L ,\d) \cup 
                	\bigcup_{ \substack{ B \subseteq \set{1,\dots,M} \\
                		|B| = \lceil \d M \rceil } }
                        \bigcap_{i\in B} \left( A_i(\e_1/2) \cap 
                        \bigcup_{J\in \mc{E}_i(\e_2,L)} J \right),
            \]
            where $Z_\omega(E_{good}(\e_2,L), M, L ,\d) $ is defined in~\eqref{defn: set of non-recurrent directions}.
       \end{lemma}
    
\begin{proof}
First, we notice that for any $\e>0$
 \[  
            	B(A,L,\e,M) \subseteq \bigcup_{ \substack{ B \subseteq \set{1,\dots,M} \\
                				|B| > 2\d M  } }
                                \bigcap_{i\in B} A_i(\e/2)
            \]
 Using the cocycle property and submultiplicativety of matrix norms, we have the following inequalities
        \begin{equation*}
        	\frac{\log \norm{A(g_{LM},h_s\omega) }}{ LM}
            	\leq\frac{1}{M} \sum_{i=1}^M 
                \frac{\log \norm{A(g_{L}, g_{L i} h_s\omega) }}{ L}
        \end{equation*}
        From this point on, using~\eqref{eqn: Lipschitz property of the cocycle} to bound $\norm{A(g_{L},\cdot)}$,
        the proof is identical to that of Lemma~\ref{lemma: bad set for f and recurrent part of F_i}.

\end{proof}

 \subsection{Measure Bounds for $A_i$}
	The goal of this section is to obtain a uniform bound on the measure of sets of the form $A_i\cap J$ for any $J\in \mc{E}_i$ and any $i$.
    This step is analogous to Lemma~\ref{lemma: bound on measure of recurrent part of F_i}.

	%%%%%%%%%%%%%%%%%%%%%%%%%%%%
    The following is the main result of this section. The key input in the proof is Lemma~\ref{lemma: E good with angles}.
	\begin{lemma}
       \label{lemma: bound on measure of recurrent part of A_i}
       	Let $C_2>0$ be as in Corollary~\ref{cor: E good with horocycles}.
       	Then,
        for every $\e>0$, there exists $L_1 >0$ such that for all $L\geq L_1$, all $\g \geq 2C_2\e$, all $i\in \N$ and all $J \in \mc{E}_i(\e,L)$,
        \[ \frac{\nu (J\cap A_i(\g) )  }{\nu(J)} \leq 120 \e,  \]
        where $\nu$ is the Lebesgue probability measure on $[-1,1]$.
    \end{lemma}
    
    \begin{proof}
    	Let $L_0 >0$ and $\l>0$ be as in Corollary~\ref{cor: E good with horocycles} and Lemma~\ref{lemma: sublinear tracking}, respectively.
        Define $L_1: = L_0/\l$.
    	Suppose $\gamma\in\mathbb R$ and $L\in \N$ are such that $\g \geq 2 C_2 \e$ and $L \geq L_1$.
        
    	Let $i\in \N$, $J\in \mc{E}_i:=\mc{E}_i(\e,L)$, and $s_0 \in J$ be such that $y_0:=g_{iL} h_{s_0} \omega \in E_{good}(\e,L)$.
        Let $v\in V$ and $G(v)\subseteq [-2,2]$ be as in Corollary~\ref{cor: E good with horocycles}.
        Choose $\eta \in J-s_0$ such that $s_0 + \eta \in A_i(\g)\cap J$.
        Then, we have
        \begin{align*}
        	\l_V + \g \leq  a_i(s_0 +\eta)= 
            \frac{\log\norm{ A( g_{L }, h_{e^{2iL  }\eta} y_0)} }{L}
        \end{align*}
        Hence, by definition of $G(v)$,
        \begin{align}
       		\label{eqn: eta not in G(v)}
       		e^{2iL  }\eta \notin G(v)
        \end{align}
        Note that $ e^{2iL  } (J-s_0) $ is a subinterval of $[-2,2]$ of length 2.
        In particular, we get that
        \begin{equation*}
        	e^{2iL  } \left((A_i(\g)\cap J)-s_0\right) \subseteq 
            	[-2,2]\backslash G(v)
        \end{equation*}
        
        Thus, since the Lebesgue measure of $G(v)$ is at least $4(1-30\e)$,
        we get the following measure estimate
        \begin{align*}
        	\left|e^{2iL  } \left((A_i(\g)\cap J)-s_0\right)\right| 
            = \frac{\nu (J\cap A_i(\g) )  }{\nu(J)}
            \leq 120 \e
        \end{align*}
        This concludes the proof in the case $L\in \N$.
        For the $L \geq L_1$ with $L \notin\N$, write $L = \lfloor L \rfloor + \set{L} $ where 
        $\lfloor L \rfloor$ is the largest natural number less than $L$ and $\set{L} = L - \lfloor L \rfloor$. Then, using the cocycle property, submultiplicativety of the norm and the Lipschitz property of the cocycle~\eqref{eqn: Lipschitz property of the cocycle}, we get
        \begin{align*}
        	\frac{\log \norm{A(g_L, \cdot )}}{L} \leq \frac{\log \norm{A(g_{\lfloor L \rfloor},\cdot)}}{L}
            + O\left( \frac{1}{L} \right)
        \end{align*}
        Thus, we see that the conclusion follows in this case from the case when $L\in \N$ by choosing $L_1$ sufficiently large depending on $\e$.
    \end{proof}

 \subsection{Independence of the Sets $A_i$}

	As a consequence of the Lipschitz property of the 
    cocycle~\eqref{eqn: Lipschitz property of the cocycle},
    we are able to prove an analogue of Lemma~\ref{lemma: 0-1 law for partitions and F_i}.

	\begin{lemma}
       \label{lemma: 0-1 law for partitions and A_i}
       There exists a constant $C_3>0$, depending only on the constants of the cocycle $A$ so that the following holds.
       	Suppose $i<j$, where $i$ and $j$ are natural numbers, $L>0$, and $\g>0$. Let $J \in \mc{P}_j$ be such that $J \cap A_i(\g) \neq \emptyset$. Then,
        $J \subseteq A_i\left(\g - C_3\frac{ e^{2(i+1-j)L}}{  L}\right). $
	\end{lemma}
    
    \begin{proof}
   		Let $s_0 \in J \cap A_i(\g)$. Then, by definition of the partition $\mc{P}_j$ in Section~\ref{section: oseledets sets partitions}, $|s_0 - \eta| \leq e^{-2jL}$ for any $\eta\in J$.
        
        Using the cocycle property, we have the following.
        \begin{align*}
        	A(g_{L}, g_{iL } h_\eta \omega) &= 
            	A(g_{L}, h_{e^{2iL}(\eta-s_0)}  g_{iL}h_{s_0}\omega)
                \nonumber \\
            	&= 	A(g_{L} h_{e^{2iL}(\eta-s_0)},  g_{iL}h_{s_0}\omega)
                	A( h_{e^{2iL}(\eta-s_0)},  g_{iL}h_{s_0}\omega)^{-1}
                    \nonumber \\
                &= A(h_{e^{2(i+1)L}(\eta-s_0)},  g_{(i+1)L}h_{s_0}\omega)
                	A(g_{L}, g_{iL}h_{s_0}\omega)
                    A( h_{e^{2iL}(\eta-s_0)},  g_{iL}h_{s_0}\omega)^{-1}
        \end{align*}
 Therefore,
 \begin{equation}\label{eqn: difference between cocycle at s_0 and eta}
 A(g_{L}, g_{iL}h_{s_0}\omega) = A(h_{e^{2(i+1)L}(\eta-s_0)},  g_{(i+1)L}h_{s_0}\omega)^{-1}A(g_{L}, g_{iL } h_\eta \omega)A( h_{e^{2iL}(\eta-s_0)},  g_{iL}h_{s_0}\omega).
 \end{equation}
        Hence, by~\eqref{eqn: Lipschitz property of the cocycle},~\eqref{eqn: Lipschitz property of the inverse of the cocycle}
        and Lemma~\ref{lemma: matrix norm and distance to identity},
        there exists a constant $C_1$ so that
        \begin{align*}
           a_i( s_0 ) - a_i( \eta )
                 &\leq
                 \frac{K \log||h_{e^{2(i+1)L}(\eta-s_0)}^{-1} ||}{L} +
                 \frac{K \log||h_{e^{2iL}(\eta-s_0)} ||}{L}\\
                 &\leq \frac{K C_1d(h_{-e^{2(i+1)L}(\eta-s_0)},id) }{L}
                 + \frac{K C_1d(h_{e^{2iL}(\eta-s_0)},id) }{L}\\
                 &\leq 2 K C_1 \frac{e^{2(i+1-j)L}}{  L}
        \end{align*}
        which concludes the proof.

    \end{proof}

	As a consequence, we obtain exponential decay in the measure of intersections of the sets $A_i$, similarly to Lemma~\ref{propn: independence lemma}.
%%%%%%%%%%%%%%%%%
    \begin{lemma} [Independence Lemma for $A_i$]
       \label{propn: independence lemma oseledets}
       Let $C_3>0$ be as in Lemma~\ref{lemma: 0-1 law for partitions and A_i} and $C_2>0$ be as in Corollary~\ref{cor: E good with horocycles}.
         Then, for all $\e>0$, there exists $L_1 >0$ such that
         for all $L \geq L_1$, all finite sets $B \subset \N$ and all $\g > 2C_2\e + \frac{C_3}{L} $,
         \[
              \nu\left( \bigcap_{i\in B} \left( A_i(\g) \cap \bigcup_{J\in \mc{E}_i(\e,L)} J  \right) \right)
              \leq ( 120\e)^{|B|}
         \]
	   \end{lemma}
       
       \begin{proof}
       	The proof is identical to the proof of Lemma~\ref{propn: independence lemma} which is a formal consequence of two results: Lemma~\ref{lemma: bound on measure of recurrent part of F_i} that gives an upper bound on the measure of $F_i$, and Lemma~\ref{lemma: 0-1 law for partitions and F_i}.
        The analogues of those two results are Lemma~\ref{lemma: bound on measure of recurrent part of A_i}
        and Lemma~\ref{lemma: 0-1 law for partitions and A_i}, respectively.
       \end{proof}

 \subsection{A Covering Lemma}
       
       The following lemma shows existence of efficient covers for intersections of the sets $A_i$, similarly to Lemma~\ref{lemma: birkhoff covering lemma}.
       
       \begin{lemma}
       	\label{lemma: bound on covers for A_i}
        Let $C_3>0$ be as in Lemma~\ref{lemma: 0-1 law for partitions and A_i} and $C_2>0$ be as in Corollary~\ref{cor: E good with horocycles}.
         Then, for all $\e>0$, there exists $L_1 >0$ such that
         for all $L \geq L_1$, $M\in \N$, sets $B \subseteq \set{1,\dots, M}$ and  $\g > 2C_2\e + \frac{2C_3}{L} $, we obtain the following.

        \[
        	\# \set{ J \in \mc{P}_{M+1}: J \cap \bigcap_{i\in B} \left( A_i(\g) \cap \bigcup_{J'\in \mc{E}_i(\e,L)} J'  \right) 
            		\neq \emptyset } \leq e^{2 L (M+1)}   (120\e)^{|B|},
        \]
        where $|B|$ is the number of element in $B$.
       \end{lemma}
       
       \begin{proof}
       		The proof is a direct consequence of Proposition~\ref{propn: independence lemma oseledets} 
            and proceeds as in the proof of Lemma~\ref{lemma: birkhoff covering lemma}.
       \end{proof}
 
 \subsection{Proof of Theorem~\ref{thrm: Oseledets deviations thrm for horocycles}}
	Fix $\e>0$.
    Suppose $\e' >0$ is a sufficiently small number (depending only on $\e$).
    Define $\d := \e/4K$, where $K$ is the exponent in~\eqref{eqn: Lipschitz property of the cocycle}.
	By Lemma~\ref{lemma: CE lemma measure of E good}, choose $L >0$ large enough, depending on $\e'$, so that
    \begin{equation} \label{eqn: measure estimate for E good epsilon'}
    	\nu_\mc{M}(E_{good}(\e',L)) > 1 - \d/2.
    \end{equation}

     Let $\chi_E$ denote the indicator function of the open set $E_{good}(\e',L)$.
     Then, using a variant Urysohn's lemma, we can find a Lipschitz compactly supported continuous function
     $f\colon\mc{M}~\r~[0,1]$, satisfying $f \leq \chi_E$ and
     \begin{equation} \label{eqn: choice of approximating function}
     	\nu_{\mc{M}}(f) > 1- 3\d/4
     \end{equation}
     
	 Moreover, we have that for all $M \in \N$ and all $\w \in \mc{M}$,
     \begin{equation*}
     	Z_\omega(E_{good}(\e',L), M, L ,\d) \subseteq B_\w\left(1-f, L, \d - \nu_\mc{M}(1-f), M\right)
     \end{equation*}
     where these sets are defined in~\eqref{defn: set of non-recurrent directions} and ~\eqref{defn: B(f,N,epsilon,M)} (for discrete Birkhoff averages).
     
     Note that $\d - \nu_\mc{M}(1-f) > 0$ by~\eqref{eqn: choice of approximating function}.
     Thus, by Theorem~\ref{thrm: discrete Birkhoff deviations thrm for horocycles},
     there exist $0<\eta<1$ and finitely many proper affine invariant manifolds $\mc{N}_1,\dots,\mc{N}_k \subset \mc{M}$, depending on $f$ and $\e$, so that the following holds
     \begin{equation} \label{eqn: Hdim of Z depends on L}
     dim_H \left( \limsup_M Z_\omega(E_{good}(\e',L), M, L ,\d) \right) \leqslant \eta \lneq 1
     \end{equation}
     uniformly for all $\w \in \mc{M}\backslash \left(\cup_{i=1}^k \mc{N}_i\right)$.
     These are the affine manifolds appearing in the conclusion of Theorem~\ref{thrm: Oseledets deviations thrm for horocycles}.
     Now, fix one such $\w$.

     Recall the definition of the sets $A_i$ and partitions $\mc{P}_i$ in Section~\ref{section: oseledets sets partitions}.
     By enlarging $L$ if necessary, we may assume that $\mc{P}_i$ form a refining sequence of partitions.
     For $i\in \N$ and $\g>0$, define
     \[
     	\mc{A}_i^\mc{E}(\g) := A_i(\g) \cap \bigcup_{J\in \mc{E}_i(\e',L)} J
     \]
     Then, by Lemma~\ref{lemma: bad set for A and recurrent part of A_i}, since $\d= \e/4K$, we get
     \begin{equation*}
     \limsup_M B(A,L,\e,M) \subseteq \limsup_M Z_\omega(E_{good}(\e',L), M, L ,\d) \cup 
     				\limsup_M
                	\bigcup_{ \substack{ B \subseteq \set{1,\dots,M} \\
                		|B| = \lceil \d M \rceil } }
                        \bigcap_{i\in B} \mc{A}_i^\mc{E}(\e/2)
     \end{equation*}
     Thus, it remains to control the Hausdorff dimension of the second set on the right side.
     We apply Lemma~\ref{lemma: bound on covers for A_i} to $\e'$ in place of $\e$ and $\g=\e/2$.
     By choosing $\e'$ to be sufficiently small and $L$ sufficiently large, we can insure that
     \[
     	\e/2 > 2 C_2 \e' + \frac{2C_3}{L}
     \]
     where $C_2$ and $C_3$ are constants depending only on the cocycle as in the statement of Lemma~\ref{lemma: bound on covers for A_i}.
     
     As a result, choosing $L$ sufficiently large, we can apply Lemma~\ref{lemma: bound on covers for A_i} and proceed as in the proof of Theorem~\ref{thrm: Birkhoff deviations thrm for horocycles} to get that
     \begin{equation*}
     	dim_H \left( \limsup_M \bigcup_{ \substack{ B \subseteq \set{1,\dots,M} \\
                |B| = \lceil \d M \rceil } } \bigcap_{i\in B} \mc{A}_i^\mc{E}(\e/2) \right) \leq 
                1+ \frac{\ln(2)+ \d\ln(120\e') }{2L}
     \end{equation*}
     By choosing $\e' < 2^{-1/\d}/120$ (thus depending only on $\e$), we get that this upper bound is strictly less than one. Moreover, observe that the parameters $\d,\e',L$ appearing in the upper bound above are independent of $\w$.
     This completes the proof.

\section{Weak Mixing IETs}

\label{section: IETs}

	This section is dedicated to the proof of Corollary~\ref{cor: IETs}. We first recall some definitions and the results of~\cite{BoshernitzanNogueira} which connect weak mixing properties of IETs with the recurrence of Teichm\"{u}ller geodesics in an appropriate stratum.

Throughout this section, we fix a natural number $d \geq 2$.
Given a permutation $\pi$ on $d$ letters and $\l= (\l_1,\dots,\l_d)\in \R_+^d$, we define $|\lambda| = \lambda_1+\lambda_2+\dots+\lambda_d$ and an interval exchange transformation (IET) with permutation $\pi$ to be the piecewise linear map $T_{\l,\pi}:[0,|\lambda|) \r [0,|\lambda|)$ defined as follows:
	first we partition the interval $[0,|\lambda|]$ into $d$ ordered half open subintervals $I_i$ so that the length of $I_i$ is equal to $\l_i$ and $T_{\l,\pi}$ maps $I_i$ linearly onto $I_{\pi(i)}$ for each $1\leq i\leq d$.    
    More formally, for $1\leq i\leq d$ and $x\in I_i$,
	\begin{equation*}
		T_{\l,\pi}(x) = x + \sum_{\pi(j) < \pi(i)} \l_j - \sum_{j<i} \l_j.
	\end{equation*}
    An IET $T_{\l,\pi}$ has finitely many points of possible discontinuity
    \[ \b_0 = 0, \qquad \b_i := \sum_{j\leq i} \l_j, \leq i\leq d-1 \]
    Define (See~\cite{Veech-bilinearform} and~\cite[Section 2.2]{MinskyWeiss}) an alternating bilinear form on $\R^d \times \R^d$ by its value on the standard basis elements $\mbf{e}_i$ as follows
    \begin{equation} \label{eqn: bilinear form}
    Q(\mbf{e}_i,\mbf{e}_j) = \begin{cases}
    1  &i > j, \pi(i) < \pi(j) \\
    -1 &i < j, \pi(i) > \pi(j) \\
    0  &\text{otherwise}.
    \end{cases}
    \end{equation}
    for $1\leq i,j\leq d$.
    Then, for each $\l\in \R_+^d$, $1\leq i\leq d$ and $x\in [\b_{i-1},\b_i)$, we have
    \[ T_{\l,\pi}(x) - x = Q(\l,\mbf{e}_i) \]

	The cone $\R_+^d$ can be viewed as the space of IETs with a given permutation $\pi$ with a natural euclidean metric and Lebesgue measure.
    IETs preserve the Lebesgue measure on the unit interval and we shall refer to ergodic properties (ergodicity, weak mixing, etc) of IETs with respect to it.

\subsection{A criterion for weak mixing}

A permutation $\pi$ on $d$ letters $\set{1,\dots,d}$ is \textit{irreducible} if for every $1\leq j <d$,
	\[ \pi(\set{1,\dots,j}) \neq \set{1,\dots,j}. \]
    
\begin{definition}
Suppose $\pi$ is an irreducible permutation on $d$ letters. Define inductively a finite sequence $\{a_p\}_{p=0,1,\dots,l}$ of natural numbers as follows. 

Set $a_0 = 1$. If $a_{p-1}\in\set{\pi^{-1}(1), d + 1}$, then set $l=p-1$ and stop. Otherwise, define $a_p = \pi^{-1}\left(\pi(a_{p-1})- 1\right) + 1$. The permutation $\pi$ is of \textbf{type} $\mathbf{W}$ if $a_l=\pi^{-1}(1)$. 
\end{definition}

    Following~\cite{BoshernitzanNogueira}, we say that $T_{\l,\pi}$ satisfies \textbf{IDOC} (the \textit{infinite distinct orbit condition}) if each discontinuity point $\b_i$  has an infinite orbit under $T_{\l,\pi}$ and for $i\neq j$, the orbits of $\b_i$ and $\b_j$ are disjoint.
    
    Using the orbits of the points $\b_i$ under the IET $T_{\l,\pi}$, we define a sequence of partitions of $[0,1]$ as follows: for each $n\geq 1$, $P_n$ denotes the partition into subintervals whose endpoints are the successive elements of the sets
    \[D_n = \bigcup_{0\leq k\leq n-1} T_{\l,\pi}^{-k}
    \left( \set{\b_0,\dots,\b_{d-1}}\right) \]
    For each $n$, we define $\epsilon_n(T_{\l,\pi})$ to be the length of the shortest interval in the partition $P_n$.
    The following criterion of weak mixing was proved in~\cite{BoshernitzanNogueira}.
    \begin{theorem} [Theorem 5.3 in~\cite{BoshernitzanNogueira}]
    \label{thrm: weak mixing criterion}
    Suppose $\pi$ is a type $W$ permutation and $T_{\l,\pi}$ is an ergodic IET satisfying IDOC for some $\l \in \R_+^d$. If $\limsup\limits_{n\r\infty} n\epsilon_n \left( T_{\l,\pi} \right) >0$, then $T_{\l,\pi}$ is weak mixing.
    \end{theorem}
    Motivated by this criterion, we will say that an IET $T_{\l,\pi}$ has \textbf{short intervals} if
    \begin{equation} \label{eqn: short intervals}
    \lim_{n\r\infty} n\epsilon_n \left( T_{\l,\pi} \right) =0
    \end{equation}

    %%%%%%%%%%%%%%%%%%%%%%%%%%%%%%%%%%%%%%%%%
    \subsection{A Compactness criterion for strata}
    Suppose $\mc{H}$ is a stratum of abelian differentials.
    We recall here a description of standard compact subsets of $\mc{H}$.
    Given $\w\in \mc{H}$, denote by $\mc{L}_\w$ the set of all of its saddle connections, i.e., the set of all flat geodesic segments joining a pair of the singularities of $\w$.
    Then, we can naturally regard $\mc{L}_\w$ as a subset of vectors in $\C$.
    Note that $\mc{L}_\w$ is a discrete set.
    Moreover, using the standard action of $\mathrm{SL}(2,\R)$ on $\C$, for any $g\in \mathrm{SL}(2,\R)$, the set $\mc{L}_{g\w}$ can be identified with $g\cdot \mc{L}_\w$.
    
    For $v\in \C$, let $\norm{v} = \max\set{|\mathrm{Re}(v)|, |\mathrm{Im}(v)|} $.
    Now, define a function $\ell: \mc{H} \r \R_+$ by
    \begin{equation} \label{eqn: mini height}
    	\ell(\w) := \min\set{ \norm{v}: v\in \mc{L}_\w }
    \end{equation}
    For any $\e>0$, we use the following notation
    \begin{equation} \label{def: K_epsilon}
    	\mc{K}_\e := \set{\w\in \mc{H}: \ell(\w) \geqslant \e  }
    \end{equation}
    It is known that the sets $\mc{K}_\e$ with $\e>0$ are compact subsets of $\mc{H}$ and that any bounded subset of $\mc{H}$ is contained in $\mc{K}_\e$ for some $\e$.

    %%%%%%%%%%%%%%%%%%%%%%%%%%%%%%%%%%%%%%%%%
    \subsection{Short intervals and recurrence}
    Given an abelian differential $\omega \in \mc{H}$ on a surface $S$, there is a well-defined vector field given by the imaginary part $\mathrm{Im}(\omega)$.
    This vector field is defined at all points in $S$ except for the (finitely many) zeros of $\omega$.
    This vector field defines a singular flow on $S$, called the \textit{vertical flow}, by moving points at linear speed in the direction $\mathrm{Im}(\omega)$.
    
    By fixing a straight line segment $I$ (a geodesic segment in the flat metric defined by $\omega$) which is transversal to the vertical flow lines, the first return map $T:I \r I$ of the flow defines an IET.
    One can pick $I$ parallel to the real part of $\omega$ so that the resulting IET has $2g+k-1$ intervals, where $g$ is the genus of $S$ and $k$ is the number of zeros of $\omega$.
    This remains true if the angle between $I$ and the real part is sufficiently small.
    
    This process allows us to define a map from a neighborhood of $\omega$ in the stratum $\mc{H}$ to the space of IETs $\R_+^{2g+k-1}$ as follows.
    Pick a segment $I$ parallel to the real part of $\omega$ as above and let $\pi$ be the permutation associated to the IET on $2g+k-1$ intervals defined by the first return map of the vertical flow defined by $\omega$ to $I$.
    Then, we can find a sufficiently small open neighborhood $\mc{U}_\omega$ of $\omega$ in $\mc{H}$ so that for all $x\in \mc{U}_\omega$, the first return map of the vertical flow defined by $x$ to the segment $I$ is an IET with $2g+k-1$ intervals and with the same permutation $\pi$.
    
    This defines a Lipschitz map 
    \begin{equation} \label{eqn: transversal map}
    \mc{T}: \mc{U}_\omega \r \R_+^{2g+k-1}
    \end{equation}
    in the Teichm\"{u}ller and Euclidean metrics respectively.
    Conversely, using Veech's zippered rectangles construction, one can find suspension of any IET using a piecewise linear roof function to obtain an abelian differential on a compact surface. However, this construction is not unique and in general the pre-image of an IET $T_{\l,\pi}$ under the map $\mc{T}$ is a positive dimensional subset of $\mc{U}_\omega$.
    
    The following proposition allows us to relate the criterion in Theorem~\ref{thrm: weak mixing criterion} to the recurrence of Teichm\"{u}ller geodesics in strata.
    A similar result was obtained in~\cite[Proposition 7.2]{MinskyWeiss} using a slightly different proof.
    We include a proof here for completeness.
    
    \begin{proposition} \label{propn: short intervals and recurrence}
    In the notation above, let $\l \in \R_+^{2g+k-1}$ be in the image of the map $\mc{T}$ in~\eqref{eqn: transversal map}. Suppose that $T_{\l,\pi}$ has short intervals.
    Then, for all $\tilde \omega\in \mc{T}^{-1}(\l)$, the geodesic $g_t\tilde \omega$ diverges in $\mc{H}_1(\a)$.
    \end{proposition}

    \begin{proof}[Proof of Proposition~\ref{propn: short intervals and recurrence}]
    Suppose $\l$ is as in the statement and let $\tilde\omega\in \mc{T}^{-1}(\l)$.
    Fix some $\e >0$ and let $n_0 \geq 1$ be such that $n \epsilon_n \left( T_{\l,\pi} \right) < \e$ for all $n \geq n_0$.
    We construct a sequence of saddle connections $v_n$ in $\tilde\omega$ so that the length of $g_{\log (n/\sqrt{\e})} v_n$ is $\ll \sqrt{\e}$ for all $n\geq n_0$. Since $\e$ is arbitrary, $g_t\tilde\omega$ diverges in $\mc{H}_1(\a)$.
    
    For this we use an argument similar to the one found in~\cite[Section 10]{BoshernitzanDuke}.
    Let $P_1,\dots, P_k$ be a collection of polygons in the plane representing $\tilde\omega$ and let $\tilde{I}$ be a lift of the transversal $I$ under the covering map $\cup P_i \r S$ which glues parallel sides by translations. We recall that $S$ is the surface of genus $g\geqslant 1$ that supports the abelian differentials in the proposition.

    For each $n \geq n_0$, denote by $I_n \subset I$ be a subinterval such that 
    \[ |I_n| = \epsilon_n \left( T_{\l,\pi} \right) < \epsilon/n \]
    We use $\tilde{I}_n$ to denote a lift of $I_n$ inside $\tilde{I}$.
    Denote by $\mc{C}$ the open cylinder consisting of the union of the vertical flow orbits of the points in the interior of $I_n$ up to the $n^{th}$ time these orbits hit the transversal $I$.
    By definition of the endpoints of the interval $I_n$, the cylinder $\mc{C}$ contains no zeros of $\tilde\omega$.
    
    Let $\tilde{\mc{C}}$ denote a lift of $\mc{C}$ to the complex plane which we unfold to a parallelogram in the following manner. Let $x$ be an arbitrary point in the interior of $\tilde{I}_n$ and denote by $x_t:= x + it $ for $t>0$. Define $t(x,n)$ to be the time $t>0$ corresponding to the $n^{th}$ return of $x$ to $I$ under the vertical flow.
    
    Next, we define a finite sequence of times $q_i \in (0, t(x,n))$ and polygons $L_i$ with $1\leq i \leq n$ by induction as follows.
    Let $L_1\in\{P_1,\dots,P_k\}$ denote the polygon containing $x$. Define
    \[  q_1 = \inf \set{ 0<t<t(x,n): x_t \text{ meets a side of } L_1 }\]
    As the endpoints of $I_n$ are discontinuities of the first return IET, the set on the right-hand side is necessarily non-empty.
    Let $l_1$ denote the side of $L_1$ such that $x_{q_1} \in l_1$.
    Let $r_1$ denote the unique side of a polygon $R_1\in\{P_1,\dots,P_k\}$ which is identified to $l_1$ by a translation $T_1$ (which defines the gluing of parallel sides).

    Once $(q_j,L_j,l_j,r_j,T_j, R_j)$ have been defined for all $1\leq j\leq i-1 < n$, we define
    \[  q_i = \inf\set{ q_{i-1} < t < t(x,n): x_t \text{ meets a side of } T_{i-1}\cdot R_{i-1} } \]
    Let $L_i = T_{i-1}\cdot R_{i-1}$, let $l_i$ denote the side of $L_i$ such that $x_{q_i} \in l_i$.
    
    Note that $l_i$ is the image of a side $l'_i$ of a polygon in $ \set{P_1,\dots,P_k} $ by a translation $A$, i.e., A brings $l_i$ back to a side $l'_i$ of one of the original polygons $\{P_1,\dots,P_k\}$.
    Denote by $r_i$ the unique side of a polygon $R_i \in \set{P_1,\dots,P_k} $ which is identified to $l'_i$ by a translation $B$.
    Define the $i^{th}$ translation $T_i$ by $T_i =A\circ B$.

    Now, consider the parallelogram
    \[ \mc{P}_n = \set{ x_t\in \C:  x\in \mathrm{Int}(\tilde{I}_n), 0\leq t\leq t(x,n)  } \]
    where $\mathrm{Int}(\tilde{I}_n)$ denotes the interior of $\tilde{I}_n$.

    By definition of the endpoints of $I_n$, each of the two vertical sides of $\mc{P}_n$ necessarily meets a vertex of one of the polygons $L_1,\dots, L_n$.
    On the other hand, the interior of $\mc{P}_n$ is free from the vertices of the polygons.
    In particular, if we let $v_n$ denote a straight line segment joining two of the vertices on the two vertical sides of $\mc{P}_n$, we see that $v_n$ represents a saddle connection for $x$ which is contained entirely in $\mc{P}_n$.
    
    If we regard $v_n$ as a vector in $\C$, we see that the imaginary part $|\mathrm{Im}(v_n)|$ is at most the height of the parallelogram $\mc{P}_n$.
    Thus, in particular, we get that
    \begin{equation} \label{eqn: imaginary part}
    	|\mathrm{Im}(v_n)| \asymp n
    \end{equation}
    where the implied constant depends only on the lengths of the sides of the polygons $P_1, \dots, P_k$.
    Moreover, the real part $|\mathrm{Re}(v_n)|$ satisfies
    \begin{equation}\label{eqn: real part}
    	|\mathrm{Re}(v_n)| \ll |I_n| \leq \e/n
    \end{equation}
    where the implied constant here depends on the angle between the segment $\tilde{I}$ and the horizontal axis, which, in turn, depends only on the neighborhood $\mc{U}_\omega$.
    Therefore, we see that the length of the saddle connection $g_{\log(n/\sqrt{\e})} v_n$ is $\ll \sqrt{\e}$ as desired.
    \end{proof}

    %%%%%%%%%%%%%%%%%%%%%%%%%%%%%%%%%%%%%%%%%%%%%%%%%%
    \subsection{Horocycles and Lines in the Space of IETs}

    It was shown by Minsky and Weiss in~\cite{MinskyWeiss} that the image of short horocycle arcs under the map~\eqref{eqn: transversal map} is short line segments in $\R_+^d$.
    This result was used in the work of Athreya and Chaika in~\cite{AthreyaChaika} to relate the dimension of divergent directions for the Teichm\"uller flow to the dimension of non-uniquely ergodic IETs.
    We use a similar idea to obtain the following proposition.
    \begin{proposition} \label{propn: dim of short intervals}
    Suppose $\pi$ is an irreducible permutation on $d$ letters. Then, the set of $\l\in  \R_+^d$ corresponding to uniquely ergodic IETs $T_{\l,\pi}$  which are IDOC and satisfy~\eqref{eqn: short intervals} has Hausdorff codimension at least $1/2$.
    \end{proposition}
    
    The proof of Proposition~\ref{propn: dim of short intervals} will be given in Section~\ref{section: proof of dim of short intervals} after some technical preparation.
    We begin by recalling a result in [MW14] characterizing line segments that can arise as the image of a short horocycle segment under the map~\eqref{eqn: transversal map}.
    \begin{proposition} [Theorem 5.3 in~\cite{MinskyWeiss}]
    \label{propn: Minsky Weiss}
    Suppose $\l \in \R_+^d$ is such that $T_{\l,\pi}$ is uniquely ergodic and satisfies IDOC.
    Suppose $\mbf{b}\in \R^d$ satisfies $Q(\l,\mbf{b})>0$.
    Then, there exists an $\e>0$ and an open neighborhood $\mc{O}$ of $(\l,\mbf{b})$ in $\R_+^d\times \R^d$ and an affine homeomorphism $\mbf{q}: \mc{O} \r \mc{H}$ such that $h_s \mbf{q}(\l,\mbf{b} ) = \mbf{q}(\l+s\mbf{b},\mbf{b})$ for $|s|<\e$.
    Moreover, $\mc{T}(\mbf{q}(\l,\mbf{b})) = \l$ for all $(\l,\mbf{b})\in \R_+^d\times\R^d$.
    \end{proposition}
    
    We remark that Theorem 5.3 in~\cite{MinskyWeiss} is not stated in the form we use here, however the statement of Proposition~\ref{propn: Minsky Weiss} follows easily from the original statement, Definition 5.1 of \textit{positive pairs} and Proposition 5.2 in~\cite{MinskyWeiss}.
    
    The next lemma shows that the positivity condition $Q(\cdot,\cdot)>0$ in Proposition~\ref{propn: Minsky Weiss} is not restrictive.
    \begin{lemma} \label{lemma: Q +ve a.e.}
    For Lebesgue almost every $(\l,\mbf{b}) \in \R_+^d\times \R^d$, $Q(\l,\mbf{b}) \neq 0$.
    \end{lemma}
    
    \begin{proof}
    Suppose $\l =(\l_1,\dots,\l_d) \in \R_+^d$. We claim that $Q(\l,\mbf{e}_1) >0$ and, in particular, non-zero.
    Note that $Q(\mbf{e}_1,\mbf{e}_1) = 0$ and for all $j >1$, $Q(\mbf{e}_j,\mbf{e}_1) \geqslant 0$.
    As $\l\in \R_+^d$, this implies that $Q(\l,\mbf{e}_1) \geqslant 0$.
    
    Now, since $\pi$ is irreducible, we have $\pi(1) > 1$.
    Hence, we can find some $j_0>1$ so that $\pi(j_0)=1 < \pi(1)$.
    It follows that $Q(\l,\mbf{e}_1) \geqslant \l_{j_0} Q(\mbf{e}_{j_0},\mbf{e}_1) = \l_{j_0}>0$.
    
    This shows that the linear form $Q(\l,\cdot)$ is not identically zero.
    That is the kernel of $Q(\l,\cdot) $ has dimension $d-1$ and thus has measure $0$.
    Hence, the lemma follows by Fubini's theorem.   
    
    \end{proof}
    
    %%%%%%%%%%%%%%%%%%%%%%%%%%%%%%%%%%%%%%%%%%%%%%%%%%%%%%%%%%%%%%%%%%%%%%%%%%

    \subsection{Hausdorff dimension, slicing, and proof of Proposition~\ref{propn: dim of short intervals}}
    \label{section: proof of dim of short intervals}
    Denote by $\mrm{Gr}(d,m)$ the Grassmanian of $m$ dimensional subspaces in $\R^d$ and let $\g_{d,m}$ denote a Lebesgue class measure on $\mrm{Gr}(d,m)$.
    The space of lines ($1$ dimensional affine subspaces) in $\R^d$ can be naturally identified with $\mrm{Gr}(d,d-1)\times \R^{d-1}$ and thus carries a Lebesgue class measure.
    The following fact about slicing Borel sets of small Hausdorff codimension with lines will be useful for us.
    
    \begin{proposition} [Theorem 10.8 and Corollary 8.9(3) in~\cite{Mattila}]
    \label{propn: dim of slices}
    Suppose $A \subset \R^d$ is a Borel set with $\mrm{dim}_H(A) > t>d-1$.
    Then, there exists a set $B\subseteq  \mrm{Gr}(d,d-1)\times \R^{d-1}$ of lines in $\R^d$ of positive Lebesgue measure such that for each line $\ell\in B$,
    \[ \mrm{dim}_H( \ell \cap A) \geqslant t-d+1 \]
    \end{proposition}
    We also recall Frostman's lemma.
 \begin{lemma}(Theorem 8.8 in \cite{Mattila})\label{lem: Frostman}
 Suppose $A$ is a Borel subset of $\mathbb R^d$. Let $s>0$. Then, the following are equivalent:
 \begin{enumerate}
 \item\label{Frostman_positive} $H^s(A)>0$, where $H^s$ denotes the $s$-dimensional Hausdorff measure.
 \item\label{Frostman_exists_measure} There exists a Borel measure $\mu$ on $\mathbb R^d$ with support in $A$ such that $\mu(B(x,r))\leq r^s$ for all $x\in R^d$ and $r>0$, where $B(x,r)$ is the closed ball with center $x$ and radius $r$.
 \end{enumerate}
 \end{lemma}
 The idea of the proof of Proposition~\ref{propn: dim of short intervals} is the following.
    First, we use Proposition~\ref{propn: dim of slices} to relate the dimension of the set of interest to the dimension of its intersection with line segments.
    Then, Proposition~\ref{propn: Minsky Weiss} allows us to relate the dimension of sets on line segments to the dimension of subsets of horocycle arcs.
    Finally, using Proposition~\ref{propn: short intervals and recurrence}, we show that the sets of interest on the horocycle arcs correspond to points with divergent $g_t$ orbits.
    As a result, Theorem~\ref{thrm: divergent on average for horocycles} concludes the argument.    
    
    The suggested outline of the proof is a modified version of an argument given in~\cite[Section 6]{AthreyaChaika}. 
    The main difference is the use of Lemma~\ref{lemma: Q +ve a.e.} to bypass the use of Rauzy induction (Lemma 6.5 in~\cite{AthreyaChaika}) which we believe makes the approach more direct.
    
    \begin{proof}[Proof of Proposition~\ref{propn: dim of short intervals}]
    Denote by $A$ the set of $\l\in \R_+^d$ such that $T_{\l,\pi}$ is uniquely ergodic, IDOC, and has short intervals and note that $A$ is Borel measurable.
    Suppose that for some $0<c<1$, we have
    \begin{equation*}
    	\mrm{codim}_H\left( A \right) \leqslant c
    \end{equation*}
    Then, by Proposition~\ref{propn: dim of slices}, there exists a positive measure set $\mc{L}$ of lines in $\R^d$ such that for each line $\ell\in\mc{L}$ a set $\ell\cap A$ has Hausdorff dimension at least $1-c$.
    By Lemma~\ref{lemma: Q +ve a.e.}, we may assume that for each line $\ell \in \mc{L}$  there exists some point $\l \in \ell$ so that $Q(\l,\mbf{b}) \neq 0$, where $\mbf{b}$ is a vector in $\R^d$ parallel to $\ell$.
    Let $\ell  \in \mc{L}$ be a line such that it passes through a point $\lambda\in\R_+^d$ and is parallel to $\mbf{b}\in \R^d$, i.e., $\ell = \set{\l+s\mbf{b}:s\in \R  }$, and $Q(\l,\mbf{b}) \neq 0$.

    By Lemma~\ref{lem: Frostman} (\eqref{Frostman_positive}$\Rightarrow$\eqref{Frostman_exists_measure}), there exists a measure $\mu$ supported on $\ell \cap A$ so that for all $x\in \ell$ and all $r>0$, we have
    \begin{equation*}
    	\mu(B(x,r)) \leqslant r^{1-c}
    \end{equation*}
    Note that the linearity of $Q$ implies that $Q(\l+s\mbf{b},\mbf{b})\neq 0$ for all $s\neq -Q(\l,\mbf{b})/Q(\mbf{b},\mbf{b})$ and for all $s\in \R$ if $Q(\mbf{b},\mbf{b})=0$.
    Hence, since $\mu$ is not a Dirac mass, we can find $x\in \mrm{supp}\, \mu \subset \ell\cap A$ such that $Q(x,\mbf{b})\neq 0$. In particular, $T_{x, \pi}$ is uniquely ergodic, IDOC, and has short intervals. Notice that a priori $\lambda\in \ell$ may not belong to $\mrm{supp}\, \mu$.
    By replacing $\mbf{b}$ with $-\mbf{b}$ if necessary, we may assume $Q(x,\mbf{b})>0$.
    
    Hence, by Proposition~\ref{propn: Minsky Weiss}, we can find $\e_0 >0$ and a local Lipschitz inverse $\mbf{q}$ of the map $\mc{T}$ so that 
    $\mbf{q}(x+s\mbf{b},\mbf{b}) = h_s(\mbf{q}(x,\mbf{b}))$ for $|s|<\e_0$.
    But, by Proposition~\ref{propn: short intervals and recurrence}, the forward $g_t$ orbit of the set $\set{\mbf{q}(x+s\mbf{b},\mbf{b})\colon |s|<\e_0,\, x+s\mbf{b}\in A}$ is divergent (on average) in the stratum $\mc{H}$.
    
    By restricting $\mu$ to the segment $x+s\mbf{b}$ with $|s|<\e_0$ and using Lemma~\ref{lem: Frostman} (\eqref{Frostman_exists_measure}$\Rightarrow$\eqref{Frostman_positive}), we see that the Hausdorff dimension of $A\cap \{x+s\mbf{b}\colon |s|<\e_0\}$ is at least $1-c$.
    Theorem~\ref{thrm: divergent on average for horocycles}, thus, implies that $1-c \leqslant 1/2$.
    
    \end{proof}
    
    % $A$ is Borel because: 1) IDOC is the complement of countably many codimension $1$ rational affine subspaces, 2) having short intervals is defined as the vanishing of a pointwise limit of Borel (piecewise continuous functions) $\epsilon_n$, 3) uniquely ergodic is Borel because ....

    %%%%%%%%%%%%%%%%%%%%%%%%%%%%%%%%%%%%%%%%%
    \subsection{Proof of Corollary~\ref{cor: IETs}}
    Let $\pi$ be a type $W$ permutation on $d\geq 4$ letters.
    By Theorem~\ref{thrm: weak mixing criterion}, we have the following inclusion
    \begin{align} \label{eqn: superset of non weak mixing}
    \set{\l\in \R_+^d: T_{\l,\pi} \text{ not weak mixing} } &\subseteq
    \set{ \l: T_{\l,\pi} \text{ is not IDOC} } \cup\set{\l: T_{\l,\pi} \text{ is NUE } }
    \\&\cup  \set{\l : T_{\l,\pi} \text{ is IDOC, and UE, and has short intervals }  }\nonumber
    \end{align}
    where (N)UE denotes (non)-uniquely ergodic and having short intervals means $T_{\l,\pi}$ satisfies~\eqref{eqn: short intervals}.
    The last two sets in the above union have codimension at least $1/2$ by \cite[Theorem 1.6]{AthreyaChaika} and Proposition~\ref{propn: dim of short intervals}, respectively. 
    
    It is shown in~\cite{Keane} that if the components of $\l$ are linearly independent over $\Q$, then $T_{\l,\pi}$ is IDOC.
    In particular, the set of IETs which are not IDOC is contained in the intersection of the simplex $\R_+^d$ with countably many codimension $1$ subspaces of $\R^d$ which are defined over $\Q$.
    This implies that the set of non-IDOC IETs has Hausdorff codimension at least  $1$.

\section{Large Deviations in Birkhoff's Theorem in Strata - An Outline}

\label{section: remarks}

The scheme suggested in this paper is quite flexible and can be applied to get similar results about the Hausdorff dimension in various settings. For example, using our approach for the proof of Theorem~\ref{thrm: Birkhoff deviations thrm}, one should be able to answer the following question affirmatively.

\begin{question} \label{bstratum}
 Suppose $\mc{M}\subseteq \mc{H}_1(\a)$ is an affine invariant submanifold and $\nu_{\mc{M}}$ is the affine measure whose support is $\mc{M}$. Let $f$ be a bounded continuous function $f$ on $\mc{M}$ and $\e>0$.
      Is the Hausdorff dimension of the set
      \begin{equation*}
          \set{x\in \mc{M}: \limsup_{T\r\infty} \left|
          \frac{1}{T} \int_0^T f(g_tx)\;dt -
          \int_{\mathcal M} f \;d\nu_\mc{M} \right| \geqslant \e }
      \end{equation*}
      strictly less than the dimension of $\mc{M}$?
\end{question}

    Note that, by a standard approximation argument, the affirmative answer to the above question implies that for any non-empty open subset $U$ of a connected component $\mc{C}$ of the stratum $\mc{H}_1(\a)$, the Hausdorff dimension of the set
    \begin{equation*}
          \set{x\in \mc{C}: g_t x \notin U \text{ for all } t>0 }
    \end{equation*}
    is strictly less than the dimension of $\mc{C}$.
    
    For clarity, we briefly outline how to apply our techniques to answer Question \ref{bstratum}. The idea is to translate all the results on horocycle arcs obtained in Section~\ref{section: birkhoff} to results on open bounded subsets of the strong unstable manifold for the Teichm\"{u}ller flow. Then, one obtains the desired result from the analogue of Theorem~\ref{thrm: Birkhoff deviations thrm} for the Hausdorff dimension of the following set in the strong unstable leaf $\mathcal{W}^{su}(\omega)$ of $\omega\in\mathcal H_1(\alpha)$.
    \begin{equation*}
          \set{x\in \mc{W}^{su}(\omega): \limsup_{T\r\infty} \left| \frac{1}{T} \int_0^T f(g_t x)\;dt -
                   \int_{\mathcal M} f \;d\nu_\mc{M}\right| \geqslant \e }.
      \end{equation*}
   where $\w$ satisfies $\overline{\mrm{SL}(2,\R)\w} = \mc{M}$.
   We recall that
    \begin{equation*}
    	\mc{W}^{su}(\omega) = \set{x\in \mc{H}_1(\a):  d_{\mc{H}_1(\a)}\left( g_t \omega, g_t x \right) \r 0 \text{ as } t \r -\infty }
    \end{equation*}
    where $d_{\mc{H}_1(\a)}$ denotes the Teichm\"{u}ller metric.
    Any such leaf is foliated with orbits of the horocycle flow $h_s$.
    In particular, we can locally find foliation charts for $h_s$-orbits within a leaf of the unstable foliation, which also provide immersed local transversals for the horocycle orbits. As a result, we obtain that a neighborhood $\mathcal W^{su}_{loc}(\omega)$ of a point $\omega$ inside $\mathcal W^{su}(\omega)$ has a product structure. This allows the disintegration of the probability measure of Lebesgue class on small bounded open sets of unstable leaves with parameter measures as conditionals along horocycles.
    
    One can then introduce subsets of $\mathcal W^{su}_{loc}(\omega)$ and the corresponding functions analogous to \eqref{defn: B(f,N,epsilon,M)},\eqref{defn: functions f_i} and \eqref{defn: the sets F_i}. Using Fubini's theorem, one should be able to translate the measure bounds on exceptional subsets of horocycle arcs (see Section ~\ref{subsection: measure bounds}) into bounds for exceptional subsets of $\mathcal W^{su}_{loc}(\omega)$. The product structure of $\mathcal W^{su}_{loc}(\omega)$ can be similarly used for translating the results for horocycles in Section~\ref{section: doa} into results for $\mathcal W^{su}_{loc}(\omega)$.

\bibliography{bibliography}{}
\bibliographystyle{alpha}

\end{document}